\documentclass[a4paper,12pt]{amsart}
\usepackage{amsfonts,amsmath,amssymb,euscript,inputenc}
\usepackage{graphicx,psfrag}

\title[Wave equation on hyperbolic spaces]
{The wave equation on hyperbolic spaces}
 
\author{Jean--Philippe Anker}
 
\address{Universit\'e d'Orl\'eans \& CNRS,
F\'ed\'eration Denis Poisson (FR 2964) \& Laboratoire MAPMO (UMR 6628),
B\^atiment de math\'ematiques -- Route de Chartres,
B.P. 6759 -- 45067 Orl\'eans cedex 2 -- France}
 
\email{anker@univ-orleans.fr}

\author{Vittoria Pierfelice}

\address{Universit\'e d'Orl\'eans \& CNRS,
F\'ed\'eration Denis Poisson (FR 2964) \& Laboratoire MAPMO (UMR 6628),
B\^atiment de math\'ematiques -- Route de Chartres,
B.P. 6759 -- 45067 Orl\'eans cedex 2 -- France}
 
\email{vittoria.pierfelice@univ-orleans.fr}

\author{Maria Vallarino}

\address{Politecnico di Torino,
Dipartimento di Matematica --
Corso Duca degli Abruzzi 24 --
10129 Torino --
Italia}

\email{maria.vallarino@polito.it}

\thanks{This work was mostly carried out
while the third author was a CNRS postdoc
at the \textit{F\'ed\'eration Denis Poisson\/} Orl\'eans--Tours.
It was further supported by the Progetto GNAMPA 2010
\textit{Analisi armonica su gruppi di Lie e variet\`a Riemanniane\/}
and by the PHC Galil\'ee 25970QB \textit{VAMP}}

\date{\today}



\newtheorem{lemma}{Lemma}[section]
\newtheorem{theorem}[lemma]{Theorem}
\newtheorem{proposition}[lemma]{Proposition}
\newtheorem{corollary}[lemma]{Corollary}
\newtheorem{remark}[lemma]{Remark}
\newtheorem{definition}[lemma]{Definition}

\newcommand\1{1\hskip-1mm\mathrm{I}}
\newcommand{\bc}{\mathbf{c}\hspace{.1mm}}
\newcommand{\C}{\mathbb{C}}
\newcommand\const{\operatorname{const.}}
\newcommand{\Hn}{\mathbb{H}^n}
\renewcommand\Im{\operatorname{Im}}
\newcommand{\N}{\mathbb{N}}
\renewcommand\Re{\operatorname{Re}}
\newcommand{\R}{\mathbb{R}}
\newcommand{\Rn}{\R^n}
\newcommand{\Sn}{\mathbb{S}^{n-1}}
\newcommand\ssb{\hskip-.25mm}
\newcommand\ssf{\hskip.25mm}
\newcommand{\Z}{\mathbb{Z}}

\begin{document}

\begin{abstract}  
We study the dispersive properties of the wave equation associated
with the shifted Laplace--Beltrami operator on real hyperbolic spaces
and deduce new Strichartz estimates for a large family of admissible pairs. 
As an application, we obtain  local well--posedness results
for the nonlinear wave equation.
\end{abstract}

\maketitle

\section{Introduction}\label{Introduction}
The aim of this paper is to study the dispersive properties of the linear wave equation on real hyperbolic spaces and their application to nonlinear Cauchy problems.  

This theory is well established for the wave equation on $\Rn$\,:
\begin{equation}\label{WaveEuclidean}
\begin{cases}
&\partial_{\ssf t}^{\ssf 2} u(t,x)-\Delta_{\ssf x}u(t,x)=F(t,x)\,,\\
&u(0,x)=f(x)\,,\\
&\partial_{\ssf t}|_{t=0}\,u(t,x)=g(x)\,,\\
\end{cases}
\end{equation}
for which the following Strichartz estimates hold\,:
\begin{equation}\label{eqR3}
\|u\|_{L^p(I;\,L^q)}\ssb
+\ssf\|u\|_{L^\infty(I;\,\dot{H}^{\sigma})}\ssb
+\ssf\|\partial_{\ssf t}u \|_{L^\infty(I;\,\dot{H}^{\sigma-1})}\ssb
\lesssim\ssf\|f\|_{\dot{H}^{\sigma}}\ssb
+\ssf\|g\|_{\dot{H}^{\sigma-1}}\ssb
+\ssf\|F\|_{L^{\tilde{p}'}(I;\,\dot{H}_{\tilde{q}'}^{\sigma+\tilde{\sigma}-1})}    
\end{equation}
on any (possibly unbounded) interval \ssf$I\!\subseteq\ssb\R$\ssf,
under the assumptions that  
\begin{equation*}\textstyle
\sigma = \frac{n+1}2\ssf\bigl(\frac12\ssb-\ssb\frac1q\bigr)\,,\quad 
\tilde{\sigma}=\frac{n+1}2\ssf\bigl(\frac12\ssb-\ssb\frac1{\tilde{q}}\bigr)\,,
\end{equation*}
and separately the couples
\ssf$(p,q),\ssf(\tilde{p},\tilde{q})\ssb
\in\ssb(2,\infty\ssf]\ssb\times\ssb[\ssf2,2\,\frac{n-1}{n-3})$
\ssf satisfy the admissibility conditions 
\begin{equation*}\textstyle
\frac2p\ssb+\ssb\frac{n-1}q\ssb=\ssb\frac{n-1}2\ssf,\quad
\frac2{\tilde{p}}\ssb+\ssb\frac{n-1}{\tilde{q}}\ssb=\ssb\frac{n-1}2\ssf.
\end{equation*}
The estimate (\ref{eqR3}) holds also
at the endpoint $(2,2\,\frac{n-1}{n-3})$ when $n\geq4$.
When $n=3$ this endpoint is $(2,\infty)$
and the estimate (\ref{eqR3}) fails in this case without additional assumptions
(see \cite{GV} and \cite{KT} for more details).

These estimates yield existence results
for the nonlinear wave equation in the Euclidean setting.
The problem of finding minimal regularity on initial data
ensuring local well--posedness for semilinear wave equation
was addressed for higher dimensions and nonlinearities in \cite{Ka},
and then almost completely answered in \cite{ LS, GLS, KT, DGK}.

Once the Euclidean case was more or less settled,
several attemps have been made in order to establish
Strichartz estimates for dispersive equations in other settings.
Here we consider real hyperbolic spaces \ssf$\Hn$,
which are the most simple examples of
noncompact Riemannian manifolds with negative curvature.
For geometric reasons,
we expect better dispersive properties hence stronger results
than in the Euclidean setting.

It is well known that the spectrum of
the Laplace--Beltrami operator \ssf$-\ssf\Delta_{\Hn}$ on $L^2(\Hn)$
is the half--line \ssf$[\ssf\rho^2,+\infty)$\ssf,
where \ssf$\rho\ssb=\!\frac{n-1}2$.
Thus one may study either the {\it non--shifted\/} wave equation
\begin{equation}\label{nonshiftedWave}
\begin{cases}
& \partial_t^{\ssf 2}u(t,x)-\Delta_{\Hn}u(t,x)=F(t,x)\,,\\
& u(0,x)=f(x)\,,\\
& \partial_t|_{t=0}\,u(t,x)=g(x)\,,
\end{cases}
\end{equation}
or the {\it shifted\/} wave equation
\begin{equation}\label{shiftedWave}
\begin{cases}
& \partial_t^{\ssf 2} u(t,x) -(\Delta_{\Hn}\!+\!\rho^2)\,u(t,x) = F(t,x)\,,\\
& u(0,x) = f(x)\,,\\
& \partial_t|_{t=0}\,u(t,x) = g(x)\,.
\end{cases}
\end{equation}

In \cite{P} Pierfelice derived Strichartz estimates
for the wave equation (\ref{nonshiftedWave}) with radial data
on a class of Riemannian manifolds containing all hyperbolic spaces.
The wave equation (\ref{nonshiftedWave}) was also investigated
on the 3--dimensional hyperbolic space by Metcalfe and Taylor \cite{MT},
who proved dispersive and Strichartz estimates
with applications to small data global well--posedness
for the semilinear wave equation.
This result was recently generalized by Anker and Pierfelice \cite{AP4}
to other dimensions.
Another recent work \cite{Ha} by Hassani
contains a first study of \eqref{nonshiftedWave}
on general Riemannian symmetric spaces of noncompact type.

To our knowledge,
the semilinear wave equation (\ref{shiftedWave})
was first considered by Fontaine \cite{F1, F2}
in dimension $n\!=\!3$ and $n\!=\!2$.
The most famous work involving (\ref{shiftedWave}) is due to Tataru.
In \cite{Ta} he obtained dispersive estimates for the operators
$\frac{\sin\ssf\left(t\ssf\sqrt{\Delta_{\Hn}+\rho^2}\ssf\right)}
{\sqrt {\Delta_{\Hn} +\rho^2}}$
and $\cos\ssf\bigl(\ssf t\ssf\sqrt{\Delta_{\Hn}\!+\!\rho^2}\ssf\bigr)$
acting on inhomogeneous Sobolev spaces
and then transferred them from $\Hn$ to $\Rn$
in order to get well--posedness results
for the Euclidean semilinear wave equation
(see also \cite{G}). Though Tataru proved dispersive estimates with exponential decay in time, these are not sufficient to obtain actual Strichartz estimates on hyperbolic spaces.
Complementary results were obtained by Ionescu \cite{I1},
who investigated $L^q\!\to\!L^q$ Sobolev estimates
for the above operators on all hyperbolic spaces.
\smallskip 

In this paper we pursue our study of dispersive equations on hyperbolic spaces,
initiated with the Schr\"odinger equation \cite{AP},
by considering the shifted wave equation (\ref{shiftedWave}) on \ssf$\Hn$.
We obtain a wider range of Strichartz estimates than in the Euclidean setting
and deduce stronger well-posedness results.
More precisely, in Section \ref{Kernel}
we use spherical harmonic analysis on hyperbolic spaces
to estimate the kernel of the operator
$W_t^{(\sigma,\tau)}\!=\ssb
D^{-\tau}\ssf\tilde{D}^{{\tau-\sigma}}\ssf e^{itD}$,
where $D\ssb=\ssb(-\ssf\Delta_{\Hn}\!-\!\rho^2)^{1/2}$,
$\tilde{D}\ssb=\ssb(-\ssf\Delta_{\Hn}\!+\!\tilde{\rho}^2-\rho^2)^{1/2}$
with $\tilde{\rho}\ssb>\ssb\rho\ssf$,
and $\sigma$, $\tau$ are suitable exponents.
In Section \ref{Dispersive} we first deduce
dispersive $L^{q'}\!\to\!L^q$ estimates for $W_t^{(\sigma,\tau)}$,
when $2\!<\!q\!<\!\infty$\ssf,
by using interpolation and the Kunze--Stein phenomenon \cite{Co1, Co2, I2}.
In Section \ref{Strichartz} we next deduce
the following  strong Strichartz estimates for solutions
to the Cauchy problem (\ref{shiftedWave})\,:
\begin{equation}\label{StrichartzEst}
\begin{aligned}
&\|u\|_{L^p(I\ssf;\,L^q)}\ssb
+\ssf\|u\|_{L^\infty\left(I\ssf;\,H^{\ssf\sigma-\frac12,\frac12}\right)}\!
+\ssf\|\partial_{\ssf t}u\|_{L^\infty\left(I\ssf;\,
H^{\ssf\sigma-\frac12,-\frac12}\right)}\\
&\lesssim\,\|f\|_{H^{\ssf\sigma-\frac12,\frac12}}\ssb
+\ssf\|g\|_{H^{\ssf\sigma-\frac12,-\frac12}}\ssb
+\ssf\|F\|_{L^{\tilde{p}'}\left(I\ssf;\,
H_{\tilde{q}'}^{\ssf\sigma+\tilde{\sigma}-1}\right)}\,,
\end{aligned}
\end{equation}
where $I$ is any (possibly unbounded) interval in \ssf$\R$\ssf,
\ssf$(p,q),\ssf(\tilde{p},\tilde{q})\ssb
\in\ssb[\ssf2,\infty)\ssb\times\ssb[\ssf2,\infty)$
are admissible couples such that separately
\begin{equation*}\textstyle
\frac2p\ssb+\ssb\frac{n-1}q\ssb\ge\ssb\frac{n-1}2\,,\quad
\frac2{\tilde p}\ssb+\ssb\frac{n-1}{\tilde q}\ssb\ge\ssb\frac{n-1}2\,,
\end{equation*}
and \ssf$\sigma\!\ge\!\frac{n+1}2\ssf\big(\frac12\!-\!\frac1q\big)$\ssf,
$\tilde{\sigma}\!\ge\!\frac{n+1}2\ssf\big(\frac12\!-\!\frac1{\tilde{q}}\big)$\ssf.
Notice that the Sobolev spaces involved in (\ref{StrichartzEst})
are naturally related to the conservation laws of the shifted wave equation
(see Section \ref{Sobolev}).
We conclude in Section \ref{LWP} with an application of \eqref{StrichartzEst}
to local well--posedness of the nonlinear wave equation
for initial data with low regularity. While we obtain the same regularity curve as in the Euclidean case
for subconformal powerlike nonlinearities,
we prove local well--posedness for superconformal powers
under lower regularity assumptions on the inital data.
\smallskip

In order to keep down the length of this paper,
we postpone applications of the Strichartz estimates 
to global well--posedness of the nonlinear wave equation
and generalizations of the previous results to Damek--Ricci spaces.

\section{Spherical analysis on real hyperbolic spaces}\label{Notation}

In this paper, we consider the simplest class of
Riemannian symmetric spaces of the noncompact type,
namely real hyperbolic spaces $\Hn$ of dimension \ssf$n\!\ge\!2$
\ssf(we shall restrict to \ssf$n\!\ge\!3$ \ssf in Section \ref{LWP})\ssf.
We refer to Helgason's books \cite{Hel1, Hel2, Hel3}
and to Koornwinder's survey \cite{Ko}
for their algebraic structure and geometric properties,
as well as for harmonic analysis on these spaces,
and we shall be content with the following information.
$\Hn$~can be realized as the symmetric space $G/K$,
where $G\ssb=\ssb\text{SO}(1,n)_0$ and $K\!=\ssb\text{SO}(n)$\ssf.
In geodesic polar coordinates on $\Hn$,
the Riemannian volume writes
\begin{equation*}\textstyle
dx=\const\,(\sinh r)^{n-1}\,dr\,d\sigma
\end{equation*}
and the Laplace--Beltrami operator
\begin{equation*}\textstyle
\Delta_{\ssf\Hn}
=\partial_r^{\ssf2}
+(n\!-\!1)\,\coth r\,\partial_r
+\ssf\sinh^{-2}\ssb r\,\Delta_{\ssf\Sn}\,.
\end{equation*}
The spherical functions \ssf$\varphi_\lambda$ on \ssf$\Hn$ are
normalized radial eigenfunctions of $\Delta_{\ssf\Hn}$\,:
\begin{equation*}\begin{cases}
\;\Delta_{\ssf\Hn}\varphi_\lambda=-(\lambda^2\!+\!\rho^2)\,\varphi_\lambda\,,\\
\;\varphi_\lambda(0)=1\,,
\end{cases}\end{equation*}
where \ssf$\lambda\!\in\!\mathbb C$ \ssf
and \ssf$\rho\ssb=\ssb\frac{n-1}2$\ssf.
They can be expressed in terms of special functions\,:
\begin{equation*}\textstyle
\varphi_\lambda(r)
=\phi_{\,\lambda}^{(\frac n2-1,-\frac12)}(r)
={}_2F_1\bigl(\frac\rho2\ssb+\ssb i\ssf\frac\lambda2,
\frac\rho2\ssb-\ssb i\ssf\frac\lambda2;
\frac n2;-\sinh^2\ssb r\bigr)\ssf,
\end{equation*}
where $\phi_\lambda^{(\alpha,\beta)}$ denotes the Jacobi functions
and ${}_2F_1$ the Gauss hypergeometric function.
In the sequel we shall use the integral representations
\begin{equation}\label{intrepr}\begin{aligned}
\varphi_\lambda(r)\,
&=\int_Kdk\;e^{-(\rho+i\lambda)\,\text{H}(a_{-r}k)}\\
&=\,\frac{\Gamma(\frac n2)}{\sqrt{\pi}\,\Gamma(\frac{n-1}2)}\,
\int_0^\pi\,d\theta\,(\sin\theta)^{n-2}\,
(\cosh r\,-\,\sinh r\cos\theta)^{-\rho-i\lambda}\\
&=\,\pi^{-\frac12}\,2^{\frac{n-3}2}
\frac{\Gamma(\frac n2)}{\Gamma(\frac{n-1}2)}\,(\sinh r)^{2-n}
{\int_{-r}^{+r}}\,du\,(\cosh r\,-\,\cosh u)^{\frac{n-3}2}e^{-i\lambda u}\,,
\end{aligned}\end{equation}
which imply in particular that
\begin{equation}\label{phi0}
|\varphi_\lambda(r)|\leq\varphi_0(r)\lesssim (1\ssb+\ssb r)\,e^{-\rho\ssf r}
\qquad\forall\;\lambda\!\in\!\R\,,\;r\!\ge\!0\,.
\end{equation}
We shall also use the Harish--Chandra expansion
\begin{equation}\label{HCexpansion}
\varphi_\lambda(r)
=\mathbf{c}\hspace{.1mm}(\lambda)\,\Phi_\lambda(r)
+\mathbf{c}\hspace{.1mm}(-\lambda)\,\Phi_{-\lambda}(r)
\qquad\forall\;\lambda\!\in\!\C\!\smallsetminus\!\Z\ssf,\;r\!>\!0\ssf,
\end{equation}
where the Harish--Chandra $\mathbf{c}$--function is given by
\begin{equation}\label{cfunction}\textstyle
\mathbf{c}\ssf(\lambda)
=\frac{\Gamma\ssf(\ssf2\ssf\rho\ssf)}{\Gamma\ssf(\ssf\rho\ssf)}\,
\frac{\Gamma\ssf(\ssf i\ssf\lambda\ssf)}
{\Gamma\ssf(\ssf i\ssf\lambda\ssf+\ssf\rho\ssf)}
\end{equation}
and
\begin{equation}\label{Philambda}\begin{aligned}
\Phi_{\lambda}(r)&\textstyle
\,=(\ssf2\sinh r)^{\ssf i\lambda-\rho}\,{}_2F_1\bigl(
\frac\rho2\!-\!i\ssf\frac\lambda2,-\frac{\rho-1}2\!-\!i\ssf\frac\lambda2\ssf;
1\!-\!i\ssf\lambda\ssf;-\sinh^{-2}\ssb r\bigr)\\
&=\,(\ssf2\sinh r)^{-\rho}\,e^{\ssf i\ssf\lambda\ssf r}\,
\sum\nolimits_{k=0}^{+\infty}\,\Gamma_k(\lambda)\,e^{-2\ssf k\ssf r}\\
&\sim\,e^{\ssf(i\lambda-\rho)\ssf r}
\qquad\text{as \,}r\!\to\!+\infty\,.
\end{aligned}\end{equation}
It is well known that
there exist $\nu\!>\!0$\ssf, $\varepsilon\!>\!0$ and $C\!>\!0$
such that, for every $k\!\in\!\N$ and $\lambda\!\in\!\C$
with $\Im\lambda\!>\!-\varepsilon$\ssf,
\begin{equation*}
|\Gamma_k(\lambda)|\le C\,(1\!+\!k)^\nu\,.
\end{equation*}
We need to improve upon this estimate,
by enlarging the domain, by estimating the derivatives of $\Gamma_k$  
and by gaining some additional decay in $\lambda$ \ssf for $k\!\in\!\N^*$.
The following recurrence formula holds\,:
\begin{equation*}\begin{cases}
\;\Gamma_0(\lambda)=1\\
\;\Gamma_k(\lambda)
=\frac{\rho\,(\rho-1)}{k\,(k-i\lambda)}\,
\sum_{\ssf j=0}^{\ssf k-1}\,(k\,-\,j)\,\Gamma_j(\lambda)\,.
\end{cases}\end{equation*}

\begin{lemma}
Let \,$0\!<\!\varepsilon\!<\!1$
and \,$\Omega_\varepsilon\ssb
=\{\,\lambda\!\in\!\C\mid
|\Re\lambda\ssf|\ssb\le\ssb\varepsilon\,|\lambda|\ssf,
\,\Im\lambda\ssb\le\!-1\!+\!\varepsilon\,\}$\ssf.
Then, for every \ssf$\ell\!\in\!\N$\ssf,
there exists \ssf$C_\ell\!>\!0$ such that
\begin{equation}\label{derGammakappa}
\bigl|\,\partial_\lambda^{\,\ell}\ssf\Gamma_k(\lambda)\,\bigr|
\le C_\ell\,k^{\ssf\nu}\,(\ssf1\!+\ssb|\lambda|\ssf)^{-\ell-1}
\qquad\forall\;k\!\in\!\N^*,\,
\lambda\!\in\!\C\ssb\smallsetminus\ssb\Omega_\varepsilon\,.
\end{equation}
\end{lemma}

\begin{proof}
Consider first the case \ssf$\ell=0$\ssf.
There exists \,$A\ssb=\!A(\varepsilon)\!>\!0$
\,such that \,$|\ssf k\!-\!i\ssf\lambda\ssf|\ssb
\ge\ssb A\,\max\,\{\ssf k,1\!+\!|\lambda|\ssf\}$\ssf.
Choose \,$\nu\ssb\ge\!1$ \,such that
\,$\frac{\rho^2}A\ssf\frac1{\nu+1}\ssb\le\ssb\frac12$
\,and \,$C\!>\!0$ \,such that
\,$\frac{\rho^2}A\ssb\le\ssb\frac C2\ssf$.
For \ssf$k\!=\!1$\ssf, we have
\ssf$\Gamma_1(\lambda)=\frac{\rho\,(\rho-1)}{1-\,i\lambda}$\ssf,
hence
\begin{equation*}\textstyle
|\ssf\Gamma_1(\lambda)\ssf|
\le\frac{\rho^2}A\ssf\frac1{1+|\lambda|}
\le C\,\frac1{1+|\lambda|}\,,
\end{equation*}
as required.
For \ssf$k\!>\!1$\ssf, we have
\begin{equation*}\textstyle
\Gamma_k(\lambda)
=\frac{\rho\,(\rho-1)}{k-i\lambda}
+\frac{\rho\,(\rho-1)}{k\,(k-i\lambda)}\,
{\displaystyle\sum\nolimits_{\ssf0<j<k}}
(k\!-\!j)\,\Gamma_j(\lambda)\,,
\end{equation*}
hence
\begin{equation*}\begin{aligned}
|\ssf\Gamma_k(\lambda)\ssf|
&\textstyle
\le\frac{\rho^2}A\ssf\frac1{1+|\lambda|}
+\frac{\rho^2}A\ssf\frac1{k^2}\,
{\displaystyle\sum\nolimits_{\ssf0<j<k}}
(k\!-\!j)\,\frac{C\,j^{\ssf\nu}}{1+|\lambda|}\\
&\textstyle
\le\frac C2\,k^{\ssf\nu}\ssf\frac1{1+|\lambda|}
+C\ssf\frac{k^{\ssf\nu}}{1+|\lambda|}\ssf\frac{\rho^2}A\,\frac1k\,
{\displaystyle\sum\nolimits_{\ssf0<j<k}}
\bigl(\frac jk\bigr)^\nu\\
&\textstyle
\le\,C\,\frac{k^{\ssf\nu}}{1+|\lambda|}\,.
\end{aligned}\end{equation*}
Derivatives are estimated by the Cauchy formula.
\end{proof}

Under suitable assumptions,
the spherical Fourier transform
of a bi--$K$\!--invariant function $f$ on $G$
is defined by 
\begin{equation*}
\mathcal{H}f(\lambda)=\int_Gdg\,f(g)\,\varphi_{\lambda}(g)
\end{equation*}
and the following formulae hold\,:
\begin{itemize}
\item
Inversion formula\,:
\begin{equation*}
f(x)\ssf=\,\const\int_{\,0}^{+\infty}\hspace{-1mm}
d\lambda\,|\mathbf{c}\hspace{.1mm}(\lambda)|^{-2}\,
\mathcal{H}f(\lambda)\,\varphi_{\lambda}(x)\,,
\end{equation*}
\item
Plancherel formula\,:
\begin{equation*}
\bigl\|\ssf f\ssf\bigr\|_{L^2}^2\ssf
=\,\const\int_{\,0}^{+\infty}\hspace{-1mm}
d\lambda\,|\mathbf{c}\hspace{.1mm}(\lambda)|^{-2}\,
|\mathcal{H}f(\lambda)|^2\,.
\end{equation*}
\end{itemize}
Here is a well--known estimate of the Plancherel density\,:
\begin{equation}\label{estimatec}
|\ssf\mathbf{c}\hspace{.1mm}(\lambda)\ssf|^{-2}
\lesssim\,|\lambda|^2\,(1\!+\!|\lambda|)^{n-3}
\qquad\forall\;\lambda\!\in\!\R\,.
\end{equation}
In the sequel we shall use the fact that
\ssf$\mathcal{H}\ssb=\ssb\mathcal{F}\ssb\circ\ssb\mathcal{A}$\ssf,
where $\mathcal{A}$ denotes the Abel transform
and $\mathcal{F}$ the Fourier transform on the real line.
Actually we shall use the factorization
\ssf$\mathcal{H}^{-1}\!=\ssb\mathcal{A}^{-1}\!\circ\ssb\mathcal{F}^{-1}$.
Recall the following expression of the inverse Abel transform\,:
\begin{equation}\label{inv1}\textstyle
\mathcal{A}^{-1}g\ssf(r)=\const\,
\bigl(-\frac1{\sinh r}\frac\partial{\partial r}\bigr)^{\frac{n-1}2}g\ssf(r)\,.
\end{equation}
If $n$ id odd, the right hand side involves a plain differential operator while,
if $n$ is even, the fractional derivative must be interpreted as follows\,:
\begin{equation}\label{inv2}\textstyle
\bigl(-\frac1{\sinh r}\frac\partial{\partial r}\bigr)^{\frac{n-1}2}g\ssf(r)
=\frac1{\sqrt{\pi}}
{\displaystyle\int_{\,r}^{+\infty}}\hspace{-1mm}ds\,
\frac{\sinh s}{\sqrt{\cosh s-\cosh r}}\,
\bigl(-\frac1{\sinh s}\bigr)^{\frac n2}g\ssf(s)\,.
\end{equation}

\section{Sobolev spaces and conservation of energy}\label{Sobolev}

Let us first introduce inhomogeneous Sobolev spaces on hyperbolic spaces $\Hn$,
which will be involved in the conservation laws, in the dispersive estimates
and in the Strichartz estimates for the shifted wave equation.
We refer to \cite{Tr2} for more details about functions spaces on Riemannian manifolds.

Let \,$1\!<\!q\!<\!\infty$ \,and \,$\sigma\!\in\!\R$\ssf.
By definition, $H_q^\sigma(\Hn)$ is the image of $L^q(\Hn)$
under $(-\Delta_{\Hn})^{-\frac\sigma2}$ (in the space of distributions on $\Hn$),
equipped with the norm
\begin{equation*}
\|\ssf f\ssf\|_{H_q^\sigma}=\ssf\|\ssf(-\Delta_{\Hn})^{\frac\sigma2}f\ssf\|_{L^q}\,.
\end{equation*}
In this definition, we may replace $-\Delta_{\Hn}$
by $-\Delta_{\Hn}\!-\!\rho^2\hspace{-1mm}+\!\tilde\rho^{\ssf2}$,
where \ssf$\tilde\rho\!>\!|\frac12\!-\!\frac1q\ssf|\ssf2\ssf\rho$\ssf.
For simpli-
\linebreak
city, we choose \ssf$\tilde\rho\!>\!\rho$
\ssf independently of \ssf$q$
\ssf and we set
\begin{equation*}
\widetilde D=(-\Delta_{\Hn}\ssb-\ssb\rho^2\!+\ssb\tilde\rho^{\ssf2}\ssf)^{\frac12}\,.
\end{equation*}
Thus $H_q^\sigma(\Hn)\!=\!\widetilde D^{-\sigma}L^q(\Hn)$
and $\|\ssf f\ssf\|_{H_q^\sigma}\!
\sim \ssb\|\ssf\widetilde D^{\ssf\sigma\ssb}f\ssf\|_{L^q}$.
If \ssf$\sigma\ssb=\ssb N$ \ssf is a nonnegative integer,
then $H_q^\sigma(\Hn)$ co\"{\i}ncides with
the Sobolev space
\begin{equation*}
W^{N,q}(\Hn)
=\{\,f\!\in\!L^q(\Hn)\mid
\nabla^j\ssb f\!\in\! L^q(\Hn)
\hspace{2mm}\forall\;1\!\le\!j\!\le\!N\,\}
\end{equation*}
defined in terms of covariant derivatives and equipped with the norm
\begin{equation*}
\|f\|_{W^{N,q}}=\ssf\sum\nolimits_{\ssf j=0}^{\,N}
\|\ssf\nabla^j\ssb f\ssf\|_{L^q}\,.
\end{equation*}

\begin{proposition}[Sobolev embedding Theorem]\label{SET}
Let \,$1\!<\!q_1\!<\!q_2\!<\!\infty$ \ssf
and \,$\sigma_1,\sigma_2\!\in\!\R$ such that
\,$\sigma_1\!-\!\frac n{q_1}\ssb\ge\ssb\sigma_2\!-\!\frac n{q_2}$
{\rm(}\footnote{\,Notice that \ssf$\sigma_1\!-\ssb\sigma_2\ssb
\ge\!\frac n{q_1}\!-\!\frac n{q_2}\!>\ssb0$\ssf.}{\rm)}\ssf.
Then
\begin{equation*}
H_{q_1}^{\sigma_1}(\Hn)\subset H_{q_2}^{\sigma_2}(\Hn)\,.
\end{equation*}
By this inclusion, we mean that
there exists a constant \,$C\!>\!0$ \ssf
such that
\begin{equation*}
\|f\|_{H_{q_2}^{\sigma_2}}\le C\,\|f\|_{H_{q_1}^{\sigma_1}}
\qquad\forall\;f\!\in\!C_c^\infty(\Hn)\,.
\end{equation*}
\end{proposition}

\begin{proof}
We sketch two proofs.
The first one is based on the localization principle
for Lizorkin--Triebel spaces \cite{Tr2}
and on the corresponding result in $\R^n$.
More precisely,
given a tame partition of unity
$1\ssb=\sum_{\ssf j=0}^{\,\infty}\varphi_j$
on $\Hn$,
we have
\begin{equation*}
\|f\|_{H_{q_2}^{\sigma_2}(\Hn)}\asymp\ssf
\Bigl\{\,\sum\nolimits_{\ssf j=0}^{\,\infty}\,
\bigl\|\ssf(\varphi_jf\ssf)\!\circ\ssb\exp_{x_j}
\bigr\|_{H_{q_2}^{\sigma_2}(\R^n)}^{\,q_2}
\ssf\Bigr\}^{\frac1{q_2}}\ssf.
\end{equation*}
Using the inclusions
\ssf$H_{q_1}^{\sigma_1}(\R^n)\!
\subset\ssb H_{q_2}^{\sigma_2}(\R^n)$
\ssf and
\,$\ell^{\ssf q_1}(\mathbb N)\!\subset\ssb\ell^{\ssf q_2}(\mathbb N)$\ssf,
we conclude that
\begin{equation*}
\|f\|_{H_{q_2}^{\sigma_2}(\Hn)}
\lesssim\,\Bigl\{\,\sum\nolimits_{\ssf j=0}^{\,\infty}\,
\bigl\|\ssf(\varphi_jf\ssf)\!\circ\ssb\exp_{x_j}
\bigr\|_{H_{q_1}^{\sigma_1}\ssb(\R^n)}^{\,q_1}
\ssf\Bigr\}^{\frac1{q_1}}
\asymp\,\|f\|_{H_{q_1}^{\sigma_1}\ssb(\Hn)}.
\end{equation*}
The second proof is based
on the \ssf$L^{q_1}\!\to\!L^{q_2}$ mapping properties
of the convolution operator \ssf$\widetilde D^{\ssf\sigma_2-\sigma_1}$
(see \cite{CGM1} and the references cited therein).
\end{proof}

\bigskip

Beside the $L^q$ Sobolev spaces $H_{\ssf q}^\sigma(\Hn)$,
our analysis of the shifted wave equation on $\Hn$
involves the following $L^2$ Sobolev spaces\,:
\begin{equation*}
H^{\sigma,\tau}(\Hn)=\ssf
\widetilde D^{-\sigma}D^{-\tau}L^2(\Hn)\ssf,
\end{equation*}
where \ssf$D\!=\!(-\Delta_{\Hn}\!-\!\rho^{\ssf2})^{\frac12}$\ssf,
\ssf$\sigma\!\in\!\R$ \ssf
and \ssf$\tau\!<\!\frac32$
\ssf(actually we are only interested in the cases
\ssf$\tau\ssb=\ssb0$ \ssf and \ssf$\tau\ssb=\ssb\pm\ssf\frac12$\ssf).
Notice that
\begin{equation*}\begin{cases}
\,H^{\sigma,\tau}(\Hn)\ssb
=\ssb H_{\ssf2}^{\sigma}(\Hn)
&\text{if \,}\tau\ssb=\ssb0\ssf,\\
\,H^{\sigma,\tau}(\Hn)\ssb
\subset\ssb H_{\ssf2}^{\sigma+\tau}(\Hn)
&\text{if \,}\tau\ssb<\ssb0\ssf,\\
\,H^{\sigma,\tau}(\Hn)\ssb
\supset\ssb H_{\ssf2}^{\sigma+\tau}(\Hn)
&\text{if \,}0\ssb<\ssb\tau\ssb<\ssb\frac32\ssf.\\
\end{cases}\end{equation*}

\begin{lemma}
If \;$0\ssb<\!\tau\!<\frac32$\ssf, then
\begin{equation*}
H^{\sigma,\tau}(\Hn)\subset
H_2^{\sigma+\tau}(\Hn)+H_{2^+}^\infty(\Hn)\ssf,
\end{equation*}
where $H_{2^+}^\infty(\Hn)\ssb=
\bigcap_{\,\substack{s\in\R\\q>2}\vphantom{\big|}}\ssb
H_q^s(\Hn)$
{\rm (}recall that $H_q^s(\Hn)$ is decreasing
as $q\!\searrow\!2$ and $s\!\nearrow\!+\infty)$.
\end{lemma}

\begin{proof}
Let \ssf$f\!\in\!L^2(\Hn)$\ssf.
We have
\ssf$\widetilde{D}^{-\sigma}D^{-\tau}f\ssb
=\ssb f*k_{\ssf\sigma,\tau}$\ssf,
\ssf where
\begin{equation*}
k_{\ssf\sigma,\tau}(x)\ssf=\,\const\int_{\,0}^{+\infty}\hspace{-1mm}
d\lambda\;|\ssf\mathbf c(\lambda)\ssf|^{-2}\,|\lambda|^{-\tau}\,
(\lambda^2\hspace{-.75mm}+\!\tilde\rho^{\ssf2})^{-\frac\sigma2}\,
\varphi_\lambda(x)
\end{equation*}
by the inversion formula
for the spherical Fourier transform on \ssf$\Hn$.
Let us split up the integral
\begin{equation*}
\int_{\,0}^{+\infty}\hspace{-1mm}=\,\int_{\,0}^{\,1}+\,\int_{\,1}^{+\infty}
\end{equation*}
and the kernel
\begin{equation*}
k_{\ssf\sigma,\tau}\ssf
=\,k_{\ssf\sigma,\tau}^{\,0}
+\,k_{\ssf\sigma,\tau}^{\ssf\infty}\,,
\end{equation*}
accordingly. On the one hand,
\begin{equation*}
\1_{\ssf(1,+\infty)\ssf}(D)\,\widetilde D^{-\sigma}\ssf D^{-\tau}f
=f\ssb*\ssb k_{\ssf\sigma,\tau}^{\ssf\infty}
\end{equation*}
maps $L^2(\Hn)$ into $H_{\ssf2}^{\sigma+\tau}(\Hn)$.
On the other hand,
$k_{\ssf\sigma,\tau}^{\,0}$ is a radial kernel in $H_2^\infty(\Hn)$,
hence
\begin{equation*}
\1_{\ssf[\ssf0,1\ssf]\ssf}(D)\,\widetilde D^{-\sigma}\ssf D^{-\tau}f
=f\ssb*\ssb k_{\ssf\sigma,\tau}^{\,0}
\end{equation*}
maps $L^2(\Hn)$ into $H_{2^+}^\infty(\Hn)$
by the Kunze-Stein phenomenon.
Thus $\widetilde D^{-\sigma}D^{-\tau}f\ssb
=\ssb f\ssb*\ssb k_{\ssf\sigma,\tau}$
belongs to $H_2^{\sigma+\tau}(\Hn)\ssb
+\ssb H_{2^+}^\infty(\Hn)$, 
as required.
\end{proof}

Let us next introduce the energy
\begin{equation}\label{energy}\textstyle
E(t)=\frac12{\displaystyle\int_{\Hn}}\hspace{-1mm}dx\,
\bigl\{\ssf|\partial_{\ssf t}u(t,x)|^2\ssb+|D_xu(t,x)|^2\ssf\bigr\}
\end{equation}
for solutions to the homogeneous Cauchy problem
\begin{equation}\label{wavehom}
\begin{cases}
&\partial_{\ssf t}^{\ssf2}u-(\Delta_{\Hn}\!+\!\rho^2)\ssf u=0\,,\\
&u(0,x)=f(x)\,,\\
&\partial_{\ssf t}|_{t=0}\,u(t,x)=g(x)\,.
\end{cases}
\end{equation}
It is easily verified that $\partial_{\ssf t}E(t)\!=\!0$\ssf,
hence \eqref{energy} is conserved.
In other words,
for every time \ssf$t$ \ssf in the interval of definition of \ssf$u$\ssf,
\begin{equation*}
\|\ssf\partial_{\ssf t}u(t,x)\|_{L_x^2}^{\ssf2}\ssb
+\|D_xu(t,x)\|_{L_x^2}^{\ssf2}
=\|g\|_{L^2}^{\ssf2}\ssb+\|D\ssb f\|_{L^2}^{\ssf2}\,.
\end{equation*}
Let \ssf$\sigma\!\in\!\R$ \ssf and \ssf$\tau\!<\!\frac32$\ssf.
By applying the operator \ssf$\tilde{D}^\sigma\ssb D^\tau$
to  (\ref{wavehom}), we deduce that
\begin{equation*}
\|\,\partial_{\ssf t}\ssf\tilde{D}_x^\sigma D_x^\tau\ssf
u(t,x)\ssf\|_{L_x^2}^{\ssf2}
+\|\ssf\tilde{D}_x^\sigma D_x^{\tau+1}\ssf u(t,x)\ssf\|_{L_x^2}^{\ssf2}
=\|\ssf \tilde{D}^\sigma\ssb D^\tau\ssb g\ssf\|_{L^2}^{\ssf2}
+\|\ssf \tilde{D}^\sigma\ssb D^{\tau+1}\ssb f\ssf\|_{L^2}^{\ssf2}\,,
\end{equation*}
which can be rewritten in terms of Sobolev norms as follows\,:
\begin{equation}\label{conservationenergy}
\|\ssf\partial_{\ssf t}u(t,\cdot)\|_{H^{\sigma,\tau}}^{\ssf2}\ssb
+\|u(t,\cdot)\|_{H^{\sigma,\tau+1}}^{\ssf2}\ssb
=\|g\|_{H^{\sigma,\tau}}^{\ssf2}\ssb
+\|f \|_{H^{\sigma,\tau+1}}^{\ssf2}\,.
\end{equation}

\section{Kernel estimates}\label{Kernel}

In this section we derive pointwise estimates
for the radial convolution kernel \ssf$w_{\,t}^{(\sigma,\tau)}$
of the operator \ssf$W_t^{(\sigma,\tau)}\!
=\ssb D^{-\tau}\tilde{D}^{\ssf\tau-\sigma}e^{\,i\,t\ssf D}$,
for suitable exponents  $\sigma\!\in\!\R$ and $\tau\!\in\![0,\frac32)$\ssf.
By the inversion formula of the spherical Fourier transform,
\begin{equation*}
w_{\,t}^{(\sigma, \tau)}(r)=\const
\int_{\,0}^{+\infty}\hspace{-1mm}d\lambda\;|\mathbf{c}\hspace{.1mm}(\lambda)|^{-2}\,
\lambda^{-\tau}\,(\lambda^2\!+\!{\tilde\rho}^{\ssf2})^{\frac{\tau-\sigma}2}\,
\varphi_\lambda(r)\,e^{\ssf i\ssf t\ssf\lambda}\,.
\end{equation*}
Contrarily to the Euclidean case,
this kernel has different behaviors,
depending whether \ssf$t$ \ssf is small or large,
and therefore we cannot use any rescaling.
Let us split up
\begin{equation*}\begin{aligned}
w_{\,t}^{(\sigma,\tau)}(r)
&=w_{\,t,0}^{(\sigma,\tau)}(r)+w_{\,t,\infty}^{(\sigma,\tau)}(r)\\
&=\const\int_{\,0}^{\ssf2}\!d\lambda\,
\chi_0(\lambda)\,|\mathbf{c}\hspace{.1mm}(\lambda)|^{-2}\,
\lambda^{-\tau}\,(\lambda^2\!+\!{\tilde\rho}^{\ssf2})^{\frac{\tau-\sigma}2}\,
\varphi_\lambda(r)\,e^{\ssf i\ssf t\ssf\lambda}\\
&+\const\int_{\,1}^{+\infty}\hspace{-1mm}d\lambda\,
\chi_\infty(\lambda)\,|\mathbf{c}\hspace{.1mm}(\lambda)|^{-2}
\,\lambda^{-\tau}\,(\lambda^2\!+\!{\tilde\rho}^{\ssf2})^{\frac{\tau-\sigma}2}\,
\varphi_\lambda(r)\,e^{\ssf i\ssf t\ssf\lambda}
\end{aligned}\end{equation*}
using smooth cut--off functions $\chi_0$ and $\chi_\infty$ on $[0,+\infty)$
such that $1\!=\ssb\chi_0\ssb+\ssb\chi_\infty$\ssf,
$\chi_0\ssb=\!1$ on $[\ssf0,1\ssf]$ and
$\chi_\infty\!=\!1$ on $[\ssf2,+\infty)$\ssf.
We shall first estimate \ssf$w_{\,t,0}^{(\sigma,\tau)}$
and next a variant of \ssf$w_{\,t,\infty}^{(\sigma,\tau)}$.
The kernel \ssf$w_{\,t,\infty}^{(\sigma,\tau)}$ has indeed
a logarithmic singularity on the sphere \ssf$r\!=\ssb t$
\ssf when \ssf$\sigma\ssb=\ssb\frac{n+1}2$.
We bypass this problem
by considering the analytic family of operators
\begin{equation*}\textstyle
\widetilde{W}_{\,t,\infty}^{\ssf(\sigma,\tau)}
=\frac{e^{\ssf\sigma^2}}{\Gamma(\frac{n+1}2-\sigma)}\;
\chi_\infty(D)\,D^{-\tau}\,\tilde{D}^{\ssf\tau-\sigma}\,e^{\,i\,t\ssf D}
\end{equation*}
in the vertical strip \ssf$0\!\le\!\Re\sigma\!\le\!\frac{n+1}2$
\ssf and the corresponding kernels
\begin{equation}\label{ISFT}\textstyle
\widetilde{w}_{\,t,\infty}^{\ssf(\sigma,\tau)}(r)
=\frac{e^{\ssf\sigma^2}}{\Gamma(\frac{n+1}2-\sigma)}\,
{\displaystyle\int_{\,1}^{+\infty}}\hskip-1mm
d\lambda\,\chi_\infty(\lambda)\,|\mathbf{c}\hspace{.1mm}(\lambda)|^{-2}\,\lambda^{-\tau}\,
(\lambda^2\hskip-.75mm+\!{\tilde\rho}^{\ssf2})^{\frac{\tau-\sigma}2}\,
e^{\,i\ssf t\ssf\lambda}\,\varphi_\lambda(r)\,.
\end{equation}
Notice that the Gamma function,
which occurs naturally in the theory of Riesz distributions,
will allow us to deal with the boundary point \ssf$\sigma\!=\!\frac{n+1}2$,
while the exponential function
yields boundedness at infinity in the vertical strip.
Notice also that, once multiplied by \ssf$\chi_\infty(D)$\ssf,
the operator \ssf$D^{-\tau}\tilde{D}^{\ssf\tau-\sigma}$
\ssf behaves like \ssf$\tilde{D}^{-\sigma}$\ssf.

\subsection{Estimate of
\,$w_{\ssf t}^{\ssf0}\ssb=\ssb w_{\,t,0}^{(\sigma,\tau)}$.}
\label{KernelEstimatewt0}

\begin{theorem}\label{Estimatewt0}
Let \,$\sigma\!\in\!\R$ and \,$\tau\!<\!2$\ssf.
The following pointwise estimates hold for the kernel
\,$w_{\ssf t}^{\ssf0}\ssb=\ssb w_{\,t,0}^{(\sigma,\tau)}:$
\begin{itemize}
\item[(i)]
Assume that \,$|t|\!\le\ssb2$\ssf.
Then, for every \,$r\ssb\ge\ssb0$\ssf,
\begin{equation*}
|\ssf w_{\ssf t}^{\ssf0}(r)|\ssf\lesssim\ssf\varphi_0(r)\ssf.
\end{equation*}
\item[(ii)]
Assume that \,$|t|\!\ge\ssb2$\ssf.
\vspace{1mm}
\begin{itemize}
\item[(a)]
If \,$0\ssb\le\ssb r\ssb\le\ssb\frac{|t|}2$\ssf, then
\begin{equation*}
|\ssf w_{\ssf t}^{\ssf0}(r)|\ssf
\lesssim\ssf|t|^{\ssf\tau-3}\,\varphi_0(r)\ssf.
\end{equation*}
\item[(b)]
If \,$r\ssb\ge\ssb\frac{|t|}2$\ssf, then
\begin{equation*}
|\ssf w_{\ssf t}^{\ssf0}(r)|\ssf\lesssim\ssf
(\ssf1\!+\ssb|\ssf r\!-\!|t|\ssf|\ssf)^{\ssf\tau-2}\,e^{-\rho\ssf r}\ssf.
\end{equation*}
\end{itemize}
\end{itemize}
\end{theorem}

\begin{proof}
Recall that
\begin{equation}\label{wt0}
w_{\ssf t}^{\ssf0}(r)=\const\int_{\,0}^{\ssf2}\!d\lambda\,
\chi_0(\lambda)\,|\mathbf{c}\hspace{.1mm}(\lambda)|^{-2}\,
\lambda^{-\tau}\,(\lambda^2\!+\!{\tilde\rho}^{\ssf2})^{\frac{\tau-\sigma}2}\,
\varphi_\lambda(r)\,e^{\ssf i\ssf t\ssf\lambda}\,.
\end{equation}
By symmetry we may assume that \ssf$t\!>\!0$\ssf.

\noindent
(i) It follows from the estimates \eqref{phi0} and \eqref{estimatec} that
\begin{equation*}
|\ssf w_{\ssf t}^{\ssf0}(r)|\,
\lesssim\int_{\,0}^{\ssf2}d\lambda\,\lambda^{2-\tau}\,\varphi_0(r)\,
\lesssim\,\varphi_0(r)\,.
\end{equation*}

\noindent
(ii) We prove first (a) by substituting in \eqref{wt0}
the first integral representation of $\varphi_\lambda$ in \eqref{intrepr}
and by reducing this way to Fourier analysis on \ssf$\R$\ssf.
Specifically,
\begin{equation*}
w_{\ssf t}^{\ssf0}(r)\ssf=\int_{\ssf K}dk\,e^{-\rho\,\text{H}(a_{-r}k)}
\int_{\,0}^{\ssf2}\!d\lambda\,\chi_0(\lambda)\,a(\lambda)\,
e^{\,i\ssf\{\ssf t-\text{H}(a_{-r}k)\ssf\}\ssf\lambda}\,,
\end{equation*}
where \ssf$a(\lambda)\ssb=\ssb
|\mathbf{c}\hspace{.1mm}(\lambda)|^{-2}\ssf\lambda^{-\tau}\ssf
(\lambda^2\ssb+\ssb{\tilde\rho}^{\ssf2})^{\frac{\tau-\sigma}2}$\ssf,
up to a positive constant.
According to the estimate (\ref{estimatec}) and to Lemma A.1 in Appendix A,
the inner integral is bounded above by
\begin{equation*}
\bigl\{\ssf t\ssb-\ssb\text{H}(a_{-r}k)\bigr\}^{\ssf\tau-3}
\le\ssf(\ssf t\ssb-\ssb r\ssf)^{\ssf\tau-3}\ssf
\asymp\,t^{\ssf\tau-3}\,.
\end{equation*}
Since
\begin{equation*}
\int_{\ssf K}dk\,e^{-\rho\,\text{H}(a_{-r}k)}\ssf=\,\varphi_0(r)\,,
\end{equation*}
we conclude that
\begin{equation*}
|\ssf w_{\ssf t}^{\ssf0}(r)|\,\lesssim\,t^{\ssf\tau-3}\,\varphi_0(r)\,.
\end{equation*}
We prove next (b) by substituting in \eqref{wt0}
the Harish--Chandra expansion (\ref{HCexpansion}) of $\varphi_\lambda$
and by reducing again to Fourier analysis on \ssf$\R$\ssf.
Specifically,
\begin{equation}\label{sumK}
w_{\ssf t}^{\ssf0}(r)=(\ssf2\sinh r)^{-\rho}\,\sum\nolimits_{\ssf k=0}^{+\infty}
e^{-2\ssf k\ssf r}\,\big\{\ssf I_{\,k}^{+,0}(t,r)+I_{\,k}^{-,0}(t,r)\ssf\big\}\,,
\end{equation}
where
\begin{equation*}
I_{\,k}^{\pm,0}(t,r)\ssf=\ssb\int_{\,0}^{\ssf2}\!d\lambda\,
\chi_0(\lambda)\,a_{\ssf k}^\pm(\lambda)\,e^{\,i\ssf(t\ssf\pm\ssf r)\ssf\lambda}
\end{equation*}
and
\begin{equation*}
a_{\ssf k}^\pm(\lambda)
=\ssf\mathbf{c}\hspace{.1mm}(\mp\lambda)^{-1}\,
\lambda^{-\tau}\,(\lambda^2\!+\ssf{\tilde\rho}^{\ssf2})^{\frac{\tau-\sigma}2}\,
\Gamma_k(\pm\lambda)\,.
\end{equation*}
By applying Lemma A.1
and by using the estimates \eqref{derGammakappa}
for $\Gamma_k$ and its derivatives,
we obtain
\begin{equation*}
|\ssf I_{\,k}^{+,0}(t,r)\ssf|\ssf
\lesssim(\ssf1\!+\ssb k\ssf)^{\ssf\nu}\,(\ssf t\!+\!r\ssf)^{\ssf\tau-2}
\le(\ssf1\!+\ssb k\ssf)^{\ssf\nu}\;r^{\ssf\tau-2}
\end{equation*}
and
\begin{equation*}
|\ssf I_{\,k}^{-,0}(t,r)\ssf|\ssf\lesssim
(\ssf1\!+\ssb k\ssf)^{\ssf\nu}\,
(\ssf1\ssb+\ssb|\ssf r\!-\!t\ssf|\ssf)^{\ssf\tau-2}\,.
\end{equation*}
We conclude the proof by summing up these estimates in \eqref{sumK}.
\end{proof} 

\subsection{Estimate of
\,$\widetilde{w}_{\,t}^{\ssf\infty}\ssb
=\ssb\widetilde{w}_{\,t,\infty}^{\ssf(\sigma,\tau)}$.}
\label{KernelEstimatewtildetinfty}

\begin{theorem}\label{Estimatewtildetinfty}
The following pointwise estimates hold for the kernel
\,$\widetilde{w}_{\,t}^{\ssf\infty}\ssb
=\ssb\widetilde{w}_{\,t,\infty}^{\ssf(\sigma,\tau)}$,
for any fixed \ssf$\tau\!\in\!\R$
and uniformy in \ssf$\sigma\!\in\!\C$
with \ssf$\Re\sigma\ssb=\ssb\frac{n+1}2:$
\begin{itemize}
\item[(i)]
Assume that \,$0\!<\!|t|\!\le\!2$\ssf.
\begin{itemize}
\item[(a)]
\,If \,$0\!\le\!r\!\le\!3$\ssf, then
\;$|\,\widetilde{w}_{\,t}^{\ssf\infty}(r)\ssf|\ssf
\lesssim\,\begin{cases}
\;|t|^{-\frac{n-1}2}
&\text{if \;}n\ssb\ge\ssb3\ssf,\\
\;|t|^{-\frac12}\ssf(\ssf1\!-\ssb\log|t|\ssf)
&\text{if \;}n\ssb=\ssb2\ssf.\\
\end{cases}$
\item[(b)]
\,If \,$r\!\ge\!3$\ssf, then
\,$\widetilde{w}_{\,t}^{\ssf\infty}(r)
=\mathrm{O}\bigl(\ssf r^{-\infty}\,e^{-\rho\ssf r}\ssf\bigr)$\ssf.
\end{itemize}
\vspace{1mm}
\item[(ii)]
Assume that \,$|t|\!\ge\!2$\ssf. Then
\begin{equation*}\label{wtildeinftytlarge}
|\,\widetilde{w}_{\,t}^{\ssf\infty}(r)\ssf|\,\lesssim\,
(\ssf1\ssb+\ssb|\ssf r\ssb-\ssb|t|\ssf|\ssf)^{-\infty}\,e^{-\rho\,r}
\qquad\forall\;r\!\ge\!0\ssf.
\end{equation*}
\end{itemize} 
\end{theorem}
\vspace{2mm}

\noindent
\textit{Proof of Theorem \ref{Estimatewtildetinfty}.ii.}
Recall that, up to a positive constant,
\begin{equation*}\label{wtildetinfty}
\widetilde{w}_{\,t}^{\ssf\infty}(r)
={\textstyle\frac{e^{\ssf\sigma^2}}{\Gamma(\frac{n+1}2-\sigma)}}
\int_{\,1}^{+\infty}\hspace{-1mm}d\lambda\;
\chi_\infty(\lambda)\,|\mathbf{c}\hspace{.1mm}(\lambda)|^{-2}\,
\lambda^{-\tau}\,(\lambda^2\!+\!{\tilde\rho}^{\ssf2})^{\frac{\tau-\sigma}2}\,
\varphi_\lambda(r)\,e^{\ssf i\ssf t\ssf\lambda}\,.
\end{equation*}
By symmetry we may assume again that \ssf$t\!>\!0$\ssf.
If \ssf$0\!\le\!r\!\le\!\frac t2$\ssf,
we resume the proof of Theorem \ref{Estimatewt0}.ii.a,
using Lemma A.2 instead of Lemma A.1,
and estimate this way
\begin{equation}\label{estimate1wtildetinfty}
|\,\widetilde{w}_{\,t}^{\ssf\infty}(r)\ssf|\,
\lesssim\,(\ssf t\!-\!r\ssf)^{-\infty}\,\varphi_0(r)\,
\lesssim\,t^{-\infty}\,e^{-\rho\ssf r}\,.
\end{equation}
If \ssf$r\!\ge\!\frac t2$\ssf,
we resume the proof of Theorem \ref{Estimatewt0}.ii.b
and expand this way
\begin{equation}\label{expansionwtildetinfty}
\widetilde{w}_{\,t}^{\ssf\infty}(r)=
{\textstyle\frac{e^{\ssf\sigma^2}}{\Gamma(\frac{n+1}2-\sigma)}}\,
(\sinh r)^{-\rho}\,\sum\nolimits_{\ssf k=0}^{+\infty}e^{-2\ssf k\ssf r}\ssf
\bigl\{\ssf I_{\,k}^{+,\infty}(t,r)+I_{\,k}^{-,\infty}(t,r)\ssf\bigr\}\,,
\end{equation}
where
\begin{equation*}
I_{\,k}^{\pm,\infty}(t,r)\ssf
=\int_{\,0}^{+\infty}\hspace{-1mm}d\lambda\,\chi_{\infty}(\lambda)\,
a_{\ssf k}^\pm(\lambda)\,e^{\ssf i\ssf(t\pm r)\ssf\lambda}
\end{equation*}
and
\begin{equation*}
a_{\ssf k}^\pm(\lambda)
=\mathbf{c}\hspace{.1mm}(\mp\lambda)^{-1}\,
\lambda^{-\tau}\,(\lambda^2\!+\ssf{\tilde\rho}^{\ssf2})^{\frac{\tau-\sigma}2}\,
\Gamma_k(\pm\lambda)\,.
\end{equation*}
It follows from the expression \eqref{cfunction} of the $\bc$--function
and from the estimates \eqref{derGammakappa} of the coefficients \,$\Gamma_k$
\ssf that \,$\chi_\infty\ssf a_{\ssf k}^\pm$ \ssf is a symbol of order
\begin{equation*}
d\ssf=\begin{cases}
-1
&\text{if \,}k\!=\!0\ssf,\\
-2
&\text{if \,}k\!\in\!\N^*\ssb.\\
\end{cases}
\end{equation*}
By applying Lemma A.2,
we obtain the following estimates
of the expressions $I_{\,k}^{\pm,\infty}(t,r)$,
except for $I_{\,0}^{-,\infty}(t,r)$\,:
$\forall\,N\!\in\!\N^*$,
$\exists\;C_N\!\ge\!0$\ssf,
\begin{eqnarray}
\label{Ikplusinfty}
&|\ssf I_{\,k}^{+,\infty}(t,r)\ssf|\ssf
\le\ssf C_N\,|\sigma|^N\,(\ssf1\!+\ssb k\ssf)^{\ssf\nu}\,(\ssf t\!+\!r\ssf)^{-N}
\le\ssf C_N\,|\sigma|^N\,(\ssf1\!+\ssb k\ssf)^{\ssf\nu}\;r^{-N}\,,\\
\label{Ikminusinfty}
&|\ssf I_{\,k}^{-,\infty}(t,r)\ssf|\ssf
\le\ssf C_N\,|\sigma|^N\,(\ssf1\!+\ssb k\ssf)^{\ssf\nu}\,
(\ssf1\ssb+\ssb|\ssf r\!-\ssb t\ssf|\ssf)^{-N}\,.
\end{eqnarray}
As far as \ssf$I_{\,0}^{-,\infty}(t,r)$ \ssf is concerned,
Lemma A.2 yields the estimates
\begin{equation}\label{I0minusinfty}
|\ssf I_{\,0}^{-,\infty}(t,r)\ssf|\ssf\le
\begin{cases}
\,C_N\,|\sigma|^N\,|\ssf r\!-\ssb t\ssf|^{-N}
&\text{if \,}|\ssf r\!-\!t\ssf|\!\ge\!1\ssf,\\
\,C\,\bigl(\ssf1\!+\ssb\log\frac1{|\ssf r-t\ssf|}\ssf\bigr)
&\text{if \,}|\ssf r\!-\!t\ssf|\!\le\!1\ssf.\\
\end{cases}\end{equation}
The second one can be improved
by applying Lemma A.3 instead of Lemma A.2.
For this purpose,
let us establish the asymptotic behavior
of the symbol \ssf$a_{\ssf0}^-(\lambda)$\ssf,
as \ssf$\lambda\!\to\!+\infty\ssf$.
On the one hand,
\begin{equation*}\begin{aligned}\textstyle
\bc(\lambda)^{-1}=\ssf\frac{\Gamma(\rho)}{\Gamma(2\rho)}\,
\frac{\Gamma(i\lambda+\rho)}{\Gamma(i\lambda)}
&\textstyle
=\ssf\frac{\Gamma(\rho)}{\Gamma(2\rho)}\;e^{-\rho}\,
\bigl(\frac{i\lambda+\rho}{i\lambda}\bigr)^{i\lambda-\frac12}\,
(\ssf i\ssf\lambda\!+\!\rho\ssf)^{\rho}\,
\bigl\{\ssf1\ssb+\text{O}\,(\lambda^{-1})\ssf\bigr\}\\
&\textstyle
=\ssf e^{\ssf i\frac{\rho\pi}2}\,\lambda^\rho\,
\bigl\{\ssf1\ssb+\text{O}(\lambda^{-1})\ssf\bigr\}\,,
\end{aligned}\end{equation*}
according to Stirling's formula
\begin{equation*}\label{Stirling}
\Gamma(\xi)=\sqrt{2\ssf\pi\ssf}\,
\xi^{\ssf\xi-\frac12}\,e^{-\xi}\,
\bigl\{\ssf1\ssb+\text{O}\ssf(\ssf|\xi|^{-1})\ssf\bigr\}\,.
\end{equation*}
On the other hand,
\begin{equation*}
\lambda^{-\tau}(\lambda^2\!+\ssb{\tilde\rho}^{\ssf2})^{\frac{\tau-\sigma}2}
=\lambda^{-\sigma}\,\bigl\{\ssf1\ssb
+\text{O}\ssf(\ssf|\sigma|\ssf\lambda^{-2}\ssf)\bigr\}\,.
\end{equation*}
Hence
\begin{equation*}
a_{\ssf0}^-(\lambda)
=c_{\ssf0}\,\lambda^{-1-i\ssf\Im\sigma}\ssb+b_{\ssf0}(\lambda)
\quad\text{with}\quad
|\ssf b_{\ssf0}(\lambda)|\le C\,|\sigma|\,\lambda^{-2}\,.
\end{equation*}
By applying Lemma A.3
with \ssf$m\ssb=\ssb0$ \ssf and \ssf$d\ssb=\ssb-\ssf2$\ssf,
we obtain
\begin{equation}\label{I0minusinftybis}\textstyle
|\ssf I_{\,0}^{-,\infty}(t,r)\ssf|\ssf
\le\,C\,\frac{|\sigma|^2}{|\Im\sigma\ssf|}
\qquad\text{if \,}|\ssf r\!-\!t\ssf|\!\le\!1\ssf.
\end{equation}
Instead of the singularity \ssf$\log\frac1{|\ssf r-t\ssf|}$
\ssf in \eqref{I0minusinfty},
the estimate \eqref{I0minusinftybis} of \ssf$I_{\,0}^{-,\infty}(t,r)$
\ssf involves this time
the singularity \ssf$\frac1{\Im\sigma}$\ssf,
which cancels with the denominator of the front expression
\begin{equation}\label{fraction}\textstyle
\frac{e^{\ssf\sigma^2\vphantom{|}}}{\Gamma(\frac{n+1}2\ssf-\,\sigma)}
\end{equation}
in \eqref{expansionwtildetinfty}.
Notice moreover that the numerator of \eqref{fraction}
yields enough decay to get uniform bounds in \ssf$\sigma$\ssf.
In conclusion, by combining
\eqref{expansionwtildetinfty},
\eqref{Ikplusinfty},
\eqref{Ikminusinfty},
\eqref{I0minusinfty},
\eqref{I0minusinftybis},
we obtain
\begin{equation*}\textstyle
|\,\widetilde{w}_{\,t}^{\ssf\infty}(r)\ssf|\,
\lesssim\,(\ssf1\ssb+\ssb|\ssf r\!-\ssb t\ssf|\ssf)^{-\infty}\,e^{-\rho\,r}
\qquad\forall\;r\!\ge\!\frac t2\,.
\end{equation*}
\hfill$\square$

\begin{remark}\label{Estimatewtinfty}
The kernel \,$w_{\,t}^\infty(r)$ can be estimated in the same way,
except that  
\begin{equation*}\textstyle
|\,w_{\,t}^\infty(r)\ssf|\,
\lesssim\,e^{-\rho\ssf t}\,\log\frac1{|\ssf r-|t|\ssf|}
\end{equation*}
when \,$r$ is close to \,$|t|$\ssf.
\end{remark}

Let us turn to the small time estimates in Theorem \ref{Estimatewtildetinfty}.
The estimate (i.a) is of local nature and thus similar to the Euclidean case.
For the  sake of completeness, we include a proof in Appendix C.
It remains for us to prove the estimate (i.b).  
\medskip

\noindent
\textit{Proof of Theorem \ref{Estimatewtildetinfty}.i.b.}
Here \ssf$0\!<\!|t|\!\le2$
\ssf and \ssf$r\!\ge\!3$\ssf.
By symmetry we may assume again that \ssf$t\!>\!0$\ssf.
We use now the inverse Abel transform
given by Formulae \eqref{inv1} and \eqref{inv2}.
Up to positive constants,
the inverse spherical Fourier transform \eqref{ISFT}
can be rewritten in the following way\,:
\begin{equation*}\textstyle
\widetilde{w}_{\,t}^{\ssf\infty}(r)\,
=\,\frac{e^{\ssf\sigma^2}}{\Gamma(\frac{n+1}2-\sigma)}\;
\mathcal{A}^{-1}g_{\ssf t}\ssf(r)\,,
\end{equation*}
where
\begin{equation*}
g_{\ssf t}(r)\ssf=\ssf2\int_{\,1}^{+\infty}\hspace{-1mm}d\lambda\;
\chi_\infty(\lambda)\,\lambda^{-\tau}\,
(\lambda^2\hspace{-1mm}+\!\tilde\rho^{\ssf2})^{\frac{\tau-\sigma}2}\,
e^{\ssf i\ssf t\ssf\lambda}\,\cos\lambda\ssf r\,.
\end{equation*}
Let us split up
\,$2\cos\lambda\ssf r\ssb
=\ssb e^{\ssf i\ssf\lambda\ssf r}\!+\ssb e^{-i\ssf\lambda\ssf r}$
and \,$g_{\ssf t}\ssf(r)\!=\ssb g_{\,t}^+(r)\!+\ssb g_{\,t}^-(r)$
\ssf accordingly, so that
\begin{equation*}
g_{\,t}^\pm(r)\ssf=\int_{\,1}^{+\infty}\hspace{-1mm}d\lambda\;
\chi_\infty(\lambda)\,\lambda^{-\tau}\,
(\lambda^2\hspace{-1mm}+\!\tilde\rho^{\ssf2})^{\frac{\tau-\sigma}2}\,
e^{\,i\ssf(t\ssf\pm\ssf r)\ssf\lambda}\,.
\end{equation*}
\smallskip

\noindent
\textit{Case 1}\,:
Assume that \ssf$n\!=\!2\ssf m\!+\!1$ \ssf is odd.
First of all, let us expand
\begin{equation*}\textstyle
\bigl(\frac1{\sinh r}\frac\partial{\partial r}\bigr)^m
=\,{\displaystyle\sum\nolimits_{\ssf\ell=1}^{\,m}}\,
\alpha_{\,\ell}^\infty(r)\,\bigl(\frac\partial{\partial r}\bigr)^\ell\,.
\end{equation*}
Since the coefficients \ssf$\alpha_{\,\ell}^\infty(r)$
are linear combinations of products
\begin{equation*}\textstyle
\bigl(\frac1{\sinh r}\bigr)\times
\bigl(\frac\partial{\partial r}\bigr)^{\ell_2}\bigl(\frac1{\sinh r}\bigr)
\times\,\cdots\,\times
\bigl(\frac\partial{\partial r}\bigr)^{\ell_m}\bigl(\frac1{\sinh r}\bigr)\ssf,
\end{equation*}
with \,$\ell_2\!+{\dots}+\ssb\ell_m\!=\ssb m\!-\!\ell$\,,
and \ssf$\frac1{\sinh r}\ssb
=2\ssf\sum_{\ssf j=0}^{\ssf+\infty}\ssf e^{-(2\ssf j+1)\ssf r}$
\ssf is \ssf$\text{O}\ssf(e^{-r})$\ssf, as well as its deri-
\linebreak
vatives, we deduce that \ssf$\alpha_{\,\ell}^\infty(r)$
\ssf is \ssf$\text{O}\ssf(e^{-m\ssf r})$
\ssf as $r\!\to\!+\infty$\ssf.
Consider next
\begin{equation*}\textstyle
\bigl(\frac\partial{\partial r}\bigr)^\ell g_{\,t}^\pm(r)\ssf
={\displaystyle\int_{\,1}^{+\infty}}\hskip-1mm
d\lambda\;\chi_\infty(\lambda)\,
\lambda^{-\tau}\,
(\lambda^2\!+\!\tilde\rho^{\ssf2})^{\frac{\tau-\sigma}2}\,
(\pm\ssf i\lambda)^{\ssf\ell}\,
e^{\,i\ssf(t\pm r)\ssf\lambda}\,.
\end{equation*}
According to Lemma A.2,
for every $N\hspace{-1mm}\in\!\N^*$\ssb,
there exists \ssf$C_N\hspace{-1mm}\ge\!0$ \ssf such that
\begin{equation*}\textstyle
\bigl|\bigl(\frac\partial{\partial r}\bigr)^\ell g_{\,t}^\pm(r)\bigr|
\le\ssf C_N\,|\sigma|^N\,(\ssf r\ssb\pm\ssb t\ssf)^{-N}\ssf.
\end{equation*}
As a conclusion,
\begin{equation*}\textstyle
|\ssf\widetilde{w}_{\,t}^{\ssf\infty}(r)|\ssf
=C\Big( \frac1{\sinh r}\frac{\partial}{\partial r} \Big)^m (g_t^++g_t^-)(r)
\leq C_N\,r^{-N}\,e^{-\frac{n-1}2\ssf r}\qquad\forall N\in \mathbb N^*\,.
\end{equation*}
\smallskip

\noindent
\textit{Case 2}\,: Assume that \ssf $n\!=\!2\ssf m$ \ssf is even.
According to Case 1\ssf,
for every $N\!\in\!\mathbb N^*$,
there exists \ssf$C_N\!\ge\!0$ \ssf such that
\begin{equation*}\textstyle
\bigl|\bigl(\frac1{\sinh s}\frac\partial{\partial s}\bigr)^mg_{\ssf t}(s)\bigr|\ssf
\leq C_N\,|\sigma|^N\,s^{-N}\,e^{-m\ssf s}
\qquad\forall\;s\!\ge\!3\,.
\end{equation*}
By estimating
\begin{equation*}\begin{aligned}
&\textstyle
\cosh s\ssb-\ssb\cosh r
=2\ssf\sinh\frac{s+r}2\ssf\sinh\frac{s-r}2
\gtrsim e^{\ssf r}\sinh\frac{s-r}2\,,\\
&\textstyle
\sinh s\lesssim e^{\ssf s}\,,
\hspace{2mm}e^{-(m-1)\ssf s}\le e^{-(m-1)\ssf r}\,,
\hspace{2mm}s^{-N}\le r^{-N}\,,
\end{aligned}\end{equation*}
and performing the change of variables \ssf$s\!=\!r\!+\!u$\ssf,
we deduce that 
\begin{equation*}
\begin{aligned}
|\ssf\widetilde{w}_{\,t}^{\ssf\infty}(r)\ssf|\,
&\lesssim\textstyle\,
\frac{e^{\ssf\sigma^2}}{\Gamma(\frac{n+1}2-\sigma)}\,
{\displaystyle\int_{\,r}^{+\infty}}\hspace{-1mm}ds\;
\frac{\sinh s}{\sqrt{\ssf\cosh s\,-\,\cosh r\ssf}}\;
\bigl|\ssf\bigl(\frac1{\sinh s}
\frac\partial{\partial s}\bigr)^mg_{\ssf t}(s)\ssf\bigr|\\
&\le\,C_N\int_{\,r}^{+\infty}\textstyle\hspace{-1mm}ds\;
\frac{\sinh s}{\sqrt{\ssf\cosh s\,-\cosh r\ssf}}\;s^{-N}\,e^{-m\ssf s}\\
&\le\,C_N\;r^{-N}\,e^{-(m-\frac12)\ssf r}
\int_{\,0}^{+\infty}\textstyle\hspace{-1mm}
\frac{du\vphantom{\big|}}{\sqrt{\ssf\sinh\frac u2\ssf}}\,
\le\,C_N\;r^{-N}\,e^{-\frac{n-1}2\ssf r}\,.
\end{aligned}
\end{equation*}
\vspace{-5mm}
\hfill$\square$

\section{Dispersive estimates}\label{Dispersive}

In this section we obtain $L^{q'}\!\to\! L^q$ estimates for the operator
$D^{-\tau}\ssf\tilde{D}^{\ssf\tau-\sigma}\ssf e^{\,i\,t\ssf D}$,
which will be crucial for our Strichartz estimates in next section.
Let us split up its kernel
\ssf$w_t\!=\ssb w_{\ssf t}^{\ssf0}\!+\ssb w_{\,t}^\infty$
as before.
We will handle the contribution of \ssf$w_{\ssf t}^{\ssf0}$,
using the pointwise estimates obtained in Subsection \ref{KernelEstimatewt0}
and the following criterion based on the Kunze-Stein phenomenon.

\begin{lemma}\label{KS}
There exists a constant \,$C\!>\!0$ such that,
for every radial measurable function \,$\kappa$ on \,$\Hn$,
for every \,$2\!\le\!q,\tilde{q}\!<\!\infty$ and \ssf$f\!\in\!L^{q'}\ssb(\Hn)$,
\begin{equation*}
\|\ssf f\ssb*\ssb\kappa\,\|_{L^q\vphantom{L^{q'}}}
\le\,C\;\|f\|_{L^{\tilde{q}'}}\,\Bigl\{\ssf\int_{\,0}^{+\infty}\hspace{-1mm}dr\,
(\sinh r)^{n-1}\,\varphi_0(r)^{\ssf\mu}\,|\kappa(r)|^{\ssf Q}\,\Bigr\}^{\frac1Q}\,.
\end{equation*}
where \,$\mu\ssb
=\ssb\frac{2\ssf\min\ssf\{q,\ssf\tilde{q}\}}{q\ssf+\ssf\tilde{q}}$
and \,$Q\ssb=\ssb\frac{q\ssf\tilde{q}}{q\ssf+\ssf\tilde{q}}$
\ssf{\rm (\footnotemark)}.
\end{lemma}

\footnotetext{\,Notice that
\ssf$\frac1Q\ssb=\ssb\frac1q\ssb+\ssb\frac1{\tilde{q}}$
\ssf and \ssf$\mu\ssb+\ssb Q\ssb>\ssb2$\ssf.}

\begin{proof}
This estimate is obtained by complex multilinear interpolation
between the following version \cite{Her} of the Kunze--Stein phenomenon
\begin{equation*}
\|\ssf f\ssb*\ssb\kappa\,\|_{L^2}
\lesssim\;\|f\|_{L^2}\int_{\,0}^{+\infty}\hspace{-1mm}dr\,
(\sinh r)^{n-1}\,\varphi_0(r)\,|\kappa(r)|
\end{equation*}
and the elementary inequalities
\begin{equation*}
\|\ssf f\ssb*\ssb\kappa\,\|_{L^q\vphantom{L^{q'}}}
\le\,\|f\|_{L^1\vphantom{L^{q'}}}\,\|\kappa\|_{L^q\vphantom{L^{q'}}}\,,
\quad
\|\ssf f\ssb*\ssb\kappa\,\|_{L^\infty\vphantom{L^{q'}}}
\le\,\|f\|_{L^{\tilde{q}'}}\,\|\kappa\|_{L^{\tilde{q}}\vphantom{L^{q'}}}\,.
\end{equation*}
Specifically,
if \ssf$2\!<\!q\!<\!\infty$ \ssf and \ssf$2\!\le\!\tilde{q}\!\le\!q$\ssf,
define \ssf$2\!\le\!r\!\le\!\infty$ \ssf by
\ssf$\frac1r\!=\!\frac12\frac{1/{\tilde{q}}\ssf-1/q}{1/2\ssf-1/q}$\ssf.
By complex interpolation,
let us deduce the intermediate estimate
\begin{equation}\label{intermediate}
\Bigl|\,\int_{\Hn}\!
f\hspace{-.4mm}*\ssb\bigl(\varphi_0^{-\frac2q}g\bigr)(x)\,h(x)\,dx\,\Bigr|\,
\lesssim\,\|f\|_{L^{\tilde{q}'}\vphantom{L^{Q'}}}\ssf
\|g\|_{L^Q\vphantom{L^{Q'}}}\ssf\|h\|_{L^{q'}\vphantom{L^{Q'}}}
\end{equation}
from the endpoint estimates
\begin{equation}\label{endpoint0}
\Bigl|\,\int_{\Hn}\!
f_0\hspace{-.4mm}*\ssb g_{\ssf0}\ssf(x)\,h_{\ssf0}(x)\,dx\,\Bigr|\,
\le\,\|f_0\|_{L^{r'}\vphantom{L^{Q'}}}\ssf
\|g_{\ssf0}\|_{L^r\vphantom{L^{Q'}}}\ssf
\|h_{\ssf0}\|_{L^1\vphantom{L^{Q'}}}
\end{equation}
and
\begin{equation}\label{endpoint1}
\Bigl|\,\int_{\Hn}\!
f_1\hspace{-.4mm}*\ssb\bigl(\varphi_0^{-1}g_{\ssf1}\bigr)(x)\,
h_{\ssf1}(x)\,dx\,\Bigr|\,
\lesssim\,\|f_1\|_{L^2\vphantom{L^{Q'}}}\ssf
\|g_{\ssf1}\|_{L^1\vphantom{L^{Q'}}}\ssf
\|h_{\ssf1}\|_{L^2\vphantom{L^{Q'}}}\,.
\end{equation}
Here
\begin{equation*}
f=\sum_{\text{finite}}\alpha_j\,\1_{A_j},\quad
g=\sum_{\text{finite}}\beta_{\ssf k}\,\1_{B_{\ssf k}},\quad
h=\sum_{\text{finite}}\gamma_{\ssf\ell}\,\1_{\ssf C_\ell}
\end{equation*}
are linear combinations with nonzero complex coefficients
of characteristic functions of disjoints Borel sets in \ssf$\Hn$
with finite positive measure,
the $B_{\ssf k}$'s being moreover spherical.
As in the proof of the Riesz--Thorin theorem
(see for instance \cite[\S\;1.1]{BL}),
we assume that \ssf$\|f\|_{L^{\tilde{q}'}\vphantom{L^{Q'}}}\!
=\ssb\|g\|_{L^Q\vphantom{L^{Q'}}}\!
=\ssb\|h\|_{L^{q'}\vphantom{L^{Q'}}}\!=\ssb1$\ssf,
we consider the analytic families of simple functions
\begin{equation*}
f_z=\sum_{\text{finite}}\alpha_j\,
|\alpha_j|^{\ssf a(z)-1}\ssf\1_{A_j},\quad
g_{\ssf z}=\sum_{\text{finite}}\beta_{\ssf k}\,
|\beta_{\ssf k}|^{\ssf b(z)-1}\ssf\1_{B_k},\quad
h_{\ssf z}=\sum_{\text{finite}}\gamma_{\ssf\ell}\,
|\gamma_{\ssf\ell}|^{\ssf c(z)-1}\ssf\1_{\ssf C_\ell}\ssf,
\end{equation*}
where
\begin{equation*}\textstyle
\frac{a(z)}{\tilde{q}'}\ssb
=\ssb\bigl(\frac1r\!-\!\frac12\bigr)\ssf z\ssb
+\ssb\frac1{r'}\ssf,\quad
\frac{b(z)}Q\ssb
=\ssb\frac1{r'}\ssf z\ssb
+\ssb\frac1r\ssf,\quad
\frac{c(z)}{q'}\ssb=\ssb-\ssf\frac12\ssf z\ssb+\!1\ssf,
\end{equation*}
and we apply the Hadamard three lines theorem
to the holomorphic function
\begin{equation*}\label{psi}
\psi(z)\,=\int_{\Hn}\!
f_z\hspace{-.4mm}*\ssb\bigl(\varphi_0^{-z}g_{\ssf z}\bigr)(x)\,
h_{\ssf z}(x)\,dx
\end{equation*}
in the vertical strip
\ssf$\{\ssf z\!\in\!\mathbb{C}\mid0\!\le\!\Re z\!\le\!1\ssf\}$\ssf.
More precisely,
if \ssf$\Re z\ssb=\ssb0$\ssf,
then
\begin{equation*}\left\{\begin{matrix}
\;\Re a(z)=\frac{\tilde{q}'}{r'}
&\Longrightarrow\vphantom{\Big|}
&\|f_z\|_{L^{r'}\vphantom{L^{Q'}}}^{\ssf r'\vphantom{\tilde{q}'}}\!
=\|f\|_{L^{\tilde{q}'}\vphantom{L^{Q'}}}^{\ssf\tilde{q}'}\!
=1\ssf,\\
\;\Re b(z)=\frac Qr
&\Longrightarrow\vphantom{\Big|}
&\|g_{\ssf z}\|_{L^r\vphantom{L^{Q'}}}^{\ssf r\vphantom{\tilde{q}'}}\!
=\|g\|_{L^Q\vphantom{L^{Q'}}}^{\ssf Q}\!
=1\ssf,\\
\;\Re c(z)=q^{\ssf\prime}
&\Longrightarrow\vphantom{\Big|}
&\|h_{\ssf z}\|_{L^1\vphantom{L^{Q'}}}^{\vphantom{\tilde{q}'}}\!
=\|h\|_{L^{q'}\vphantom{L^{Q'}}}^{\ssf q'}\!
=1\ssf,
\end{matrix}\right.\end{equation*}
hence \ssf$|\psi(z)|\ssb\le\ssb1$\ssf,
according to \eqref{endpoint0}.
Similarly,
if \ssf$\Re z\ssb=\ssb1$\ssf,
then
\begin{equation*}\left\{\begin{matrix}
\;\Re a(z)=\frac{\tilde{q}'}2
&\Longrightarrow\vphantom{\Big|}
&\|f_z\|_{L^2\vphantom{L^{Q'}}}^{\ssf2\vphantom{\tilde{q}'}}\!
=\|f\|_{L^{\tilde{q}'}\vphantom{L^{Q'}}}^{\ssf\tilde{q}'}\!
=1\ssf,\\
\;\Re b(z)=Q
&\Longrightarrow\vphantom{\Big|}
&\|g_{\ssf z}\|_{L^1\vphantom{L^{Q'}}}^{\vphantom{\tilde{q}'}}\!
=\|g\|_{L^Q\vphantom{L^{Q'}}}^{\ssf Q\vphantom{\tilde{q}'}}\!
=1\ssf,\\
\;\Re c(z)=\frac{q'}2
&\Longrightarrow\vphantom{\Big|}
&\|h_{\ssf z}\|_{L^2\vphantom{L^{Q'}}}^{\ssf2\vphantom{\tilde{q}'}}\!
=\|h\|_{L^{q'}\vphantom{L^{Q'}}}^{\ssf q'}\!
=1\ssf,
\end{matrix}\right.\end{equation*}
hence \ssf$|\psi(z)|\ssb\le\ssb C$,
according to \eqref{endpoint1}.
The estimate \eqref{intermediate} is obtained
by applying the three lines theorem to \ssf$\psi(z)$
\ssf at the point \ssf$z\ssb=\ssb\frac2q$\ssf,
where
\begin{equation*}\left\{\begin{matrix}
\;a(z)\ssb=\ssb1
&\Longrightarrow
&f_z\ssb=\ssb f\ssf,\vphantom{\big|}\\
\;b(z)\ssb=\ssb1
&\Longrightarrow
&g_{\ssf z}\ssb=\ssb g\ssf,\vphantom{\big|}\\
\;c(z)\ssb=\ssb1
&\Longrightarrow
&h_{\ssf z}\ssb=\ssb h\ssf.\vphantom{\big|}
\end{matrix}\right.\end{equation*}
Eventually, the symmetric case, where
\ssf$2\!<\!\tilde{q}\!<\!\infty$ \ssf and \ssf$2\!\le\!q\!\le\!\tilde{q}$\ssf,
\ssf is handled similarly.
\end{proof}

For the second part \ssf$w_{\ssf t}^{\ssf\infty}$,
we resume the Euclidean approach,
which consists in interpolating analytically between 
$L^2\!\to\!L^2$ and $L^1\!\to\!L^\infty$ estimates
for the family of operators 
\begin{equation}\label{AnalyticFamily}\textstyle
\widetilde{W}_{\,t,\infty}^{\ssf(\sigma,\tau)}
=\,\frac{e^{\ssf\sigma^2}}{\Gamma(\frac{n+1}2-\sigma)}\;
\chi_\infty(D)\,D^{-\tau}\,\tilde{D}^{\ssf\tau-\sigma}\,e^{\,i\,t\ssf D}
\end{equation}
in the vertical strip \ssf$0\ssb\le\ssb\Re\sigma\ssb\le\!\frac{n+1}2$\ssf.

\subsection{Small time dispersive estimate}
 
\begin{theorem}\label{dispersive0}
Assume that
\,$0\ssb<\ssb|t|\ssb\le\ssb2$\ssf,
\ssf$2\ssb<\ssb q\ssb<\ssb\infty$\ssf,
\ssf$0\ssb\le\ssb\tau\ssb<\ssb\frac32$
and \,$\sigma\ssb\ge\ssb(n\ssb+\!1)\ssf(\frac12\!-\!\frac1q)$\ssf.
Then, 
\begin{equation*}
\bigl\|\ssf D^{-\tau}\ssf\tilde{D}^{\ssf\tau-\sigma}\ssf e^{\,i\,t\ssf D}
\ssf\bigr\|_{L^{q'}\ssb\to L^q}\lesssim\,\begin{cases}
\,|t|^{-(n-1)(\frac12-\frac1q)}
&\text{if \,}n\ssb\ge\ssb3\ssf,\\
\,|t|^{-(\frac12-\frac1q)}\ssf(\ssf1\!-\ssb\log|t|\ssf)^{1-\frac2q}
&\text{if \,}n\ssb=\ssb2\ssf.\\
\end{cases}
\end{equation*}
\end{theorem}

\begin{proof}
We divide the proof into two parts,
corresponding to the kernel decomposition
\ssf$w_t\!=\ssb w_{\ssf t}^{\ssf0}\!+\ssb w_{\,t}^\infty$.
By applying Lemma \ref{KS}
and by using the pointwise estimates in Theorem \ref{Estimatewt0}.i,
we obtain on one hand
\begin{equation*}
\begin{aligned}
\bigl\|\ssf f\ssb*\ssb w_{\ssf t}^{\ssf0}\ssf\bigr\|_{L^q}
&\lesssim\,\Bigl\{\ssf\int_{\,0}^{+\infty}\hspace{-1mm}dr\,
(\sinh r)^{n-1}\,\varphi_0(r)\,|\ssf w_{\ssf t}^{\ssf0}(r)|^{\frac q2}
\,\Bigr\}^{\frac2q}\;\|f\|_{L^{q'}}\\
&\lesssim\,\Big\{\ssf\int_{\,0}^{+\infty}\hspace{-1mm}dr\,
(1\!+\ssb r)^{1+\frac q2}\,e^{-\rho\,r\ssf(\frac q2-1)}
\ssf\Bigr\}^{\frac2q}\;\|f\|_{L^{q'}}\\
&\lesssim\;\|f\|_{L^{q'}}
\qquad\forall\;f\!\in\!L^{q'}.
\vphantom{\int_0^1}
\end{aligned}
\end{equation*}
For the second part,
we consider the analytic family \eqref{AnalyticFamily}.
If \ssf$\Re\sigma\ssb=\ssb0$\ssf, then
\begin{equation*}
\|\ssf f\ssb*\ssb\widetilde{w}_{\,t}^{\ssf\infty}\ssf\|_{L^2}
\lesssim\,\|f\|_{L^2}
\qquad\forall\;f\!\in\!L^2.
\end{equation*}
If \ssf$\Re\sigma\ssb=\ssb\frac{n+1}2$,
we deduce
from the pointwise estimates in Theorem \ref{Estimatewtildetinfty}.i
that
\begin{equation*}
\|\ssf f\ssb*\ssb\widetilde{w}_{\,t}^{\ssf\infty}\ssf\|_{L^\infty}
\lesssim\,|t|^{-\frac{n-1}2}\,\|f\|_{L^1}
\qquad\forall\;f\!\in\!L^1.
\end{equation*}
By interpolation we conclude
for \ssf$\sigma\ssb=\ssb(n+1)\bigl(\frac12\!-\!\frac1q\bigr)$
\ssf that
\begin{equation*}
\bigl\|\ssf f\ssb*\ssb w_{\ssf t}^\infty\ssf\|_{L^q\vphantom{L^{q'}}}
\lesssim\,|t|^{-(n-1)(\frac12-\frac1q)} \|f\|_{L^{q'}}
\qquad\forall\;f\!\in\!L^{q'}.
\end{equation*}
\end{proof}

\subsection{Large time dispersive estimate}
 
\begin{theorem}\label{dispersiveinfty}
Assume that
\,$|t|\ssb\ge\ssb2$\ssf,
\ssf$2\ssb<\ssb q\ssb<\ssb\infty$\ssf,
\ssf$0\ssb\le\ssb\tau\ssb<\ssb\frac32$
and \,$\sigma\ssb\ge\ssb(n\ssb+\!1)\ssf(\frac12\!-\!\frac1q)$\ssf.
Then
\begin{equation*}
\bigl\|\ssf D^{-\tau}\ssf\tilde{D}^{\ssf\tau-\sigma}\ssf e^{\,i\,t\ssf D}
\ssf\bigr\|_{L^{q'}\ssb\to L^q}\lesssim\,|t|^{\ssf\tau-3}\,.
\end{equation*}
\end{theorem}

\begin{proof}
We divide the proof into three parts,
corresponding to the kernel decomposition
\begin{equation*}
w_t=\1_{\ssf B(0,\frac{|t|}2)}\ssf w_{\ssf t}^{\ssf0}
+\1_{\,\Hn\smallsetminus\ssf B(0,\frac{|t|}2)}\ssf w_{\ssf t}^{\ssf0}
+\ssf w_{\,t}^\infty\ssf.
\end{equation*}

\noindent
\textit{Estimate 1}\,:
By applying Lemma \ref{KS}
and using the pointwise estimates in Theorem \ref{Estimatewt0}.ii.a, 
we obtain
\begin{equation*}
\begin{aligned}
\|\ssf f*\{\1_{\ssf B(0,\frac{|t|}2)}\ssf w_{\ssf t}^{\ssf0}\ssf\}\,\|_{L^q}
&\lesssim\,\Bigl\{\ssf\int_{\,0}^{\frac{|t|}2}\!dr\,
(\sinh r)^{n-1}\,\varphi_0(r)\,|\ssf w_{\ssf t}^{\ssf0}(r)|^{\frac q2}
\,\Bigr\}^{\frac2q} \;\|f\|_{L^{q'}}\\
&\lesssim\;\underbrace{\Bigl\{\ssf\int_{\,0}^{+\infty}\hspace{-1mm}dr\,
(1\!+\ssb r)^{1+\frac q2}\,e^{-\rho\,r\ssf(\frac q2-1)}
\ssf\Bigr\}^{\frac2q}}_{<+\infty}\,
|t|^{\ssf\tau-3}\;\|f\|_{L^{q'}}
\qquad\forall\;f\!\in\!L^{q'}.
\end{aligned}
\end{equation*}

\noindent
\textit{Estimate 2}\,:
By applying Lemma \ref{KS}
and using the pointwise estimates in Theorem \ref{Estimatewt0}.ii.b, 
we obtain
\begin{equation*}
\begin{aligned}
\|\ssf f*\{\1_{\ssf\Hn\smallsetminus\ssf B(0,\frac{|t|}2)}\ssf
w_{\ssf t}^{\ssf0}\ssf\}\,\|_{L^q}
&\lesssim\,\Bigl\{\ssf\int_{\,\frac{|t|}2}^{+\infty}\hspace{-1mm}dr\,
(\sinh r)^{n-1}\,\varphi_0(r)\,
|\ssf w_{\ssf t}^{\ssf0}(r)|^{\frac q2}
\,\Bigr\}^{\frac2q}\;\|f\|_{L^{q'}}\\
&\lesssim\,\underbrace{
\Bigl\{\ssf\int_{\,\frac{|t|}2}^{+\infty}\hspace{-1mm}dr\,
r\,e^{-(\frac q2-1)\ssf\rho\ssf r}\,\Bigr\}^{\frac2q}
}_{\lesssim\;|t|^{-\infty}}\,
\|f\|_{L^{q'}}
\qquad\forall\;f\!\in\!L^{q'}.
\end{aligned}
\end{equation*}

\noindent
\textit{Estimate 3}\,:
In order to estimate
the \ssf$L^{q'}\hspace{-1mm}\rightarrow\!L^q$ norm
of \ssf$f\ssb\mapsto\ssb f\ssb*\ssb w_{\,t}^{\ssf\infty}$,
we may apply Lemma \ref{KS}
and use pointwise estimates of \ssf$w_{\,t}^{\ssf\infty}$
(see Remark \ref{Estimatewtinfty}).
While
\begin{equation*}
\int_{\,0}^{\ssf|t|-1}\hspace{-1mm}dr\,
(\sinh r)^{n-1}\,\varphi_0(r)\,|\ssf w_{\,t}^\infty(r)|^{\frac q2}
\quad\text{and}\quad 
\int_{\ssf|t|+1}^{+\infty}\hspace{-1mm}dr\,
(\sinh r)^{n-1}\,\varphi_0(r)\,|\ssf w_{\,t}^\infty(r)|^{\frac q2}
\end{equation*}
are \ssf$\text{O}(\ssf|t|^{-\infty})$ \ssf
for any \ssf$\sigma\!\in\!\R$\ssf,
the integral
\begin{equation*}
\int_{\ssf|t|-1}^{\ssf|t|+1}\hspace{-1mm}dr\,
(\sinh r)^{n-1}\,\varphi_0(r)\,|\ssf w_{\,t}^\infty(r)|^{\frac q2}
\end{equation*}
is finite provided \ssf$\sigma\!>\!\frac{n+1}2\!-\!\frac2q$\ssf,
which is too large compared with the critical exponent
\ssf$(n\!+\!1)(\frac12\!-\!\frac1q)$\ssf.
Instead we use again interpolation
for the analytic family \eqref{AnalyticFamily}.
If \ssf$\Re\sigma\ssb=\ssb0\ssf$, then
\begin{equation*}
\|\ssf f*\widetilde{w}_{\,t}^{\ssf\infty}\ssf\|_{L^2}
\lesssim\,\|f\|_{L^2}
\qquad\forall\;f\!\in\!L^2.
\end{equation*}
If \ssf$\Re\sigma\ssb=\ssb\frac{n+1}2$,
we deduce from Theorem \ref{Estimatewtildetinfty}.ii that
\begin{equation*}
\|\ssf f*\widetilde{w}_{\,t}^{\ssf\infty}\ssf\|_{L^\infty}
\lesssim\,|t|^{-\infty} \,\|f\|_{L^1}
\qquad\forall\;f\!\in\!L^1.
\end{equation*}
By interpolation we conclude for
\ssf$\sigma\ssb=\ssb(n\!+\!1)\bigl(\frac12\!-\!\frac1q\bigr)$
that
\begin{equation*}
\bigl\|\ssf f\ssb*\ssb w_{\ssf t}^\infty\ssf\|_{L^q\vphantom{L^{q'}}}
\lesssim\,|t|^{-\infty} \,\|f\|_{L^{q'}}
\qquad\forall\;f\!\in\!L^{q'}.
\end{equation*}
\end{proof}

By taking \ssf$\tau\!=\!1$
\ssf in Theorems \ref{dispersive0} and \ref{dispersiveinfty},
we obtain in particular the following dispersive estimates.

\begin{corollary}\label{DispersiveGlobal}
Let \,$2\!<\!q\!<\!\infty$
and \,$\sigma\!\ge\!(n\!+\!1)\bigl(\frac12\!-\!\frac1q\bigr)$.
Then
\begin{equation*}\textstyle
\|\,\tilde{D}^{-\sigma+1}\,\frac{e^{\,i\ssf t\ssf D}}D\,\|_{L^{q'}\!\to L^q}
\lesssim\,\begin{cases}
\;|t|^{-(n-1)(\frac12-\frac1q)}
&\text{if \;}0\!<\!|t|\!\le\!2\ssf,\\
\;|t|^{-2}
&\text{if \;}|t|\!\ge\!2\ssf,
\end{cases}
\end{equation*}
with \,$|t|^{-(n-1)(\frac12-\frac1q)}$ replaced by
\,$|t|^{-(\frac12-\frac1q)}\ssf(\ssf1\!-\ssb\log|t|\ssf)^{1-\frac2q}$
in dimension \,$n\!=\!2$\ssf.
\end{corollary}

\begin{remark}
Notice that Tataru \ssf\cite{Ta} obtained
dispersive estimates with exponential decay in time
for the operators \,$\cos t\ssf D$ and \,$\frac{\sin t\ssf D}D$\ssf,
but did not prove actual Strichartz estimates.
Here we obtain dispersive estimates with polynomial decay in time
for the operator \,$e^{\,i\,t\ssf D}$\ssf.
This difference reflects the fact that
the Fourier multipliers associated with the operators
\,$\cos t\ssf D$ and \,$\frac{\sin t\ssf D}D$
are analytic in a strip of the complex plane,
which is not the case of \,$e^{\,i\,t\ssf D}$\ssf.
\end{remark}

By applying Lemma \ref{KS} in full generality,
we obtain the following decoupled estimate for the operators
\begin{equation*}
W_{\ssf t,0}^{\ssf(\sigma,\tau)}
=\,\chi_0(D)\,D^{-\tau}\,\tilde{D}^{\ssf\tau-\sigma}\,e^{\,i\,t\ssf D}\,.
\end{equation*}

\begin{proposition}\label{decoupledwt0}
Let \,$2\!<\!q,\tilde{q}\!<\!\infty$\ssf,
\ssf$0\!\le\!\tau\!<\!\frac32$
\,and \ssf$\sigma\!\in\!\R$\ssf.
Then
\begin{equation*}
\|\ssf W_{\ssf t,0}^{\ssf(\sigma,\tau)}
\|_{L^{\tilde{q}'}\ssb\to L^q}\ssb
\lesssim\ssf(\ssf1\ssb+\ssb|t|\ssf)^{\ssf\tau-3}
\qquad\forall\;t\!\in\!\R\ssf.
\end{equation*}
\end{proposition}

\section{Strichartz estimates}\label{Strichartz}

We shall assume \ssf$n\!\ge\!3$ \ssf throughout this section
and discuss the 2--dimensional case in the final remark.
Consider the inhomogeneous linear wave equation on \ssf$\Hn$\,:
\begin{equation}\label{IP}
\begin{cases}
&\partial_{\ssf t}^{\ssf2}u(t,x)-(\Delta_{\Hn}\!+\ssb\rho^2)\ssf u(t,x)
=F(t,x)\ssf,\\
&u(0,x)=f(x)\ssf,\\
&\partial_{\ssf t}|_{t=0}\,u(t,x)=g(x)\ssf,
\end{cases}
\end{equation}
whose solution is given by Duhamel's formula\,:
\begin{equation*}\textstyle
u(t,x)=(\cos t\ssf D_x)\ssf f\ssf(x)
+\frac{\sin t\ssf D_x}{D_x}\ssf g\ssf(x)
+{\displaystyle\int_{\,0}^{\ssf t}}ds\,
\frac{\sin\ssf(t-s)\ssf D_x}{D_x}\ssf F(s,x)\,.
\end{equation*}

\begin{definition} \label{admcouple}
A couple \,$(p,q)$ is called \,{\rm admissible}
if \,$(\frac1p,\frac1q)$ belongs to the triangle
\begin{equation}\label{triangle}\textstyle
T_n=\bigl\{\ssf\bigl(\frac1p,\frac1q\bigr)\!\in\!
\bigl(0,\frac12\bigr]\!\times\!\bigl(0,\frac12\bigr)
\bigm|\frac2p\ssb+\ssb\frac{n-1}q\ssb\ge\ssb\frac{n-1}2\,\bigr\}
\end{equation}
{\rm(}see Figure 1\,{\rm)}.
\end{definition}

\begin{figure}[ht]
\begin{center}
\psfrag{0}[c]{$0$}
\psfrag{1}[c]{$1$}
\psfrag{1/2}[c]{$\frac12$}
\psfrag{1/2-1/(n-1)}[c]{$\frac12\!-\!\frac1{n-1}$}
\psfrag{1/p}[c]{$\frac1p$}
\psfrag{1/q}[c]{$\frac1q$}
\includegraphics[width=70mm]{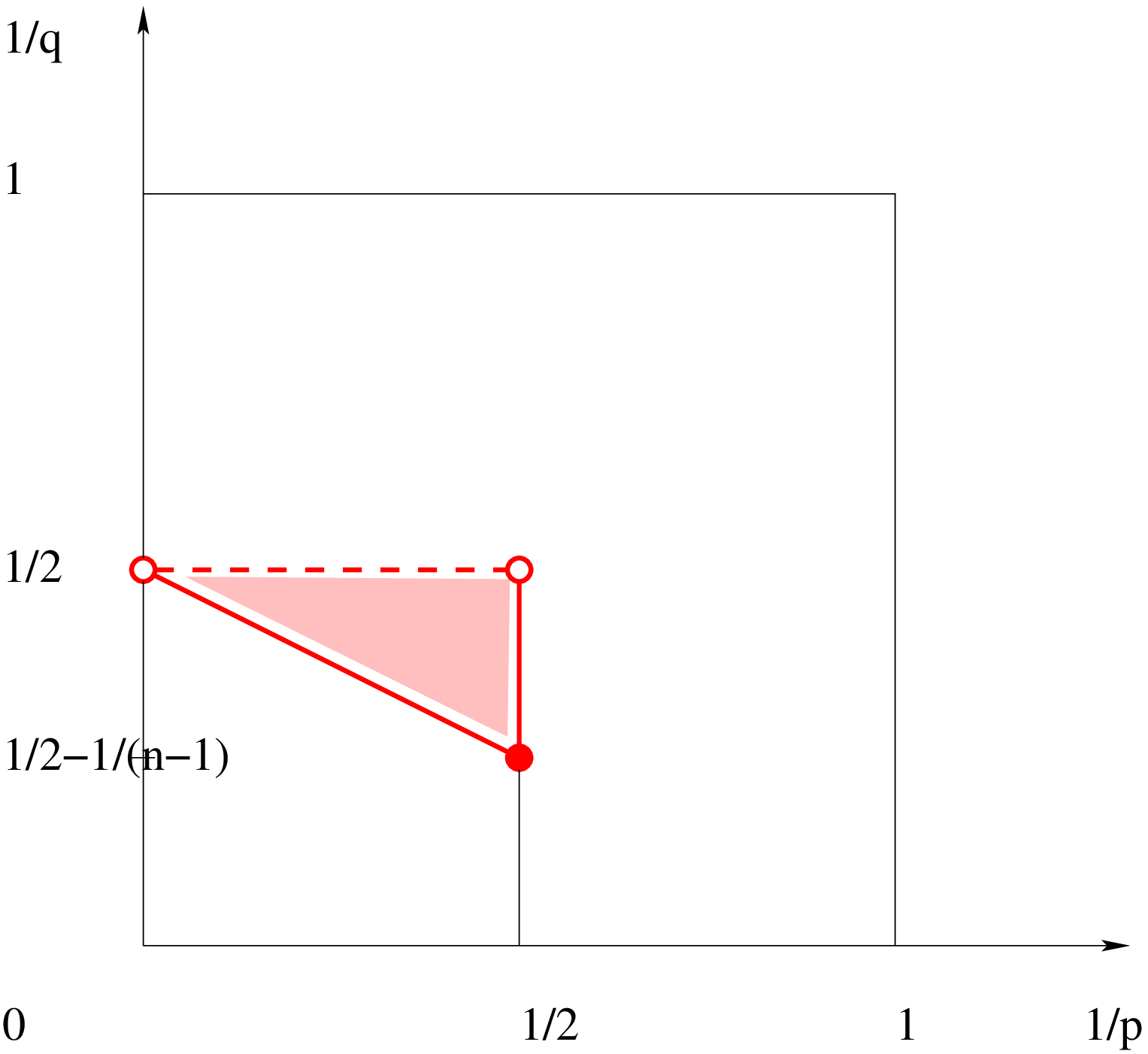}
\end{center}
\caption{Admissibility in dimension \ssf$n\!\ge\!4$}
\end{figure}

\begin{remark}\label{dim3end} 
Observe that the endpoint \,$(\frac12,\frac12\!-\!\frac1{n-1})$
is included in the triangle \,$T_n$ in dimension \,$n\ssb\ge\ssb3$
but not in dimension \,$n\ssb=\ssb3$ \,{\rm(}see Figure 2\,{\rm)}.
\end{remark}

\begin{figure}[ht]
\begin{center}
\psfrag{0}[c]{$0$}
\psfrag{1}[c]{$1$}
\psfrag{1/2}[c]{$\frac12$}
\psfrag{1/p}[c]{$\frac1p$}
\psfrag{1/q}[c]{$\frac1q$}
\includegraphics[width=70mm]{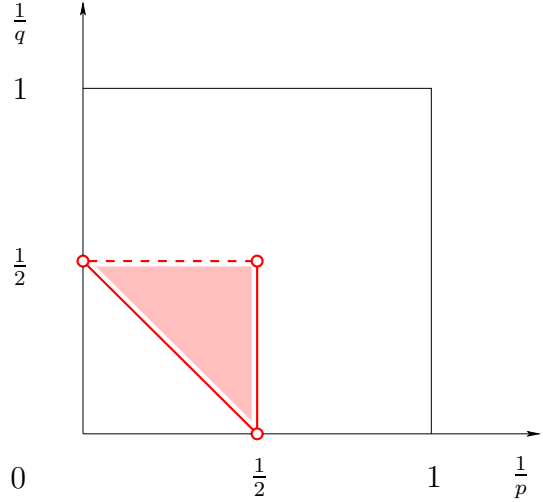}
\end{center}
\caption{Admissibility in dimension \ssf$n\!=\!3$}
\end{figure}

\begin{theorem}\label{StrichartzEstimates}
Let \,$(p,q)$ and \,$(\tilde p, \tilde q)$ be two admissible couples.
Then the following Strichartz estimate holds
for solutions to the Cauchy problem \eqref{IP}\,{\rm :}
\begin{equation}\label{Str1}
\|u\|_{L^p(\R\ssf;\ssf L^q)\vphantom{\big|}}
\lesssim\;\|f\|_{H^{\sigma-\frac12,\frac12}}
+\,\|g\|_{H^{\sigma-\frac12,-\frac12}}
+\,\|F\|_{L^{\tilde{p}'}\ssb\bigl(\R\ssf;\ssf
H_{\tilde{q}'}^{\sigma+\tilde{\sigma}-1}\bigr)}\ssf,
\end{equation}
where \,$\sigma\ssb\ge\ssb\frac{(n+1)}2\ssf\big(\frac12\!-\!\frac1q\big)$
and  \,$\tilde{\sigma}\ssb\ge\ssb
\frac{(n+1)}2\ssf\big(\frac12\!-\!\frac1{\tilde{q}}\big)$.
Moreover,
\begin{equation}\label{Str2}\begin{aligned}
&\|u\|_{L^{\infty}\bigl(\R\ssf;\ssf
H^{\sigma-\frac12,\frac12}\bigr)}
+\|\ssf\partial_{\ssf t}\ssf u\ssf\|_{L^{\infty}\bigl(\R\ssf;\ssf
H^{\sigma-\frac12,-\frac12}\bigr)}\\
&\lesssim\,\|f\|_{H^{\sigma-\frac12,\frac12}\vphantom{\big|}}
+\,\|g\|_{H^{\sigma-\frac12,-\frac12}\vphantom{\big|}}
+\,\|F\|_{L^{\tilde{p}'}\ssb\bigl(\R\ssf;\ssf
H_{\tilde{q}'}^{\sigma+\tilde{\sigma}-1}\bigr)}\,.
\end{aligned}\end{equation}
\end{theorem}
\begin{proof}
Consider the operator
\begin{equation*}\textstyle
T\ssb f\ssf(t,x)=\tilde{D}_{\ssf x}^{-\sigma+\frac12}\,
\frac{e^{\ssf\pm\ssf i\,t\ssf D_x}}{\sqrt{D_x}}\ssf f\ssf(x)\ssf,
\end{equation*}
initially defined from $L^2(\Hn)$ into
$L^\infty(\R\ssf;\ssb H^{-\frac12,\ssf\frac12}(\Hn)\ssb)$,
and its formal adjoint
\begin{equation*}
T^*\ssb F\ssf(x)=\ssb\int_{-\infty}^{+\infty}\hspace{-1mm}ds\;
\textstyle\tilde{D}_{\ssf x}^{-\sigma+1/2}\,
\frac{e^{\ssf\mp\ssf i\,s\ssf D_x}}{\sqrt{D_x}}
\,F\ssf (s,x)\ssf,
\end{equation*}
initially defined from $L^1\ssb(\R\ssf;\ssb L^2\ssb(\Hn)\ssb)$
into $H^{-\frac12,\ssf\frac12}(\Hn)$.
The \ssf$TT^*$ \ssb method consists in proving first the
\ssf$L^{p'}\ssb(\R\ssf;\ssb L^{q'}\ssb(\Hn)\ssb)\ssb
\to\ssb L^p(\R\ssf;\ssb L^q(\Hn)\ssb)$
\ssf boundedness of the operator
\begin{equation*}
TT^*\ssb F\ssf(t,x)
=\ssb\int_{-\infty}^{+\infty}\hspace{-1mm}ds\;\textstyle
\tilde{D}_{\ssf x}^{-2\ssf\sigma+1}\,\frac{e^{\ssf\pm\ssf i\ssf(t-s)\ssf D_x}}{D_x}
\,F\ssf(s,x)
\end{equation*}
and of its truncated version
\begin{equation*}
\mathcal{T}\ssb F\ssf(t,x)
=\ssb\int_{-\infty}^{\,t}\hspace{-1mm}ds\;\textstyle
\tilde{D}_{\ssf x}^{-2\ssf\sigma+1}\,\frac{e^{\ssf\pm\ssf i\ssf(t-s)\ssf D_x}}{D_x}
\,F\ssf(s,x)\ssf,
\end{equation*}
for every admissible couple $(p,q)$
and for every $\sigma\!\ge\!\frac{n+1}2\bigl(\frac12\!-\!\frac1q\bigr)$,
and in decoupling next the indices.

Assume that the admissible couple \ssf$(p,q)$ \ssf
is different from the endpoint \ssf$(2,2\ssf\frac{n-1}{n-3})$\ssf.
Then we deduce from Corollary \ref{DispersiveGlobal} that
the norms \,$\|\ssf T\ssf T^*\ssb F(t,x)\ssf\|_{L_t^pL_x^q}$
and \,$\|\ssf\mathcal{T}\ssb F(t,x)\ssf\|_{L_t^pL_x^q}$
are bounded above by
\begin{equation}\label{HLS}
\Bigl\|\,\int_{\ssf0<|t-s|<1}\hspace{-1mm}ds\;
|\ssf t\ssb-\ssb s\ssf |^{-\alpha}\,
\|\ssf F(s,x)\ssf\|_{L_x^{q'}\vphantom{\big|}}\,\Bigr\|_{L_t^p}
+\;\Bigl\|\,\int_{\ssf|t-s|\ge1}\hspace{-1mm}ds\;
|\ssf t\ssb-\ssb s\ssf |^{-2}\,
\|\ssf F(s,x)\ssf\|_{L_x^{q'}\vphantom{\big|}}\,\Bigr\|_{L_t^p}\,,
\end{equation}
where $\alpha\!=\!(n\!-\!1)(\frac12\!-\!\frac1q)
\hspace{-.75mm}\in\hspace{-.75mm}(0,1)$.
On the one hand, the convolution kernel
\ssf$|\ssf t\!-\!s\ssf|^{-2}\ssf{\1}_{\ssf\{|\ssf t-s\ssf|\ge1\}}$
defines obviously
a bounded operator from \ssf $L^{p_1}\ssb(\R)$ to \ssf$L^{p_2}(\R)$,
for all \ssf$1\ssb\le\ssb p_1\!\le\ssb p_2\!\le\!\infty$\ssf,
in particular from \ssf$L^{p'}\ssb(\R)$ to \ssf$L^p(\R)$,
since \ssf$p\ssb\ge\ssb2$\ssf.
On the other hand, the convolution kernel
\,$|\ssf t\!-\!s\ssf|^{-\alpha}\,{\1}_{\,\{\ssf0<|\ssf t-s\ssf|<1\ssf\}}$
\ssf with \ssf$0\!<\!\alpha\!<\!1$
\ssf defines a bounded operator
from \ssf$L^{p_1}\ssb(\R)$ to \ssf$L^{p_2}\ssb(\R)$,
for all \ssf$1\!<\!p_1,p_2\!<\!\infty$ \ssf such that
\ssf$0\ssb\le\!\frac1{p_1}\!-\!\frac1{p_2}\!\le\!1\!-\!\alpha\ssf$,
\ssf in particular from \ssf$L^{p'}\ssb(\R)$ to \ssf$L^p(\R)$,
since \ssf$p\ssb\ge\ssb2$ \ssf and \ssf$\frac2p\!\ge\ssb\alpha$\ssf.

At the endpoint \ssf$(p,q)\!=\!(2,2\ssf\frac{n-1}{n-3})$\ssf,
we have \ssf$\alpha\ssb=\!1$\ssf.
Thus the previous argument breaks down
and is replaced by the refined analysis carried out in \cite{KT}.
Notice that the problem lies only in the first part of \eqref{HLS}
and not in the second one,
which involves an integrable convolution kernel on \ssf$\R$\ssf.

Thus \,$TT^*$ and \,$\mathcal{T}$ are bounded
from \ssf$L^{p'}\ssb(\R\ssf;\ssb L^{q'}\ssb(\Hn)\ssb)$
to \ssf$L^p(\R\ssf;\ssb L^q(\Hn)\ssb)$\ssf,
for every admissible couple $(p,q)$\ssf.
As a consequence, \,$T^*$ is bounded
from \ssf$L^{p'}\ssb(\R\ssf;\ssb L^{q'}\ssb(\Hn)\ssb)$ \ssf to \ssf$L^2(\Hn)$
\ssf and \,$T$ \ssf is bounded
from \ssf$L^2(\Hn)$ \ssf to \ssf$L^p(\R\ssf;\ssb L^q(\Hn)\ssb)$.
In particular,
\begin{equation*}\textstyle
\|\ssf(\cos t\ssf D_x)f(x)\ssf\|_{L_t^pL_x^q}
\lesssim\,\|\ssf\tilde{D}_{\,x}^{-\sigma+\frac12}D_{\,x}^{-\frac12}\ssf
e^{\ssf\pm\ssf i\,t\ssf D_x}\tilde{D}_{\,x}^{\ssf\sigma-\frac12}
D_{\ssf x}^{\frac12}\ssf f\ssf(x)\ssf\|_{L_t^pL_x^q}
\lesssim\,\|f\|_{H^{\sigma-\frac12,\frac12}}
\end{equation*}
and
\begin{equation*}\textstyle
\|\ssf\frac{\sin t\ssf D_x}{D_x}\ssf g\ssf(x)\ssf\|_{L_t^pL_x^q}
\lesssim\,\|\ssf\tilde{D}_{\,x}^{-\sigma+\frac12}D_{\,x}^{-\frac12}\ssf
e^{\ssf\pm\ssf i\,t\ssf D_x}\tilde{D}_{\,x}^{\ssf\sigma-\frac12}
D_{\,x}^{-\frac12}\ssf g\ssf(x)\ssf\|_{L_t^pL_x^q}
\lesssim\,\|g\|_{H^{\sigma-\frac12,-\frac12}}\,.
\end{equation*}

We next decouple the indices.
Let $(p,q)\ssb\ne\ssb(\tilde{p},\tilde{q})$ be two admissible couples
and let $\sigma\!\ge\!\frac{n+1}2\bigl(\frac12\!-\!\frac1q\bigr)$,
$\tilde{\sigma}\!\ge\!\frac{n+1}2\bigl(\frac12\!-\!\frac1{\tilde{q}}\bigr)$.
Since \,$T$ and \,$T^*$ are separately continuous,
the operator
\begin{equation*}
TT^*\ssb F(t,x)\ssf=\int_{-\infty}^{+\infty}\hspace{-1mm}ds\;\textstyle
\tilde{D}_{\,x}^{-\sigma-\tilde{\sigma}+1}\,
\frac{e^{\ssf\pm\ssf i\ssf(t-s)\ssf D_x}}{D_x}\,F(s,x)
\end{equation*}
is bounded from $L^{\tilde{p}'}\ssb(\R\ssf;\ssb L^{\tilde{q}'}\ssb(\Hn)\ssb)$
to $L^p(\R\ssf;\ssb L^q(\Hn)\ssb)$\ssf.
According to \cite{CK},
this result remains true for the truncated operator
\begin{equation*}
\mathcal{T}\ssb F(t,x)\,
=\int_{-\infty}^{\,t}\hspace{-1mm}ds\;\textstyle
\tilde{D}_{\,x}^{-\sigma-\tilde{\sigma}+1}\,
\frac{e^{\ssf\pm\ssf i\ssf(t-s)\ssf D_x}}{D_x}\,F(s,x)
\end{equation*}
and hence for
\begin{equation*}
\widetilde{\mathcal{T}}\ssb F(t,x)\,
=\int_{\,0}^{\ssf t}ds\;\textstyle
\tilde{D}_{\,x}^{-\sigma-\tilde{\sigma}+1}\,
\frac{\sin\ssf(t-s)\ssf D_x}{D_x}\,F(s,x)
\end{equation*}
as long as \ssf$p$ \ssf and \ssf$\tilde{p}$
\ssf are not both equal to \ssf$2$\ssf.
We handle the remaining case,
where \ssf$p\ssb=\ssb\tilde{p}\ssb=\ssb2$ \ssf and
\ssf$2\ssb<\ssb q\ssb\ne\ssb\tilde{q}\ssb\le\ssb2\ssf\frac{n-1}{n-3}$\ssf,
by combining the bilinear approach in \cite{KT}
with our previous estimates.
Specifically let us split up again
\ssf$I\!=\ssb\chi_0(D)\ssb+\ssb\chi_\infty(D)^2$,
using smooth cut--off functions, and 
\ssf$\mathcal{T}\hspace{-1mm}
=\ssb\mathcal{T}^{\ssf0}\!+\ssb\mathcal{T}^\infty$
accordingly.
On one hand, it follows from Proposition \ref{decoupledwt0} that
\begin{equation*}
\mathcal{T}^{\ssf0}\hspace{-.6mm}F(t,x)\,
=\int_{-\infty}^{\,t}\hspace{-1mm}ds\;
\chi_0(D_x)\,D_{\ssf x}^{-1}\,\tilde{D}_{\,x}^{\ssf1-\sigma-\tilde{\sigma}}\,
e^{\ssf\pm\ssf i\ssf(t-s)\ssf D_x}\,F(s,x)
\end{equation*}
is bounded from
\ssf$L^{\tilde{p}'}\ssb(\R\ssf;\ssb L^{\tilde{q}'}\ssb(\Hn)\ssb)$
\ssf to \ssf$L^p(\R\ssf;\ssb L^q(\Hn)\ssb)$\ssf,
\ssf for every \ssf$2\ssb\le\ssb p,\tilde{p}\ssb\le\ssb\infty$
\ssf and \ssf$2\ssb<\ssb q,\tilde{q}\ssb<\ssb\infty$\ssf,
in particular for \ssf$p\ssb=\ssb\tilde{p}\ssb=\ssb2$
\ssf and \ssf$2\ssb<\ssb q,\tilde{q}\ssb\le\ssb2\ssf\frac{n-1}{n-3}$\ssf.
As far as it is concerned,
the $L^2L^{\tilde{q}'}\!\to L^2L^q$ boundedness of
\begin{equation*}
\mathcal{T}^\infty\hspace{-.5mm}F(t,x)\,
=\int_{-\infty}^{\,t}\hspace{-1mm}ds\;
\chi_\infty(D_x)^2\,D_{\ssf x}^{-1}\,
\tilde{D}_{\,x}^{\ssf1-\sigma-\tilde{\sigma}}\,
e^{\ssf\pm\ssf i\ssf(t-s)\ssf D_x}\,F(s,x)
\end{equation*}
amounts to estimating the hermitian form
\begin{equation*}\begin{aligned}
\mathcal{B}^{\ssf\infty\ssb}(F,G)\ssf
=\iint_{\ssf s<t}\!ds\,dt\ssf\int_{\ssf\Hn}\!dx\;
&\chi_\infty(D_x)\,D_{\,x}^{-1/2}\ssf
\tilde{D}_{\,x}^{\ssf1/2-\tilde{\sigma}}\ssf
e^{\ssf\mp\ssf i\ssf s\ssf D_x}\ssf F(s,x)\\
\times\,
&\overline{\chi_\infty(D_x)\,D_{\,x}^{-1/2}\ssf
\tilde{D}_{\,x}^{\ssf1/2-\sigma}\ssf
e^{\ssf\mp\ssf i\ssf t\ssf D_x}\,G(t,x)}
\end{aligned}\end{equation*}
by \ssf$\|F\|_{L^2L^{\tilde{q}'}}\|G\|_{L^2L^{q'}}$.
Let us split up dyadically
\begin{equation*}
\iint_{\ssf s<t}
=\sum\limits_{j=-\infty}^{+\infty}\,
\iint_{\,2^{\ssf j}\le\ssf t-s\ssf<\ssf2^{\ssf j+1}}
\end{equation*}
and \,$\mathcal{B}^{\ssf\infty}\!
=\ssb\sum_{\ssf j=-\infty}^{\,+\infty}\mathcal{B}_{\,j}^{\ssf\infty}$
\ssf accordingly.
For every \ssf$j\!\in\!\Z$\ssf,
let us further split up
\begin{equation*}
F(s,x)=\!\sum_{k=-\infty}^{+\infty}\underbrace{
\1_{\ssf[\ssf k\ssf2^{\ssf j},\ssf(k+1)\ssf2^j)}(s)\,F(s,x)
}_{F_{\,k}^{\ssf(j)}(s,x)}
\quad\text{and}\quad
G(t,x)=\!\sum_{\ell=-\infty}^{+\infty}
\underbrace{
\1_{\ssf[\ssf\ell\ssf2^{\ssf j},\ssf(\ell+1)\ssf2^j)}(t)\,G(t,x)
}_{G_{\,\ell}^{\ssf(j)}(t,x)}\,.
\end{equation*}
Notice the orthogonality
\begin{equation*}
\|\ssf F\ssf\|_{L^2L^{\tilde{q}'}}\!
=\Bigl\{\,\sum\nolimits_{\ssf k=-\infty}^{\ssf+\infty}\ssf
\|\ssf F_{\,k}^{\ssf(j)}\ssf\|_{L^2L^{\tilde{q}'}}^{\ssf2}
\Bigr\}^{1/2}\,,\quad
\|\ssf G\ssf\|_{L^2L^{q'}}\!
=\Bigl\{\,\sum\nolimits_{\ssf\ell=-\infty}^{\ssf+\infty}\ssf
\|\,G_{\,\ell}^{\ssf(j)}\ssf\|_{L^2L^{q'}}^{\ssf2}
\Bigr\}^{1/2}
\end{equation*}
and the almost orthogonality
\begin{equation*}
\mathcal{B}_{\,j}^{\ssf\infty\ssb}(F,G\ssf)\,=\,
\sum\nolimits_{\substack{k,\ssf\ell\ssf\in\ssf\Z\\\ell-k\ssf\in\ssf\{1,2\}}}
\mathcal{B}_{\,j}^{\ssf\infty\ssb}(F_{\,k}^{\ssf(j)}\ssb,G_{\,\ell}^{\ssf(j)})\,.
\end{equation*}
We claim that
\begin{equation}\label{HermitianEstimate}
|\,\mathcal{B}_{\,j}^{\ssf\infty\ssb}(F_{\,k}^{\ssf(j)}\ssb,G_{\,\ell}^{\ssf(j)})\ssf|
\,\lesssim\,\begin{cases}
\;2^{\,\kappa(q,\tilde{q})\ssf j}\,
\|\ssf F_{\,k}^{\ssf(j)}\ssf\|_{L^2L^{\tilde{q}'}}
\|\,G_{\,\ell}^{\ssf(j)}\ssf\|_{L^2L^{q'}}
&\text{if \;}j\ssb\le\ssb0\ssf,\\
\;2^{-\infty\ssf j}\,
\|\ssf F_{\,k}^{\ssf(j)}\ssf\|_{L^2L^{\tilde{q}'}}
\|\,G_{\,\ell}^{\ssf(j)}\ssf\|_{L^2L^{q'}}
&\text{if \;}j\ssb>\ssb0\ssf,
\end{cases}\end{equation}
when \ssf$2\ssb<\ssb q,\tilde{q}\ssb\le\ssb2\,\frac{n-1}{n-3}$
\ssf and \ssf$\kappa(q,\tilde{q})\ssb
=\ssb\frac{n-1}2\ssf(\frac1q\!+\!\frac1{\tilde{q}})\ssb-\ssb\frac{n-3}2$\ssf.
These estimates will be obtained
by complex interpolation between the following cases
(see Figure 3)\,:
\begin{itemize}
\item[(a)]
\ssf$q\ssb=\ssb2$ \ssf and
\ssf$2\ssb\le\ssb\tilde{q}\ssb\le\ssb2\,\frac{n-1}{n-3}$\ssf,
\item[(b)]
\ssf$2\ssb\le\ssb q\ssb\le\ssb2\,\frac{n-1}{n-3}$
\ssf and \ssf$\tilde{q}\ssb=\ssb2$\ssf,
\item[(c)]
\ssf$2\ssb<\ssb q\ssb=\ssb\tilde{q}\ssb<\ssb\infty$\ssf.
\end{itemize}

\begin{figure}[ht]\label{Interpolation}
\begin{center}
\psfrag{0}[c]{$0$}
\psfrag{1/2}[c]{$\frac12$}
\psfrag{1/2-1/(n-1)}[c]{$\frac12\!-\!\frac1{n-1}$}
\psfrag{1/q}[c]{$\frac1q$}
\psfrag{1/tildeq}[c]{$\frac1{\tilde{q}}$}
\includegraphics[width=75mm]{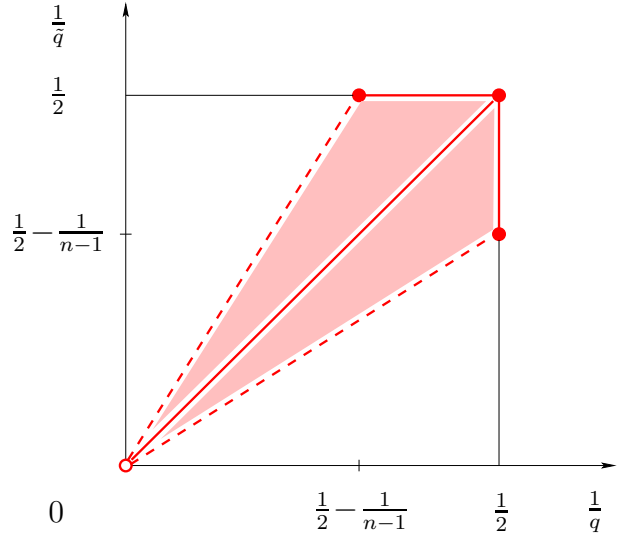}
\end{center}
\caption{Interpolation}
\end{figure}

\noindent
\textit{Case }(\textit{a}\ssf)\,:
Assume that \ssf$q\ssb=\ssb2$\ssf,
$2\ssb\le\ssb\tilde{q}\ssb\le\ssb2\,\frac{n-1}{n-3}$
\ssf and
\ssf$\Re\sigma\ssb=\ssb0$\ssf,
\ssf$\Re\tilde{\sigma}\ssb=\ssb\frac{n+1}2\ssf(\frac12\!-\!\frac1q)$\ssf.
Consider the operators
\begin{equation*}\textstyle
T^{\ssf\infty\ssb}f(t,x)
=\chi_\infty(D_x)\,\tilde{D}_{\,x}^{-\tilde{\sigma}+\frac12}\,
\frac{e^{\ssf\pm\ssf i\,t\ssf D_x}}{\sqrt{D_x}}\ssf f\ssf(x)
\end{equation*}
and
\begin{equation*}
(T^{\ssf\infty})^*F\ssf(x)
=\ssb\int_{-\infty}^{+\infty}\hspace{-1mm}ds\;\textstyle
\chi_\infty(D_x)\,\tilde{D}_{\,x}^{-\tilde{\sigma}+\frac12}\,
\frac{e^{\ssf\mp\ssf i\,s\ssf D_x}}{\sqrt{D_x}}\,F(s,x)\,.
\end{equation*}
By resuming the proof of Theorem \ref{dispersive0}
and by applying the \ssf$T^{\ssf\infty\ssb}(T^{\ssf\infty})^*$ argument,
we obtain that \ssf$(T^{\ssf\infty})^*$ is bounded
from \ssf$L^{\tilde{p}'}\!(\R\ssf;\ssb L^{\tilde{q}'}\!(\Hn)\ssb)$
\ssf to \ssf$L^2(\Hn)$\ssf,
where \ssf$\frac1{\tilde{p}}\!=\!\frac{n-1}2\ssf(\frac12\!-\!\frac1{\tilde{q}})$\ssf.
By combining this result with H\"older's inequality, we deduce that
\begin{equation*}\begin{aligned}
|\,\mathcal{B}_{\,j}^{\ssf\infty\ssb}
(F_{\,k}^{\ssf(j)}\ssb,G_{\,\ell}^{\ssf(j)})\ssf|\,
&\lesssim\,\sup_{t\in\R}\;\Bigl\|\,
\int_{\ssf t-2^{\ssf j+1}<\ssf s\ssf\le\ssf t-2^{\ssf j}}\hspace{-1mm}ds\;
\chi_\infty(D_x)\,D_{\,x}^{-\frac12}\,\tilde{D}_{\,x}^{\frac12-\tilde{\sigma}}\,
e^{\ssf\mp\ssf i\ssf s\ssf D_x}\ssf F_{\,k}^{\ssf(j)}(s,x)\,\Bigr\|_{L_x^2}\\
&\hspace{1.25mm}\times\;\bigr\|\,\chi_\infty(D_x)\,
D_{\,x}^{-\frac12}\,\tilde{D}_{\,x}^{\ssf\frac12}\,
e^{\ssf\mp\ssf i\ssf t\ssf D_x}\,
G_{\,\ell}^{\ssf(j)}(t,x)\ssf\bigr\|_{L_t^1L_x^2}\\
&\lesssim\,\sup_{t\in\R}\;\bigl\|\,
\1_{\ssf(\ssf t-2^{\ssf j+1}\ssb,\ssf t-2^{\ssf j})}(s)\,
F_{\,k}^{\ssf(j)}(s,x)\,\bigr\|_{L_s^{\tilde{p}'}\!L_x^{\tilde{q}'}}\,
\bigl\|\,G_{\,\ell}^{\ssf(j)}(t,x)\,\bigr\|_{L_t^1\!L_x^2}\\
&\lesssim\,2^{\frac j{\tilde{p}'}}\,
\|\,F_{\,k}^{\ssf(j)}\ssf\|_{L^2L^{\tilde{q}'}}\ssf
\|\,G_{\,\ell}^{\ssf(j)}\ssf\|_{L^2L^2}\,,
\end{aligned}\end{equation*}
with \ssf$\frac1{\tilde{p}'}\!=\ssb\kappa(2,\tilde{q})$\ssf.
\smallskip

\noindent
\textit{Case }(\textit{b}\ssf)\,:
If \ssf$2\ssb<\ssb q\ssb\le\ssb2\,\frac{n-1}{n-3}$\ssf,
$\tilde{q}\ssb=\ssb2$ \ssf and
\ssf$\Re\sigma\ssb=\ssb\frac{n+1}2\ssf(\frac12\ssb-\ssb\frac1q)$\ssf,
$\Re\tilde{\sigma}\ssb=\ssb0$\ssf,
we have symmetrically
\begin{equation*}
|\,\mathcal{B}_{\,j}^{\ssf\infty\ssb}
(F_{\,k}^{\ssf(j)}\ssb,G_{\,\ell}^{\ssf(j)})\ssf|\ssf
\lesssim\ssf2^{\ssf\kappa(q,2)\ssf j}\,
\|\ssf F_{\,k}^{\ssf(j)}\ssf\|_{L^2L^2}\,
\|\ssf G_{\,\ell}^{\ssf(j)}\ssf\|_{L^2L^{q'}}\ssf.
\end{equation*}

\noindent
\textit{Case }(\textit{c}\ssf)\,:
Assume that \ssf$2\ssb<\ssb q\ssb=\ssb\tilde{q}\ssb<\ssb\infty$
\ssf and \ssf$\Re\sigma\ssb=\ssb\Re\tilde{\sigma}\ssb
=\ssb\frac{n+1}2\ssf(\frac12\!-\!\frac1q)$\ssf.
Let us rewrite
\begin{equation*}\begin{aligned}
\mathcal{B}_{\,j}^{\ssf\infty\ssb}(F_{\,k}^{\ssf(j)}\ssb,G_{\,\ell}^{\ssf(j)})\,
&=\iint_{\,2^{\ssf j}\le\ssf t-s\ssf<\ssf2^{\ssf j+1}}\hspace{-1mm}ds\,dt
\ssf\int_{\ssf\Hn}\hspace{-1mm}dx\\
&\ssf\times\,\bigl\{\,\chi_\infty(D_x)^2\ssf D_{\,x}^{-1}\ssf
\tilde{D}_{\,x}^{\ssf1-\sigma-\tilde{\sigma}}\ssf
e^{\ssf\pm\,i\ssf(t-s)\ssf D_x}\ssf F_{\,k}^{\ssf(j)}(s,x)\ssf\bigr\}\;
\overline{G_{\,\ell}^{\ssf(j)}(t,x)}\,.
\end{aligned}\end{equation*}
By using the dispersive estimates
\begin{equation*}
\bigl\|\,\chi_\infty(D)^2\,D^{-1}\,\tilde{D}^{\ssf1-\sigma-\tilde{\sigma}}\,
e^{\,\pm\,i\ssf(t-s)\ssf D}\,\bigr\|_{L^{\tilde{q}'}\!\to L^{\tilde{q}}}
\lesssim\,\begin{cases}
\,(t\!-\!s)^{-(n-1)(\frac12-\frac1{\tilde{q}})}
&\text{if \;}0\ssb<\ssb t\!-\!s\ssb<\ssb2\\
\,(t\!-\!s)^{-\infty}
&\text{if \;}t\!-\!s\ssb\ge\ssb2\\
\end{cases}\end{equation*}
(see the proofs of Theorems \ref{dispersive0} and \ref{dispersiveinfty}),
we obtain
\begin{equation*}
|\,\mathcal{B}_{\,j}^{\ssf\infty\ssb}(F_{\,k}^{\ssf(j)}\ssb,
G_{\,\ell}^{\ssf(j)})\ssf|\,
\lesssim\,\begin{cases}
\;2^{-(n-1)(\frac12-\frac1q)\ssf j}\,
\|\ssf F_{\,k}^{\ssf(j)}\ssf\|_{L^1L^{q'}}\,
\|\,G_{\,\ell}^{\ssf(j)}\ssf\|_{L^1L^{q'}}
&\text{if \;}j\ssb\le\ssb0\ssf,\\
\;2^{-\infty\ssf j}\,
\|\ssf F_{\,k}^{\ssf(j)}\ssf\|_{L^1L^{q'}}\,
\|\,G_{\,\ell}^{\ssf(j)}\ssf\|_{L^1L^{q'}}
&\text{if \;}j\ssb>\ssb0\ssf.\\
\end{cases}\end{equation*}
Hence, by H\"older's inequality,
\begin{equation*}
|\,\mathcal{B}_{\,j}^{\ssf\infty\ssb}(F_{\,k}^{\ssf(j)}\ssb,
G_{\,\ell}^{\ssf(j)})\ssf|\,
\lesssim\,\begin{cases}
\;2^{\,\kappa(q,q)\ssf j}\,
\|\ssf F_{\,k}^{\ssf(j)}\ssf\|_{L^2L^{q'}}\,
\|\,G_{\,\ell}^{\ssf(j)}\ssf\|_{L^2L^{q'}}
&\text{if \;}j\ssb\le\ssb0\ssf,\\
\;2^{-\infty\ssf j}\,
\|\ssf F_{\,k}^{\ssf(j)}\ssf\|_{L^2L^{q'}}\,
\|\,G_{\,\ell}^{\ssf(j)}\ssf\|_{L^2L^{q'}}
&\text{if \;}j\ssb>\ssb0\ssf.\\
\end{cases}\end{equation*}
Our claim \eqref{HermitianEstimate} follows now
by complex interpolation between the estimates
obtained in Cases (a), (b) and (c) above.
By summing up \eqref{HermitianEstimate}
and by using H\"older's inequality,
we conclude that
\vspace{3mm}

$\displaystyle
|\,\mathcal{B}^{\ssf\infty\ssb}(F,G\ssf)\ssf|\,
\le\,\sum\nolimits_{\ssf j\in\Z}
|\,\mathcal{B}_{\,j}^{\ssf\infty\ssb}(F,G\ssf)\ssf|\,
\le\,\sum\nolimits_{\substack
{j,\ssf k,\ssf\ell\ssf\in\ssf\Z\\\ell-k\ssf\in\ssf\{1,2\}}}\!
|\,\mathcal{B}_{\,j}^{\ssf\infty\ssb}(F_{\,k}^{\ssf(j)}\ssb,
G_{\,\ell}^{\ssf(j)})\ssf|$
\vspace{-1mm}

$\displaystyle\lesssim\,
\Bigl\{\,\sum\nolimits_{\ssf j\le0}2^{\ssf\kappa(q,\tilde{q})\ssf j}
+\sum\nolimits_{\ssf j>0}2^{-\infty\ssf j}\,\Bigr\}\,
\Bigl\{\,\sum\nolimits_{\ssf k\in\Z}
\|\,F_{\,k}^{\ssf(j)}\ssf\|_{L^2L^{\tilde{q}'}}^{\,2}\Bigr\}^{1/2}\,
\Bigl\{\,\sum\nolimits_{\ssf\ell\in\Z}
\|\,G_{\,\ell}^{\ssf(j)}\ssf\|_{L^2L^{q'}}^{\,2}\Bigr\}^{1/2}$
\vspace{3mm}

$\displaystyle\lesssim\;
\|\ssf F\ssf\|_{L^2L^{\tilde{q}'}}\,\|\ssf G\ssf\|_{L^2L^{q'}}$
\vspace{3mm}

\noindent
if \ssf$2\ssb<\ssb q\ssb\ne\ssb\tilde{q}\ssb\le\ssb2\,\frac{n-1}{n-3}$\ssf.
Notice that \ssf$\kappa(q,\tilde{q})\!>\!0$ \ssf under this assumption.

Let us turn to  \eqref{Str2}.
On the one hand,
the energy estimate (\ref{conservationenergy}) yields
\begin{equation*}\begin{aligned}
&\textstyle
\big\|\,(\cos t\ssf D)f\ssb+\ssb\frac{\sin t\ssf D}D\ssf g
\,\bigr\|_{H^{\sigma-\frac12,\frac12}}
+\,\|-\ssb(\sin t\ssf D)\ssf D\ssb f\ssb+\ssb(\cos t\ssf D)\ssf g
\,\|_{H^{\sigma-\frac12,-\frac12}}\\
&\le\ssf\sqrt{2\ssf}\;\bigl\{\,\|f\|_{H^{\sigma-\frac12,\frac12}}
+ \,\|g\|_{H^{\sigma-\frac12,-\frac12}}\ssf\bigr\}
\end{aligned}\end{equation*}
for every \ssf$t\!\in\!\R$\ssf. 
On the other hand,
since \,$T^*$ is bounded
from $L^{\tilde{p}'}\ssb(\R\ssf;\ssb L^{\tilde{q}'}\ssb(\Hn)\ssb)$
to $L^2(\Hn)$,
both expressions
\begin{equation*}\textstyle
\Bigl\|\,{\displaystyle\int_{\,0}^{\ssf t}}ds\;
\frac{\sin\ssf(t-s)\ssf D_x}{D_x}\,F(s,x)
\,\Bigr\|_{H_{\ssf x}^{\sigma-\frac12,\frac12}}
=\,\Bigl\|\,{\displaystyle\int_{\,0}^{\ssf t}}ds\;
\tilde{D}_{\,x}^{\ssf\sigma-\frac12}D_{\,x}^{-\frac12}
\sin\ssf(t\!-\!s)D_x\,F(s,x)\,\Bigr\|_{L_x^2}
\end{equation*}
and
\begin{equation*}\textstyle
\Bigl\|\,{\displaystyle\int_{\,0}^{\ssf t}}ds\,
\cos\ssf(t\!-\!s)D_x\,F(s,x)
\,\Bigr\|_{H_{\ssf x}^{\sigma-\frac12,-\frac12}}
=\,\Bigl\|\,{\displaystyle\int_{\,0}^{\ssf t}}ds\;
\tilde{D}_{\,x}^{\ssf\sigma-\frac12}D_{\,x}^{-\frac12}
\cos\ssf(t\!-\!s)D_x\,F(s,x)\,\Bigr\|_{L_x^2}
\end{equation*}
are bounded above by
\begin{equation*}\begin{aligned}
&\Big\|\;e^{\ssf\pm\ssf i\,t\ssf D_x}\!
\int_{-\infty}^{+\infty}\hspace{-1mm}ds\;\textstyle
\tilde{D}_{\,x}^{-\tilde{\sigma}+\frac12}D_{\,x}^{-\frac12}\ssf
e^{\mp\ssf i\ssf s\ssf D_x}\ssf
\tilde{D}_{\,x}^{\ssf\sigma+\tilde{\sigma}-1}\,
\1_{\ssf(0,\ssf t)}(s)\,F(s,x)\,\Big\|_{L^2_x}\\
&\lesssim\,\bigl\|\,\1_{\ssf(0,\ssf t)}(s)\,
\tilde{D}_{\,x}^{\ssf\sigma+\tilde{\sigma}-1}
F(s,x)\ssf\bigr\|_{L_s^{\tilde{p}'}L_x^{\tilde{q}'}}
\lesssim\,\|\ssf F\ssf
\|_{L^{\tilde{p}'}\ssb(\R\ssf;\ssf H_{\tilde{q}'}^{\sigma+\tilde{\sigma}-1}(\Hn))}\,,
\end{aligned}\end{equation*}
uniformly in \ssf$t\!\in\!\R$\ssf.
We conclude the proof of \eqref{Str2}
by summing up the previous estimates
and by taking the supremum over \ssf$t\!\in\!\R\ssf$.
\end{proof}

\begin{remark}\label{StrichartzI}
Observe that, in the statement of Theorem \ref{StrichartzEstimates},
we may replace \,$\R$ by any time interval \,$I$ containing \,$0$\ssf.
\end{remark}

\begin{remark}\label{dim2end}
An analogous result holds in dimension \,$n\!=\!2$
and its proof is similar,
except for the first convolution kernel in \eqref{HLS},
which becomes
\begin{equation*}
|\ssf t\ssb-\ssb s\ssf|^{-\alpha}\,
(\ssf1\!-\ssb\log|\ssf t\ssb-\ssb s\ssf|\ssf)^{\ssf\beta}\,
{\1}_{\,\{\ssf0<|\ssf t-s\ssf|<1\ssf\}}\,,
\end{equation*}
with \,$\alpha\ssb=\ssb\frac12\!-\!\frac1q$
and \,$\beta\ssb=\ssb2\ssf(\frac12\!-\!\frac1q)$\ssf.
It turns out that, in this case, a couple \,$(p,q)$ is \,{\rm admissible}
if \,$(\frac1p,\frac1q)$ belongs to the region
\,$T_2\ssb=\ssb\bigl\{\ssf(\frac1p,\frac1q)\!
\in\!(0,\frac12\ssf]\!\times\!(0,\frac12)
\bigm|\frac2p\!+\!\frac1q\!>\!\frac12\,\bigr\}$
{\rm(}see Figure 4\,{\rm)}.
\end{remark}

\begin{figure}[ht]
\begin{center}
\psfrag{0}[c]{$0$}
\psfrag{1}[c]{$1$}
\psfrag{1/2}[c]{$\frac12$}
\psfrag{1/4}[c]{$\frac14$}
\psfrag{1/p}[c]{$\frac1p$}
\psfrag{1/q}[c]{$\frac1q$}
\includegraphics[width=70mm]{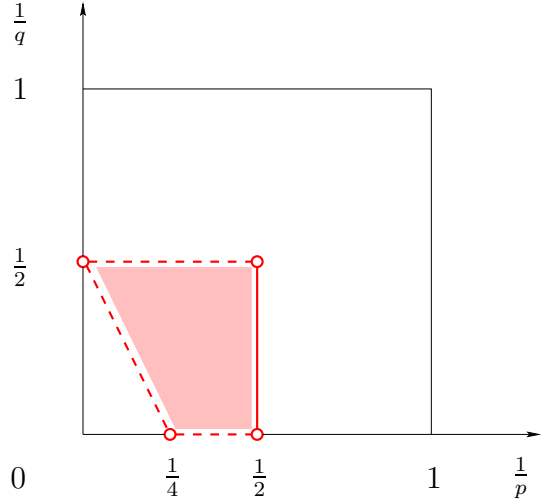}
\end{center}
\caption{Admissibility in dimension \ssf$n\!=\!2$}
\end{figure}

\section{LWP results for NLW equation on ${\Hn}$}\label{LWP}

We shall assume \ssf$n\!\ge\!4$ \ssf throughout this section
and discuss the lower dimensional cases \ssf$n\!=\!3$ \ssf and \ssf$n\!=\!2$
\ssf in the final remarks.
We apply Strichartz estimates
for the inhomogeneous linear Cauchy problem associated with the wave equation
to prove local well--posedness results for the following nonlinear  Cauchy problem
\begin{equation}\label{NLWhyperbolic}
\begin{cases}
& \partial_{\,t}^{\ssf 2} u(t,x) -(\Delta_{\Hn}\!+\!\rho^2)\,u(t,x) = F(u(t,x))\,\\
& u(0,x) = f(x)\,\\
& \partial_t|_{t=0}\,u(t,x) = g(x)\,,
\end{cases}
\end{equation}
with a power--like nonlinearity \ssf$F(u)$.
By this we mean that
\begin{equation}\label{power}
|F(u)|\le C\,|u|^\gamma
\quad\text{and}\quad
|\ssf F(u)\ssb-\ssb F(v)\ssf|\ssf\le\ssf
C\,(\ssf|u|^{\gamma-1}\ssb+\ssb|v|^{\gamma-1}\ssf)\,|\ssf u\ssb-\ssb v\ssf|
\end{equation}
for some \ssf$C\!\ge\!0$ \ssf and \ssf$\gamma\!>\!1$\ssf.
Let us recall the definition of local well--posedness.

\begin{definition}
The NLW Cauchy problem \eqref{NLWhyperbolic} is \ssf{\rm locally well--posed}
in \ssf$H^{\sigma,\tau}\!\times\ssb H^{\sigma,\tau-1}$
if, for any bounded subset \ssf$B$ of
\ssf$H^{\sigma,\tau}\!\times\ssb H^{\sigma,\tau-1}$,
there exist \ssf$T\!>\!0$ and a Banach space \ssf$X_T$,
continuously embedded into
\ssf$C\ssf(\ssf[-T,T\ssf]\ssf;H^{s,\tau})\cap
C^1(\ssf[-T,T\ssf]\ssf;H^{s,\tau-1})$\ssf,
such that
\newline
$\bullet$
\,for any initial data \ssf$(f,g)\!\in\!B$, 
$\eqref{NLWhyperbolic}$ has a unique solution \ssf$u\!\in\!X_T$,
\newline
$\bullet$
\,the map \ssf$(f,g)\ssb\mapsto\ssb u$
is continuous from \ssf$B$ into \ssf$X_T$.
\end{definition}
The amount of smoothness $\sigma$
requested for LWP of \eqref{NLWhyperbolic}
in $H^{\sigma-\frac12,\frac12}\!\times\ssb H^{\sigma-\frac12,-\frac12}$
depends on $\gamma$
and is represented in Figure 5.

\begin{figure}[ht]\label{LWPn}
\begin{center}
\psfrag{sigma}[c]{$\sigma$}
\psfrag{0}[c]{$0$}
\psfrag{1}[c]{$1$}
\psfrag{1/2}[c]{$\frac12$}
\psfrag{n/2}[c]{$\frac n2$}
\psfrag{gamma}[c]{$\gamma$}
\psfrag{gamma1}[c]{$\gamma_1$}
\psfrag{gamma2}[c]{$\gamma_2$}
\psfrag{gammaconf}[c]{$\gamma_{\text{conf}}$}
\psfrag{gammainfty}[c]{$\gamma_\infty$}
\psfrag{C1}[c]{$C_1$}
\psfrag{C2}[c]{$C_2$}
\psfrag{C3}[c]{$C_3$}
\includegraphics[width=100mm]{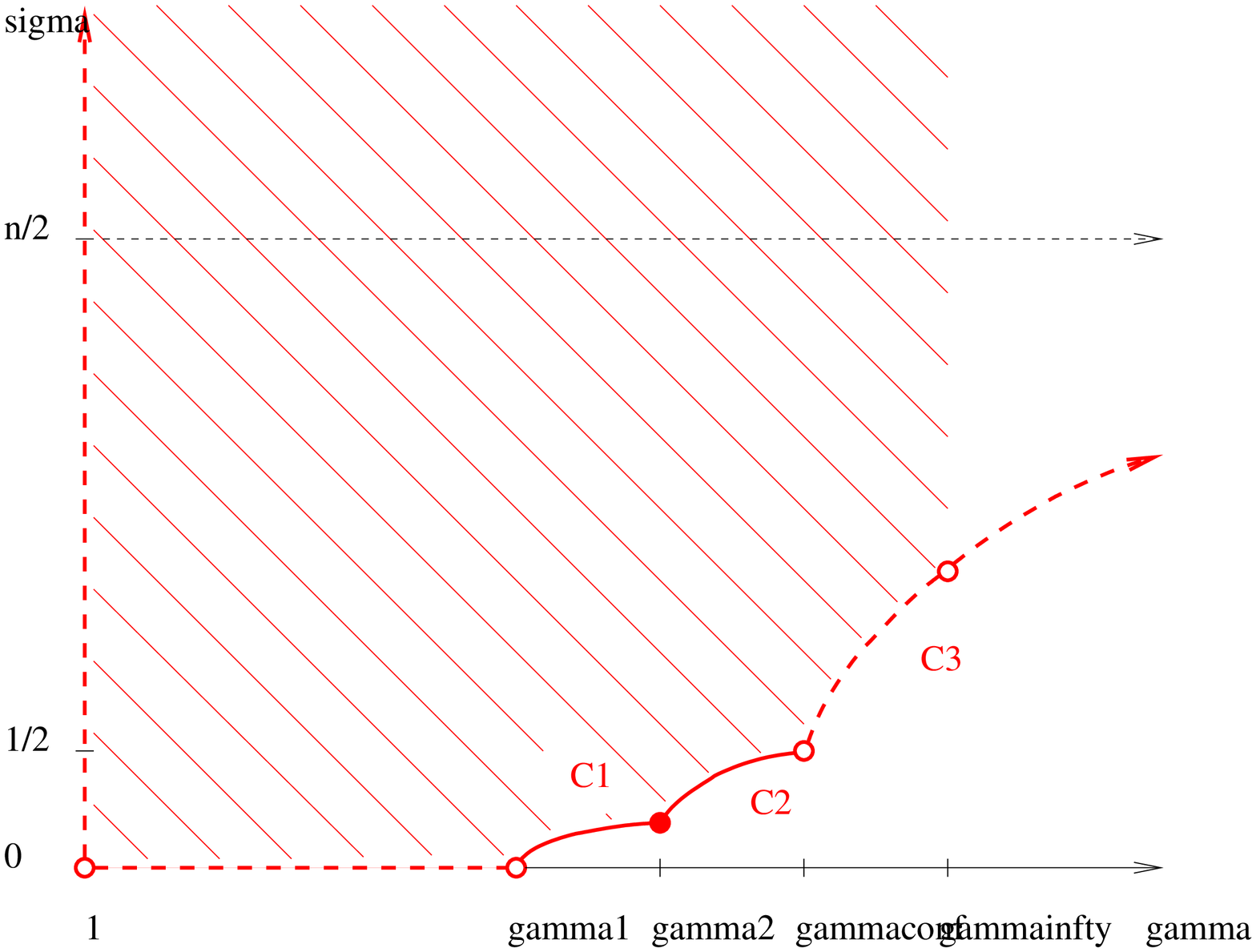}
\end{center}
\caption{Regularity in dimension $n\!\ge\!4$}
\end{figure}

\noindent
There
\vspace{2mm}

\centerline{$\hfill
\gamma_1\ssb
=\ssb\frac{n\ssf+\ssf3}n\ssb
=\ssb1\!+\ssb\frac3n\ssf,\hfill
\gamma_2\ssb
=\ssb\frac{(n+1)^2}{(n-1)^2+\ssf4}\ssb
=\ssb1\!+\ssb\frac{2}{\frac{n-1}2+\frac2{n-1}}\ssf,\hfill
\gamma_{\text{conf}}\ssb
=\ssb\frac{n\ssf+\ssf3}{n\ssf-\ssf1}\ssb
=\ssb1\!+\ssb\frac{4}{n\ssf-\ssf1}\ssf,
\hfill$}
\centerline{$
\gamma_3\ssb
=\ssb\frac{n^2+\ssf5\ssf n\ssf-\ssf2\ssf+\ssf
\sqrt{\ssf n^4+\ssf2\ssf n^3+\ssf21\ssf n^2-\ssf12\ssf n\ssf+\ssf4\ssf}}
{2\ssf n^2\ssf-\ssf2\ssf n}\ssb
=\ssb1\!+\ssb\frac{\sqrt{\ssf4\ssf n\ssf+\ssf(\frac{n-6}2-\frac2{n-1})^2}
\,-\,(\frac{n-6}2-\frac2{n-1})}n\ssf,
$}
\centerline{$\hfill
\gamma_4\ssb
=\ssb\frac{n^2+\ssf2\ssf n\ssf-\ssf5}{n^2-\ssf2\ssf n\ssf-\ssf1}\ssb
=\ssb1\!+\ssb\frac2{\frac{n-1}2-\frac1{n-1}}\ssf,\hfill
\gamma_\infty\ssb
=\min\ssf\{\gamma_3,\gamma_4\}
=\begin{cases}
\,\gamma_3
&\text{if \,}n\!=\!4,5\\
\,\gamma_4
&\text{if \,}n\!\ge\!6\\
\end{cases}
\hfill$}
\vspace{1mm}

\noindent
and the curves \ssf$C_1$, $C_2$\hspace{.1mm}, $C_3$ are given by
\vspace{2mm}

\noindent
\centerline{$\hfill
C_1(\gamma)\ssb
=\ssb\frac{n\ssf+\ssf1}4\ssf\bigl(\ssf1\ssb-\ssb\frac{n\ssf+\ssf5}
{2\,n\ssf\gamma\ssf-\ssf n\ssf-\ssf1}\ssf\bigr)\ssf,\hfill
C_2(\gamma)\ssb
=\ssb\frac{n\ssf+\ssf1}4\ssb-\ssb\frac1{\gamma\ssf-\ssf1}\ssf,
\hfill
C_3(\gamma)\ssb=\ssb\frac n2\ssb-\ssb\frac2{\gamma\ssf-\ssf1}\ssf.
\hfill$}\vspace{1mm}

\noindent
When \ssf$\gamma\!<\!\gamma_{\infty}$\ssf,
we obtain the same regularity curve as in the Euclidean case.
Since our Strichartz estimates hold for a large family of admissible pairs,
they are sufficient to study the regularity problem via a fixed point argument\,;
in the Euclidean setting this problem was solved  by different methods,
depending on the range of the power $\gamma$ involved in the nonlinearity
and on the regularity of initial data.

\begin{theorem}\label{WPL2}
Let \ssf$n\!\ge\!4$ and assume that \ssf$F(u)$ satisfies \eqref{power}.
Then the NLW \eqref{NLWhyperbolic} is locally well--posed
in \ssf$H^{\sigma-\frac12,\frac12}\!\times\ssb H^{\sigma-\frac12,-\frac12}$
in the following cases$\,:$
\begin{itemize}
\item[(A)]
\,$1\!<\!\gamma\!\le\!\gamma_1$
and \,$\sigma\!>\!0\,;$
\item[(B)]
\,$\gamma_1\!<\!\gamma\!\le\!\gamma_2$
and \,$\sigma\!\ge\!C_1(\gamma)\,;$
\item[(C)]
\,$\gamma_2\!\le\!\gamma\!<\!\gamma_{\mathrm{conf}}$
and \,$\sigma\!\ge\!C_2(\gamma)\,;$
\item[(D)]
$\gamma_{\mathrm{conf}}\!\le\!\gamma\!<\!\gamma_{\infty}$
and \,$\sigma\!>\!C_3(\gamma)\,.$
\end{itemize}
More precisely,
for all such nonlinearity power \ssf$\gamma$
and regularity  \ssf$\sigma$,
there exists a positive $T$,  depending on the initial data,
and a unique solution \ssf$u$ to NLW \eqref{NLWhyperbolic} such that
\begin{equation*}
u\in C\bigl(\ssf[-T,T\ssf]\ssf;H^{\sigma-\frac12,\frac12}(\Hn)\bigr)\,
\cap\ssf L^{p_0}\bigl(\ssf[-T,T\ssf]\ssf;L^{q_0}(\Hn))\,,
\end{equation*}
for a suitable admissible couple $(p_0,q_0)$, and
\begin{equation*}
\partial_{\ssf t}u\in C\bigl(\ssf[-T,T\ssf]\ssf;H^{\sigma-\frac12,-\frac12}(\Hn)\bigr)\,.
\end{equation*}
\end{theorem}

\begin{proof}
We apply the standard fixed point method based on Strichartz estimates.
Define \,$u\!=\!\Phi(v)$ \,as the solution
to the following linear Cauchy problem
\begin{equation}\begin{cases}
\,\partial_{\,t}^{\ssf2}u(t,x)\ssb-\ssb D_x^2u(t,x)=F(v(t,x))\ssf,\\
\,u(0,x)\ssb=\ssb f(x)\ssf,\\
\,\partial_t|_{t=0}\ssf u(t,x)\ssb=\ssb g(x)\ssf,\\
\end{cases}\end{equation}
which is given by the Duhamel formula
\begin{equation*}\textstyle
u(t,x)
=(\cos t\ssf D_x)\ssf f\ssf(x)
+\frac{\sin t\ssf D_x}{D_x}\,g\ssf(x)
+{\displaystyle\int_{\,0}^{\,t}}\!ds\,\frac{\sin\ssf(t-s)\ssf D_x}{D_x}\ssf F(v(s,x))\,.
\end{equation*}
We deduce from the Strichartz estimates \eqref{Str1}, \eqref{Str2}
and from Remark \ref{StrichartzI} that
\begin{equation*}\label{StrichartzL2v1}
\begin{aligned}
&\|u\|_{L^{\infty}\big([-T,T\ssf]\ssf;\ssf H^{\sigma-\frac12,\frac12}\big)}
+\|\partial_tu\|_{L^{\infty}\big([-T,T\ssf]\ssf;
\ssf H^{\sigma-\frac12,-\frac12}\big)}
+\|u\|_{L^p([-T,T\ssf];\ssf L^q\vphantom{L_t^{\tilde p'}})}\\
&\lesssim\,\|f\|_{H^{\sigma-\frac12,\frac12}}
+\,\|g\|_{H^{\sigma-\frac12,-\frac12}}
+\,\|F(v)\|_{L^{\tilde{p}'}\ssb\bigl([-T,\ssf T\ssf]\ssf;
\ssf H^{{\sigma+\tilde{\sigma}}-1}_{\tilde{q}'}\bigr)}\,,
\end{aligned}
\end{equation*}
for all admissible couples $(p,q)$, $(\tilde{p},\tilde{q})$
introduced in Definition \ref{admcouple},
for all \ssf$\sigma\!\ge\!\frac{n+1}2\bigl(\frac12\!-\!\frac1q\bigr)$,
$\tilde{\sigma}\!\ge\!\frac{n+1}2\bigl(\frac12\!-\!\frac1{\tilde{q}}\bigr)$,
and for a positive \ssf$T$ to be determined later.
According to the  nonlinear assumption \eqref{power},
we estimate the inhomogeneous term as follows\,:
\begin{equation*}
\|F(v)\|_{L^{\tilde p'}\bigl([-T,T\ssf]\ssf;
\ssf\,H^{{\sigma+\tilde{\sigma}}-1}_{\tilde q'}\big)}
\lesssim \,\|\,|v|^\gamma\|_{L^{\tilde p'}\big([-T,T\ssf]\ssf;\ssf
H^{{\sigma+\tilde{\sigma}}-1}_{\tilde q'}\big)}\ssf.
\end{equation*}
Assuming \ssf$\sigma\ssb+\ssb\tilde{\sigma}\ssb-\!1\ssb
\le\ssb n\ssf(\frac1{\tilde{q}'}\ssb-\ssb\frac1{\tilde{q}_1'})\ssb\le\ssb0$\ssf,
we deduce from Sobolev's embedding (Proposition \ref{SET}) that
\begin{equation*}
\begin{aligned}
&\|u\|_{L^{\infty}\big([-T,T\ssf]\ssf;\ssf H^{\sigma-\frac12,\frac12}\big)}
+\,\|\partial_tu\|_{L^{\infty}\big([-T,T\ssf]\ssf;\ssf
H^{\sigma-\frac12,-\frac12}\big)}
+\,\|u\|_{L^p([-T,T\ssf]\ssf;\ssf L^q\vphantom{L_t^{\tilde p'}})}\\
&\lesssim\,\|f\|_{H^{\sigma-\frac12,\frac12}}
+\,\|g\|_{H^{\sigma-\frac12,-\frac12}}
+\,\|v\|_{L^{\tilde{p}'\gamma}\big([-T,T\ssf]\ssf;\ssf
L^{\tilde{q}_1'\ssb\gamma}\big)}^{\,\gamma}\ssf.
\end{aligned}
\end{equation*}
In order to remain within the same function space,
we require that \ssf$q\ssb=\ssb\tilde{q}_1'\gamma$\ssf.
After applying H\"older's inequality in time, we obtain
\begin{equation}\label{StrichartzL2v2}
\begin{aligned}
&\|u\|_{L^{\infty}\big([-T,T];\,H^{\sigma-\frac12,\frac12}\big)}
+\,\|\partial_tu\|_{L^{\infty}\big([-T,T];\,H^{\sigma-\frac12,-\frac12}\big)}
+\,\|u\|_{L^p([-T,T];\,L^q)\vphantom{H^{\frac12}}}\\
&\lesssim\,\|f\|_{H^{\sigma-\frac12,\frac12}}+
\,\|g\|_{H^{\sigma-\frac12,-\frac12}}
+\,T^{\ssf\lambda}\,\|v\|_{L^p([-T,T];\,L^q)\vphantom{H^{\frac12}}}^{\,\gamma}\,.
\end{aligned}
\end{equation}
Here we have assumed \ssf$p\ssb>\ssb\tilde{p}'\gamma$
\ssf and set \ssf$\lambda\ssb
=\ssb\frac1{\widetilde{p}'}\!-\!\frac\gamma p\ssb>\ssb0$\ssf.
It remains for us to check that
the following conditions can be fulfilled simultaneously\,:
\begin{equation}\label{condts}
\begin{cases}
\;\text{(i)}&
p\ssb>\ssb\tilde{p}'\gamma\,,\\
\;\text{(ii)}&
0\ssb<\ssb\frac1{\tilde{q}'}\ssb\le\ssb\frac\gamma q\ssb<\ssb1\,,\\
\;\text{(iii)}&
\frac{n-1}2\ssb-\ssb\frac{n+1}2\ssf
\bigl(\frac1q\ssb+\ssb\frac1{\tilde{q}}\bigr)\ssb
\le\ssb n\ssf\bigl(\frac1{\tilde{q}'}\ssb-\ssb\frac\gamma q\bigr)\,,\\
\;\text{(iv)}&
\frac2p\ssb+\ssb\frac{n-1}q\ssb\ge\ssb\frac{n-1}2\,,\\
\;\text{(v)}&
\frac2{\tilde{p}}\ssb+\ssb\frac{n-1}{\tilde{q}}\ssb\ge\ssb\frac{n-1}2\,,\\
\;\text{(vi)}&
\bigl(\frac1p,\frac1q\bigr)\ssb
\in\ssb\bigl(0,\frac12\bigr]\ssb
\times\ssb\bigl(\frac{n-3}{2\ssf(n-1)},\frac12\bigr)\,,\\ 
\;\text{(vii)}&
\bigl(\frac1{\tilde{p}},\frac1{\tilde{q}}\bigr)\ssb
\in\ssb\bigl(0,\frac12\bigr]\ssb
\times\ssb\bigl(\frac{n-3}{2\ssf(n-1)},\frac12\bigr)\,. 
\end{cases}
\end{equation}
Suppose indeed that there exist indices
\ssf$p,q,\tilde{p},\tilde{q}$ \ssf
satisfying all conditions in (\ref{condts}).
Then \eqref{StrichartzL2v2} shows that
\ssf$\Phi$ \ssf maps \ssf$X$ \ssf into itself,
where \ssf$X$ \ssf denotes the Banach space
\begin{equation*}\begin{aligned}
X=\bigl\{\,u\,\big|\;
&u\ssb\in\ssb
C\ssf(\ssf[-T,T\ssf]\ssf;H^{\sigma-\frac12,\frac12}(\Hn))
\ssf\cap\ssf L^p(\ssf[-T,T\ssf]\ssf;L^q(\Hn))\,,\\
&\partial_{\ssf t\ssf}u\ssb\in\ssb
C\ssf([-T,T\ssf]\ssf;H^{\sigma-\frac12,-\frac12}(\Hn))\,\bigr\}\,,
\end{aligned}\end{equation*}
equipped with the norm
\begin{equation*}
\|u\|_{X\vphantom{H^{\frac12}}}
=\,\|u\|_{L^\infty\bigl(\ssf[-T,T\ssf]\ssf;\ssf
H^{\sigma-\frac12,\frac12}\bigr)}
+\,\|\partial_{\ssf t\ssf}u\|_{L^\infty\bigl(\ssf[-T,T\ssf]\ssf;\ssf
H^{\sigma-\frac12,-\frac12}\bigr)}
+\,\|u\|_{L^p\bigl(\ssf[-T,T\ssf]\ssf;\ssf L^q\bigr)}\,,
\end{equation*}
Moreover we shall show that $\Phi$ is a contraction on the ball
\begin{equation*}
X_M=\{\,u\!\in\!X\mid\|u\|_X\!\le\!M\,\}\,,
\end{equation*}
provided the time \ssf$T\!>\!0$ \ssf is sufficiently small
and the radius \ssf$M\!>\!0$ \ssf is sufficiently large.
Let \ssf$v,\tilde{v}\!\in\!X$ \ssf
and \ssf$u\!=\!\Phi(v)$\ssf, $\tilde{u}\!=\!\Phi(\tilde{v})$\ssf.
By arguing as above and using H\"older's inequality,
we have
\begin{equation}\label{contractionL2}
\begin{aligned}
\|\,u\ssb-\ssb\tilde{u}\,\|_{X\vphantom{L_t^{\tilde{p}'}}}
&\le\,C\;\|\ssf F(v)\ssb-\ssb F(\tilde{v})\ssf\|
_{L^{\tilde{p}'}\bigl(\ssf[-T,T\ssf]\ssf;\ssf
H_{\tilde{q}'}^{\sigma+\tilde{\sigma}-1}\big)}\\
&\le\,C\;\|\ssf
\{\ssf|v|^{\gamma-1}\!+|\tilde v|^{\gamma-1}\ssf\}\,
|\ssf v\ssb-\ssb\tilde{v}\ssf|\,\|
_{L^{\tilde{p}'}\bigl(\ssf[-T,T\ssf]\ssf;\ssf L^{\tilde{q}_1'}\bigr)}\\
&\le\,C\;T^{\ssf\lambda}\,
\bigl\{\ssf\|v\|_{L^p\bigl([-T,T\ssf]\ssf;\ssf L^q\bigr)}^{\,\gamma-1}\!
+\|\tilde{v}\|_{L^p\bigl([-T,T\ssf]\ssf;\ssf L^q\bigr)}^{\,\gamma-1}\bigr\}\,
\|\ssf v\ssb-\ssb\tilde{v}\ssf\|_{L^p\bigl([-T,T\ssf]\ssf;\ssf L^q\bigr)}\\
& \le\,C\;T^{\ssf\lambda}\,
\bigl\{\ssf\|v\|_X^{\gamma-1}\!+\|\tilde{v}\|_X^{\gamma-1}\ssf\bigr\}\,
\|\ssf v\ssb-\ssb\tilde{v}\ssf\|_X\,.
\end{aligned}
\end{equation}
If \ssf$\|v\|_X\!\le\! M$ \ssf and \ssf$\|\tilde{v}\|_X\!\le\! M$\ssf,
then \eqref{StrichartzL2v2} yields on the one hand
\begin{equation*}
\|u\|_X\le\ssf C\,\bigl\{\ssf\|f\|_{H^{\sigma-\frac12,\frac12}}\ssb
+\ssf\|g\|_{H^{\sigma-\frac12,-\frac12}}\ssb
+\ssf T^{\ssf\lambda}\ssf M^{\ssf\gamma}\ssf\bigr\}
\end{equation*}
and
\begin{equation*}
\|\tilde{u}\|_X\le\ssf C\,\bigl\{\ssf\|f\|_{H^{\sigma-\frac12,\frac12}}\ssb
+\ssf\|g\|_{H^{\sigma-\frac12,-\frac12}}\ssb
+\ssf T^{\ssf\lambda}\ssf M^{\ssf\gamma}\ssf\bigr\}\ssf,
\end{equation*}
while \eqref{contractionL2} yields on the other hand
\begin{equation*}
\|u-\tilde u\|_X\le
2\;C\;T^{\lambda}\ssf M^{\gamma-1}\,\|\ssf v\ssb-\ssb\tilde{v}\ssf\|_X\,.
\end{equation*}
Thus, if we choose \ssf$M\!>\ssb0$ \ssf so large that
\ssf$\frac M2\!\ge\ssb C\,\bigl\{
\ssf\|f\|_{H^{\sigma-\frac12,\frac12}}\ssb
+\ssf\|g\|_{H^{\sigma-\frac12,-\frac12}}\bigr\}$
\ssf and \ssf$T\!>\!0$ \ssf so small that
\ssf$C\,T^{\ssf\lambda}M^{\ssf\gamma}\!\le\!\frac M2$
\ssf and \ssf$2\,C\,T^{\ssf\lambda}M^{\gamma-1}\!\le\ssb\frac12$\ssf,
then
\begin{equation*}\textstyle
\|u\|_X\le M\ssf,\;\|\tilde{u}\|_X\le M
\quad\text{and}\quad
\|\ssf u\ssb-\ssb\tilde u\ssf\|_X\le\frac12\,\|\ssf v\ssb-\ssb\tilde{v}\ssf\|_X
\end{equation*}
if \ssf$v,\tilde{v}\ssb\in\!X_M$
and \ssf$u\ssb=\ssb\Phi(v)$\ssf, $\tilde{u}\ssb=\ssb\Phi(\tilde{v})$\ssf.
Hence the map \ssf$\Phi$ \ssf is a contraction
on the complete metric space $X_M$
and the fixed point theorem allows us to conclude.

Let us eventually prove the existence
of couples $(p,q)$ and $(\tilde{p},\tilde{q})$
satisfying all conditions in \eqref{condts}.
Condition (\ref{condts}.iii) amounts to
\begin{equation}\label{cond0}\textstyle
\frac{2\ssf n\ssf\gamma\ssf-\ssf n\ssf-\ssf1}q+\frac{n-1}{\tilde{q}}\le n\ssb+\ssb1
\quad\mathrm{i.e.}\quad
\frac1{\tilde{q}}\le\frac{n+1}{n-1}\ssb-\ssb\frac{2n\gamma-n-1}{n-1}\frac1q\,.
\end{equation}
By combining \eqref{cond0} with (\ref{condts}.ii)  and (\ref{condts}.vi),
we deduce that
\begin{equation*}\textstyle
\frac{n-3}{2(n-1)}\le\frac1q\le\frac2{(\gamma-1)(n+1)}\,.
\end{equation*}
This implies that \ssf$\gamma\ssb
\le\ssb\widetilde{\gamma}_\infty\!
=\ssb\frac{n^2+2n-7}{(n+1)(n-3)}\ssb
=\ssb1\!+\ssb\frac{4(n-1)}{(n+1)(n-3)}$\ssf.
By combining \eqref{cond0} with (\ref{condts}.vii),
we obtain
\begin{equation*}\textstyle
\frac{n-3}{2(n-1)}\le\frac1{\tilde{q}}\le
\min\big\{\ssf\frac12,\frac{n+1}{n-1}\ssb-\ssb\frac{2n\gamma-n-1}{n-1}\frac1q\ssf\bigr\}\,,
\quad\frac1{\tilde{q}}\ne\frac12\,.
\end{equation*}
By combining (\ref{cond0}) with (\ref{condts}.vii),
we also obtain \ssf$\frac1q\le\frac{n+5}{2(2n\gamma-n-1)}$\ssf.
In summary, the conditions on $q$ reduce to
\begin{equation*}\textstyle
\frac{n-3}{2(n-1)}\le\frac1q\le\min\bigl\{
\frac12,
\frac1\gamma,
\frac2{(\gamma-1)(n+1)},
\frac{n+5}{2(2n\gamma-n-1)}
\bigr\}\,,\quad
\frac1q\ne\frac12,\frac1\gamma\,,
\end{equation*}
or case by case to
\begin{itemize}
\item
$1\ssb<\ssb\gamma\ssb\le\ssb\gamma_1$ \ssf and
\ssf$\frac{n-3}{2(n-1)}\ssb\le\ssb\frac1q\ssb<\ssb\frac12$\ssf,
\item
$\gamma_1\ssb<\ssb\gamma\ssb\le\ssb\gamma_2$
\ssf and \ssf$\frac{n-3}{2(n-1)}\ssb
\le\ssb\frac1q\ssb
\le\ssb\frac{n+5}{2(2n\gamma-n-1)}$\ssf,
\item
$\gamma_2\ssb<\ssb\gamma\ssb\le\ssb\widetilde{\gamma}_{\infty}$ \ssf and
\ssf$\frac{n-3}{2(n-1)}\ssb\le\ssb\frac1q\ssb\le\ssb\frac2{(\gamma-1)(n+1)}$\ssf.
\end{itemize}
Let us turn to the indices \ssf$p$ \ssf and \ssf$\tilde{p}$\ssf.
According to \eqref{condts}, we have
\begin{equation*}\textstyle
\frac{n-1}2\ssf\bigl(\frac12\ssb-\ssb\frac1q\bigr)\ssb
\le\ssb\frac1p\ssb\le\ssb\frac12
\end{equation*}
and
\begin{equation*}\textstyle
\frac{n-1}2\ssf\bigl(\frac12\ssb-\ssb\frac1{\tilde{q}}\bigr)\ssb
\le\ssb\frac1{\tilde{p}}\ssb
\le\ssb\min\ssf\bigl\{\frac12,1\ssb-\ssb\frac\gamma p\bigr\}\,,
\quad\frac1{\tilde{p}}\ssb\ne\ssb1\ssb-\ssb\frac\gamma p\,.
\end{equation*}
By taking into account the previous conditions on \ssf$q$\ssf,
we end up with the following conditions on \ssf$p$ \ssf and \ssf$\tilde{p}$\;: 
\begin{equation}\label{condindecesp}\begin{cases}
\;\mathrm{(i)}&\textstyle
\frac{n-1}2\ssf\bigl(\frac12\ssb-\ssb\frac1q\bigr)\ssb
\le\ssb\frac1p\ssb\le\ssb\min\ssf\big\{\frac12,
\frac{5-n}{4\ssf\gamma}\ssb+\ssb\frac{n-1}{2\ssf\gamma\ssf\tilde{q}}\bigr\}\,,
\quad\frac1p\ssb\ne\ssb
\frac{5-n}{4\ssf\gamma}\ssb+\ssb\frac{n-1}{2\ssf\gamma\ssf\tilde{q}}\,,\\
\;\mathrm{(ii)}&\textstyle
\frac{n-1}2\ssf\bigl(\frac12\ssb-\ssb\frac1{\tilde{q}}\bigr)\ssb
\le\ssb\frac1{\tilde{p}}\ssb
\le\ssb\min\ssf\bigl\{\frac12,1\ssb-\ssb\frac\gamma p\bigr\}\,,
\quad\frac1{\tilde{p}}\ssb\ne\ssb1\ssb-\ssb\frac\gamma p\,.\\
\end{cases}\end{equation}
There exist indices \ssf$p$ \ssf and \ssf$\tilde{p}$
\ssf which satisfy \eqref{condindecesp} provided that
\ssf$\frac1{\tilde{q}}\ssb
>\ssb\frac{\gamma}2\ssb
+\ssb\frac{n-5}{2\ssf(n-1)}\ssb
-\ssb\frac{\gamma}{q}$\ssf.
We thus have to find \ssf$\tilde{q}$ \ssf such that  
\begin{equation}\label{condtildeq}\textstyle
\max\ssf\bigl\{\frac{n-3}{2\ssf(n-1)},
\frac\gamma2\ssb+\ssb\frac{n-5}{2\ssf(n-1)}\ssb-\ssb\frac\gamma q\bigr\}
\le\frac1{\tilde{q}}
\le\min\ssf\bigl\{\frac12,\frac{n+1}{n-1}\ssb
-\ssb\frac{2\ssf n\ssf \gamma\ssf-\ssf n\ssf-1}{(n-1)\ssf q}\bigr\}\ssf,
\end{equation}
with \ssf$\frac1{\tilde{q}}\ssb\ne\ssb\frac12\ssf,\,
\frac\gamma2\ssb+\ssb\frac{n-5}{2\ssf(n-1)}\ssb-\ssb\frac\gamma q$\ssf.
This implies that \ssf$q$ \ssf has to satisfy the following conditions\,: 
\begin{equation}\label{condq}\textstyle
\max\ssf\bigl\{\frac{n-3}{2\ssf(n-1)},
\frac12\ssb-\ssb\frac2{\gamma\ssf(n-1)}\bigr\}
\le\frac1q
\le\min\ssf\bigl\{\frac12,\frac1\gamma,
\frac2{(\gamma-1)\ssf(n+1)},
\frac{n+5}{2\ssf(2n\gamma-n-1)},
\frac{n+7-\gamma\ssf(n-1)}{2\ssf(\gamma-1)\ssf(n+1)}\bigr\}\ssf,
\end{equation}
with \ssf$\frac1q\ssb
\ne\ssb\frac12\ssb-\ssb\frac2{\gamma\ssf(n-1)}\ssf,
\,\frac12\ssf,
\,\frac1\gamma\ssf,
\,\frac{n+7-\gamma\ssf(n-1)}{2\ssf(\gamma-1)\ssf(n+1)}$\ssf. 
The fact that
\ssf$\frac{n-3}{2\ssf(n-1)}\ssb
<\ssb\frac{n+7-\gamma\ssf(n-1)}{2\ssf(\gamma-1)\ssf(n+1)}$
\ssf easily implies that
\ssf$\gamma\!<\!\gamma_4\!<\tilde{\gamma}_\infty$\ssf. 
The fact that
\ssf$\frac12\ssb-\ssb\frac2{\gamma\ssf(n-1)}\ssb
<\ssb\frac{n+7-\gamma\ssf(n-1)}{2\ssf(\gamma-1)\ssf(n+1)}$
\ssf implies that \ssf$\gamma\!<\!\gamma_3$\ssf.
In summary, here are the final conditions on \ssf$q$\ssf,
depending on \ssf$\gamma$
\ssf and possibly on the dimension \ssf$n$\,:
\begin{itemize}
\item[(A)] 
\,$1\ssb<\ssb\gamma\ssb\le\ssb\gamma_1\ssb=\ssb1\ssb+\ssb\frac3n$
\ssf and
\ssf$\frac{n-3}{2\ssf(n-1)}\ssb\le\ssb\frac1q\ssb<\ssb\frac12$\ssf.
\item[(B)]
\,$\gamma_1\ssb<\ssb\gamma\ssb\le\ssb\gamma_2\ssb
=\ssb\frac{(n+1)^2}{n^2-2n+5}$
\ssf and
\ssf$\frac{n-3}{2\ssf(n-1)}\ssb\le\ssb\frac1q\ssb
\le\ssb\frac{n+5}{2\ssf(2n\gamma-n-1)}$\ssf.
\item[(C)]
\,$\gamma_2\ssb<\ssb\gamma\ssb<\ssb\gamma_{\mathrm{conf}}$
\ssf and
\ssf$\frac{n-3}{2\ssf(n-1)}\ssb\le\ssb\frac1q\ssb
\le\ssb\frac2{(\gamma-1)\ssf(n+1)}$
\ssf when \ssf$n\ssb\ge\ssb5$\ssf.\\
When \ssf$n\ssb=\ssb4$\ssf,
we distinguish two subcases\,:
\begin{itemize}
\item[$\bullet$]
\,$\gamma_2\ssb<\ssb\gamma\ssb\le\ssb2$
\ssf and
\ssf$\frac{n-3}{2\ssf(n-1)}\ssb\le\ssb\frac1q\ssb
\le\ssb\frac2{(\gamma-1)\ssf(n+1)}$\ssf,
\item[$\bullet$]
\,$2\ssb<\ssb\gamma\ssb<\ssb\gamma_{\mathrm{conf}}$
\ssf and
\ssf$\frac12\ssb-\ssb\frac2{\gamma\ssf(n-1)}\ssb
<\ssb\frac1q\ssb\le\ssb\frac2{(\gamma-1)\ssf(n+1)}$\ssf.
\end{itemize}
\item[(D)]
\,When \ssf$n\ssb\ge\ssb6$\ssf,
we distinguish two subcases\,:
\begin{itemize}
\item[$\bullet$]
\,$\gamma_{\mathrm{conf}}\ssb\le\ssb\gamma\ssb\le\ssb2$
\ssf and
\ssf$\frac{n-3}{2\ssf(n-1)}\ssb\le\ssb\frac1q\ssb
<\ssb\frac{n+7-\gamma\ssf(n-1)}{2\ssf(\gamma-1)\ssf(n+1)}$\ssf,
\item[$\bullet$]
\,$2\ssb<\ssb\gamma\ssb<\ssb\gamma_4$
\ssf and
\ssf$\frac12\ssb-\ssb\frac2{\gamma\ssf(n-1)}\ssb<\ssb\frac1q\ssb
<\ssb\frac{n+7-\gamma\ssf(n-1)}{2\ssf(\gamma-1)\ssf(n+1)}$\ssf. 
\end{itemize}
When \ssf$n\ssb=\ssb5$\ssf,
we replace \ssf$\gamma_4$ \ssf by \ssf$\gamma_3$\ssf.\\
When \ssf$n\ssb=\ssb4$\ssf,
\ssf$\gamma_{\mathrm{conf}}\ssb\le\ssb\gamma\ssb<\ssb\gamma_3$
\ssf and
\ssf$\frac12\ssb-\ssb\frac2{\gamma\ssf(n-1)}\ssb<\ssb\frac1q\ssb<\ssb\frac{n+7-\gamma\ssf(n-1)}{2\ssf(\gamma-1)\ssf(n+1)}$\ssf.
\end{itemize}

Let us now examine these cases separately.
\smallskip

\noindent
{\bf Case (A).}
In this case,
we choose successively \ssf$q$ \ssf such that
\begin{equation*}\textstyle
\frac{n-3}{2(n-1)}\le\frac1q<\frac12\,,
\end{equation*}
$\tilde q$ \ssf satisfying \eqref{condtildeq},
and \ssf$p$\ssf, $\tilde{p}$ \ssf satisfying \eqref{condindecesp}. 
Thus, when
\ssf$1\ssb<\ssb\gamma\ssb\le\gamma_1$
\ssf and \ssf$\sigma\ssb>\ssb0$\ssf,
there exists always an admissible couple $(p,q)$
such that all conditions \eqref{condts} are satisfied and
\ssf$\sigma\ssb\ge\ssb\frac{(n+1)}2\ssf(\frac12\ssb-\ssb\frac1q)$\ssf. 
 \smallskip

\noindent
{\bf{Case (B)}}.
In this case,
we choose successively \ssf$q$ \ssf such that
\begin{equation*}\textstyle
\frac{n-3}{2(n-1)}  \leq \frac{1}{q} \leq \frac{n+5}{2 ( 2n\gamma - n - 1)}
\end{equation*}
$\tilde q$ \ssf satisfying \eqref{condtildeq},
and \ssf$p$\ssf, $\tilde{p}$ \ssf satisfying \eqref{condindecesp}. 
and a correspondent $\tilde q$ which satisfies \eqref{condtildeq}.
Thus, when
\ssf$\gamma_1\ssb<\ssb\gamma\ssb\le\ssb\gamma_2$
\ssf and \ssf$\sigma\ssb
\ge\ssb\frac{n+1}4\ssb-\ssb\frac{(n+1)\ssf(n+5)}{4\ssf(2n\gamma-n-1)}$\ssf,
there exists always an admissible couple $(p,q)$ such that 
all conditions \eqref{condts} are satisfied and
\ssf$\sigma\ssb\ge\ssb\frac{(n+1)}2\ssf(\frac12\ssb-\ssb\frac1q)$\ssf.
\smallskip

\noindent
{\bf Case (C).} Assume first that \ssf$n\ssb\ge\ssb5$\ssf.
we choose successively \ssf$q$ \ssf such that
\begin{equation}\label{conditionC1}\textstyle
\frac{n-3}{2\ssf(n-1)}\le\frac1q\le\frac2{(\gamma-1)\ssf(n+1)}\,,
\end{equation}
$\tilde q$ \ssf satisfying \eqref{condtildeq},
and \ssf$p$\ssf, $\tilde{p}$ \ssf satisfying \eqref{condindecesp}. 
 
Assume next that \ssf$n\ssb=\ssb4$\ssf.
If \ssf$\gamma_2\ssb<\ssb\gamma\ssb\le\ssb2$\ssf,
we choose \ssf$q$ \ssf according to \eqref{conditionC1}.
If \ssf$2\ssb<\ssb\gamma\ssb<\ssb\gamma_{\mathrm{conf}}\ssf$,
we replace \eqref{conditionC1} by
\begin{equation*}\textstyle
\frac12\ssb-\ssb\frac2{\gamma\ssf(n-1)}\ssb<\ssb\frac1q\ssb
\le\ssb\frac2{(\gamma-1)\ssf(n+1)}\,.
\end{equation*}
In both cases,
we can choose afterwards \ssf$\tilde{q},p,\tilde{p}$ \ssf
satisfying \eqref{condtildeq} and \eqref{condindecesp}.  

In summary, when
\ssf$\gamma_2\ssb<\ssb\gamma\ssb<\ssb\gamma_{\mathrm{conf}}$
\ssf and \ssf$\sigma\ssb\ge\ssb\frac{n+1}4\ssb-\ssb\frac1{\gamma-1}$\ssf,
there exists always an admissible couple $(p,q)$
such that all conditions \eqref{condts} are satisfied
and \ssf$\sigma\ssb\ge\ssb\frac{(n+1)}2\ssf(\frac12\ssb-\ssb\frac1q)$\ssf.
\smallskip

\noindent
{\bf Case (D).}
Assume first that \ssf$n\ssb\ge\ssb6$\ssf.
If \ssf$\gamma_{\text{conf}}\ssb\le\ssb\gamma\ssb\le\ssb2$\ssf,
we choose successively \ssf$q$ \ssf such that
\begin{equation}\label{choiceD1}\textstyle
\frac{n-3}{2(n-1)}
\le\frac1q
<\frac{n+7-\gamma(n-1)}{2(\gamma-1)(n+1)}\,,
\end{equation}
$\tilde q$ \ssf satisfying \eqref{condtildeq},
and \ssf$p$\ssf, $\tilde{p}$ \ssf satisfying \eqref{condindecesp}. 
If \ssf$2\ssb<\ssb\gamma\ssb<\ssb\gamma_4$\ssf,
\eqref{choiceD1} is replaced by
\begin{equation}\label{choiceD2}\textstyle
\frac12\ssb-\ssb\frac2{\gamma(n-1)}
<\frac1q
<\frac{n+7-\gamma(n-1)}{2(\gamma-1)(n+1)}\,.
\end{equation}

Assume next that \ssf$n\ssb=\ssb5$\ssf.
We choose again \ssf$q$ \ssf according to \eqref{choiceD1}
if \ssf$\gamma_{\text{conf}}\ssb\le\ssb\gamma\ssb\le\ssb2$
\ssf and according to \eqref{choiceD2}
if \ssf$2\ssb<\ssb\gamma\ssb<\ssb\gamma_3$\ssf.
In both cases,
we can choose afterwards \ssf$\tilde{q},p,\tilde{p}$ \ssf
satisfying \eqref{condtildeq} and \eqref{condindecesp}.  

Assume eventually that \ssf$n\ssb=\ssb4$\ssf.
Then we choose \ssf$q$ \ssf according to \eqref{choiceD1}
and \ssf$\tilde{q},p,\tilde{p}$ \ssf
satisfying \eqref{condtildeq} and (\ref{condindecesp}).

In summary, when
\ssf$\gamma_{\text{conf}}\ssb\le\ssb\gamma\ssb<\ssb\gamma_{\infty}$
\ssf and \ssf$\sigma\ssb>\ssb\frac n2\ssb-\ssb\frac2{\gamma-1}$\ssf,
there exists always an admissible couple $(p,q)$ such that 
all conditions \eqref{condts} are satisfied and
\ssf$\sigma\ssb\ge\ssb\frac{n+1}2\ssf(\frac12\ssb-\ssb\frac1q)$\ssf.
\smallskip

This concludes the proof of Theorem \ref{WPL2}.
\end{proof}

\begin{remark}
Notice that, in dimension \,$n\!=\!3$\ssf,
the Strichartz estimates are available in the triangle \,$T_3$ 
without the endpoint (see Remark \ref{dim3end}).
By arguing as above, we prove that
the NLW \eqref{NLWhyperbolic} is locally well--posed
in \ssf$H^{\sigma-\frac12,\frac12}\ssb
\times\ssb H^{\sigma-\frac12,-\frac12}$ if  \,{\rm(}see Figure 6\,{\rm)}
\begin{itemize}
\item$\vphantom{\frac{\sqrt0}{\sqrt0}}$
$\,1\!<\!\gamma\ssb\le\ssb\gamma_1\!=\ssb2$
and \ssf$\sigma\ssb>\ssb0$\,{\rm;}
\item$\vphantom{\frac{\sqrt0}{\sqrt0}}$
$\,2\ssb<\!\gamma\ssb<\!\gamma_{\mathrm{conf}}\ssb=\ssb3$ and
\ssf$\sigma\ssb\ge\ssb C_2(\gamma)\ssb
=\ssb1\!-\ssb\frac1{\gamma-1}$\,{\rm;}
\item$\vphantom{\frac{\sqrt0}{\sqrt0}}$
$\,3\ssb\le\!\gamma\ssb<\!\gamma_3\ssb
=\ssb\frac{11\ssf+\ssf\sqrt{73\ssf}}6$
and \ssf$\sigma\ssb>\ssb C_3(\gamma)\ssb
=\ssb\frac32\ssb-\ssb\frac2{\gamma-1}$\ssf.  
\end{itemize}
\end{remark}

\begin{figure}[ht]\label{LWP3}
\begin{center}
\psfrag{sigma}[c]{$\sigma$}
\psfrag{0}[c]{$0$}
\psfrag{1}[c]{$1$}
\psfrag{1/2}[c]{$\frac12$}
\psfrag{3/2}[c]{$\frac32$}
\psfrag{gamma}[c]{$\gamma$}
\psfrag{gamma1=gamma2=2}[c]{$\gamma_1\!=\ssb\gamma_2\!=\ssb2$}
\psfrag{gammaconf=3}[c]{$\gamma_{\text{conf}}\ssb=\ssb3$}
\psfrag{gammainfty}[c]{$\gamma_\infty\ssb=\ssb\frac{11+\sqrt{73}}6$}
\psfrag{C2}[c]{$C_2$}
\psfrag{C3}[c]{$C_3$}
\hspace{9mm}\includegraphics[width=100mm]{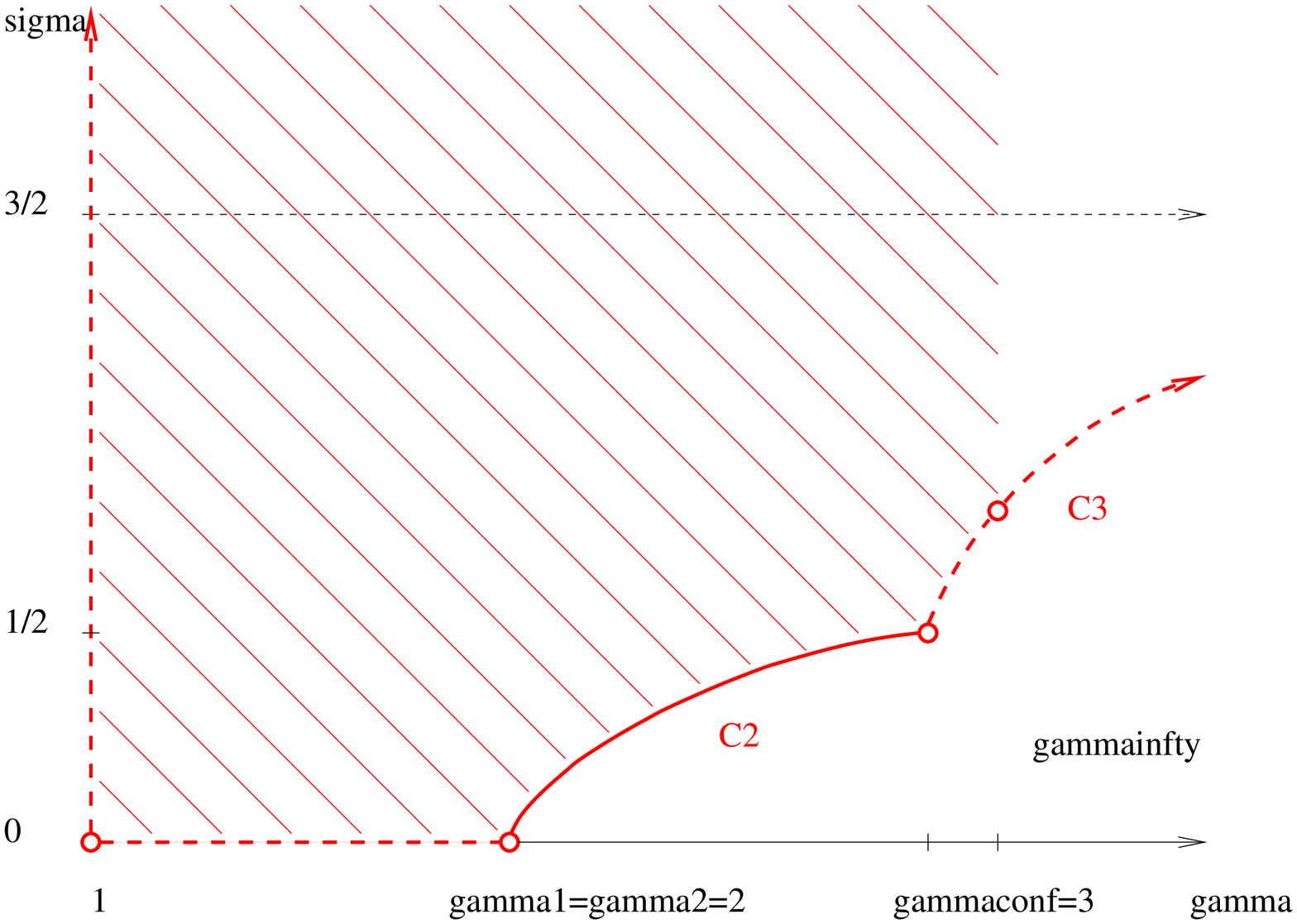}
\end{center}
\caption{Regularity in dimension $n\!=\!3$}
\end{figure}

\pagebreak

\begin{remark} In dimension \,$n\!=\!2$\ssf,
the Strichartz estimates are available in the region \,$T_2$
(see Remark \ref{dim2end}).
By following again the same line of the above proof,
we obtain that the NLW \eqref{NLWhyperbolic} is locally well--posed
in \ssf$H^{\sigma-\frac12,\frac12}\ssb
\times\ssb H^{\sigma-\frac12,-\frac12}$ if
\,{\rm(}see Figure 7\,{\rm)}
\begin{itemize}
\item$\vphantom{\frac{\sqrt0}{\sqrt0}}$
$\,1\!<\!\gamma\ssb\le\ssb2$ and
\ssf$\sigma\!>\!0$\,{\rm;}
\item$\vphantom{\frac{\sqrt0}{\sqrt0}}$
$\,2\ssb\le\!\gamma\ssb\le\ssb3$ and
\ssf$\sigma\!>\ssb\widetilde{C}_1(\gamma)\ssb
=\ssb\frac34\!-\frac32\ssf\frac1\gamma$\,{\rm;}
\item$\vphantom{\frac{\sqrt0}{\sqrt0}}$
$\,3\!<\!\gamma\!<\!\gamma_{\mathrm{conf}}\ssb=\ssb5$ and
\ssf$\sigma\!\ge\ssb C_2(\gamma)\ssb
=\ssb\frac34\!-\!\frac1{\gamma-1}$\,{\rm;}
\item$\vphantom{\frac{\sqrt0}{\sqrt0}}$
$\,5\ssb\le\!\gamma\!<\!\gamma_3\!=\ssb3\ssb+\!\sqrt{6}$ and
\ssf$\sigma\!>\ssb C_3(\gamma)\ssb=\ssb1\!-\!\frac2{\gamma-1}$\,.
\end{itemize}
\end{remark}

\begin{figure}[ht]
\begin{center}
\psfrag{sigma}[c]{$\sigma$}
\psfrag{0}[c]{$0$}
\psfrag{1}[c]{$1$}
\psfrag{1/4}[c]{$\frac14$}
\psfrag{1/2}[c]{$\frac12$}
\psfrag{2}[c]{$2$}
\psfrag{3}[c]{$3$}
\psfrag{gamma}[c]{$\gamma$}
\psfrag{gammaconf}[c]{$\gamma_{\text{conf}}\ssb=\ssb5$}
\psfrag{gammainfty}[c]{$\gamma_\infty\ssb=\ssb3\ssb+\!\sqrt{6\ssf}$}
\psfrag{Ctilde1}[c]{$\widetilde{C}_1$}
\psfrag{C2}[c]{$C_2$}
\psfrag{C3}[c]{$C_3$}
\hspace{9mm}\includegraphics[width=120mm]{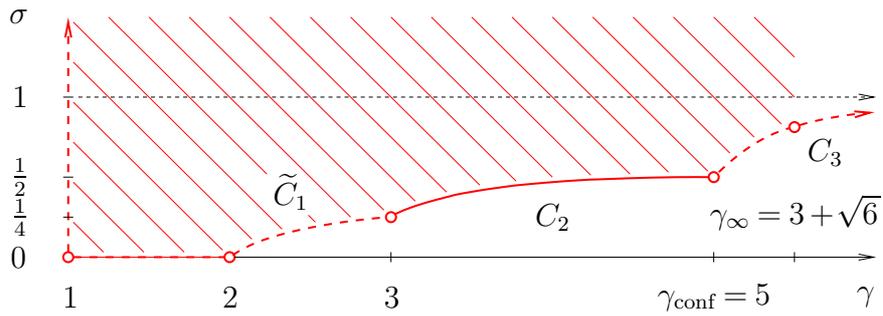}
\end{center}
\caption{Regularity in dimension $n\!=\!2$}
\end{figure}

\section*{Appendix A}\label{AppendixA}

In this appendix we collect some lemmata in Fourier analysis on $\R$
which are used for the kernel analysis in Section \ref{Kernel} and in Appendix C.
\medskip

\noindent
\textbf{Lemma A.1.}\label{LemmaA1}\textit{
Let \ssf$a$ be a compactly supported homogeneous symbol
on \ssf$\R$ of order \ssf$d\!>\!-1\ssf$.
In other words, \ssf$a$ is a smooth function on \,$\R^*$,
whose support is bounded in \ssf$\R$
and which has the following behavior at the origin\,:
\begin{equation*}
\sup_{\lambda\in\R^*}\,|\lambda|^{\,\ell-d}\,
|\,\partial_\lambda^{\,\ell}a(\lambda)\ssf|\,<+\infty
\qquad\forall\;\ell\!\in\!\N\,.
\end{equation*}
Then its Fourier transform
\begin{equation*}
k(x)=\int_{\,0}^{+\infty}\hspace{-1mm}d\lambda\,
a(\lambda)\,e^{\ssf i\ssf\lambda\ssf x}
\end{equation*}
is a smooth function on \,$\R$\ssf, 
with the following behavior at infinity:
\begin{equation*}
k(x)=\mathrm{O}\bigl(\ssf|x|^{-d-1}\ssf\bigr)
\quad\text{as \,}|x|\!\to\!\infty\,.
\end{equation*}
More precisely, let \ssf$N$ be the smallest integer $>d\!+\!1$\ssf.
Then $\exists\;C\!\ge\!0$\ssf, $\forall\,x\!\in\!\R^*$,
\begin{equation*}
|\ssf k(x)\ssf|\,\le\,C\;|x|^{-d-1}\,\sum_{\ell=0}^N\,\sup_{\lambda\in\R^*}\,
(\ssf1\!+\!|\lambda|\ssf)^{\ssf\ell-d}\,|\,\partial_\lambda^{\,\ell}a(\lambda)\ssf|\,.
\end{equation*}}

\begin{proof}
Let us split up
\begin{equation*}
a(\lambda)=\sum\nolimits_{\ssf j=-\infty}^{\ssf+\infty}\,
\chi(2^{-j}\lambda)\,a(\lambda)
\end{equation*}
and \,$k=\sum_{\ssf j=-\infty}^{\ssf+\infty}k_j$ \,accordingly,
using a homogeneous dyadic partition of unity
\begin{equation*}
1=\sum\nolimits_{\ssf j=-\infty}^{\ssf+\infty}\chi(2^{-j}\,\cdot\,)
\end{equation*}
on \ssf$(0,\infty)$\ssf.
Notice that \ssf$a_j$ hence \ssf$k_j$ vanishes for $j$ large,
since $a$ is compactly supported.
By the Leibniz formula, we obtain, for every \ssf$\ell\!\in\!\mathbb N$\ssf,
\begin{equation*}\begin{aligned}
|x|^{\ssf\ell}\,|k_j(x)|\,
&\le\int_{\ssf|\lambda|\ssf\asymp\ssf2^j}\hspace{-1mm}d\lambda\,
\bigl|\ssf\partial_\lambda^{\,\ell}\{\chi(2^{-j}\lambda)\ssf a(\lambda)\}\ssf\bigr|\\
&\lesssim\,\sum\nolimits_{\ssf k=0}^{\,\ell}2^{-k\ssf j}
\int_{\ssf|\lambda|\ssf\asymp\ssf2^j}\hspace{-1mm}d\lambda\,
|\lambda|^{\ssf d-\ell+k}\,
\lesssim\,2^{\,j\ssf(1+d-\ell)}\,.
\end{aligned}\end{equation*}
Let \ssf$N\!\in\!\N^*$ such that \ssf$N\!>\ssb d\!+\!1$\ssf.
Then
\begin{equation*}\begin{aligned}
|\ssf k(x)\ssf|\,
&\le\hspace{-1mm}\sum_{2^j\le|x|^{-1}}\hspace{-1mm}|\ssf k_j(x)\ssf|\,
+\hspace{-1mm}\sum_{2^j\ge|x|^{-1}}\hspace{-1mm}|\ssf k_j(x)\ssf|\\
&\lesssim\hspace{-1mm}\sum_{2^j\le|x|^{-1}}\hspace{-1mm}2^{\,j\ssf(1+d)}\,
+\;|x|^{-N}\hspace{-1.5mm}\sum_{2^j\ge|x|^{-1}}\hspace{-1mm}2^{\,j\ssf(1+d-N)}\,
\lesssim\,|x|^{-d-1}\,.
\end{aligned}\end{equation*}
\end{proof}

\noindent
\textbf{Lemma A.2.}\label{LemmaA2}\textit{
Let \ssf$a$ be an inhomogeneous symbol on \ssf$\R$
of order \ssf$d\!\in\!\R$\ssf.
In other words,
\ssf$a$ is a smooth function on \ssf$\R$
such that
\begin{equation*}
\sup_{\lambda\in\R}\;(\ssf1\!+\!|\lambda|\ssf)^{\ssf\ell-d}\,
|\,\partial_\lambda^{\,\ell}a(\lambda)\ssf|\,<\,+\infty
\qquad\forall\;\ell\!\in\!\N\,.
\end{equation*}
Then its Fourier transform
\begin{equation*}
k(x)\ssf={\displaystyle\int_{-\infty}^{+\infty}}\hskip-1mm
d\lambda\,a(\lambda)\,e^{\ssf i\ssf\lambda\ssf x}
\end{equation*}
is a smooth function on \ssf$\R^*$,
which has the following asymptotic behaviors\,:
\begin{itemize}
\item[(i)] At infinity,
\ssf$k(x)\ssb=\ssb\mathrm{O}\bigl(\ssf|x|^{-\infty}\bigr)$\ssf.
More precisely,
for every \ssf$N\!>\ssf d\!+\!1$\ssf,
there exists \,$C_N\!\ge\ssb0$ such that,
for every \,$x\!\in\!\R^*$,
\begin{equation*}
|\ssf k(x)\ssf|\le C_N\,|x|^{-N}
\sup_{\lambda\in\R}\;(\ssf1\!+\!|\lambda|\ssf)^{N-d}\,
|\,\partial_\lambda^{\ssf N}\ssb a(\lambda)\ssf|\,.
\end{equation*}
\item[(ii)] At the origin,
\begin{equation*}
k(x)=\begin{cases}
\mathrm{O}\ssf(1)
&\text{if \;}d\!<\!-1\ssf,\\
\mathrm{O}\ssf(\ssf\log\frac1{|x|})
&\text{if \;}d\!=\!-1\ssf,
\vphantom{\frac||}\\
\mathrm{O}\ssf(\ssf|x|^{-d-1})
&\text{if \;}d\!>\!-1\ssf.\\
\end{cases}
\end{equation*}
More precisely\,:
\\\hskip-1mm$\circ$
\,If \,$d\!<\!-1$\ssf,
then there exists \,$C\!\ge\!0$ such that,
for every \,$x\!\in\!\R$\ssf,
\begin{equation*}
|\ssf k(x)\ssf|\le C\,\sup_{\ssf\lambda\in\R}\,
(\ssf1\!+\!|\lambda|\ssf)^{-d}\,|\ssf a(\lambda)\ssf|\,.
\end{equation*}
\hskip-1mm$\circ$
\,If \,$d\!=\!-1$\ssf,
then there exists \,$C\!\ge\!0$ such that,
for every \,$0\!<\!|x|\!<\!\frac12$\ssf,
\begin{equation*}
|\ssf k(x)\ssf|\ssf\le\ssf C\,\log{\textstyle\frac1{|x|}}\;
\bigl\{\,\sup_{\ssf\lambda\in\R}\;
(\ssf1\!+\!|\lambda|\ssf)\,|\ssf a(\lambda)\ssf|\ssf
+\,\sup_{\lambda\in\R}\;
(1\!+\!|\lambda|)^2\,|\ssf a'(\lambda)\ssf|\,\bigr\}\,.
\end{equation*}
\hskip-1mm$\circ$
\,If \ssf$d\!>\!-1$\ssf,
let \ssf$N$ be the smallest integer $>d\!+\!1$\ssf.
Then there exists \,$C\!\ge\!0$ \ssf such that,
for every \,$0\!<\!|x|\!<\!1$\ssf,
\begin{equation*}
|\ssf k(x)\ssf|\ssf\le\ssf C\,|x|^{-d-1}\,\sum_{\ell=0}^{N}\;
\sup_{\ssf\lambda\in\R}\;(\ssf1\!+\!|\lambda|\ssf)^{\ssf\ell-d}\,
|\,\partial_\lambda^{\,\ell}a(\lambda)\ssf|\,.
\end{equation*}
\item[(iii)] Similar estimates hold for the derivatives
\begin{equation*}
\partial_{\ssf x}^{\,\ell}\,k(x)\ssf=
{\displaystyle\int_{-\infty}^{+\infty}}\hskip-1mm
d\lambda\,(i\ssf\lambda)^\ell\,a(\lambda)\,e^{\ssf i\ssf\lambda\ssf x}
\end{equation*}
which correspond to symbols
\,$a_\ell(\lambda)\ssb=\ssb(i\ssf\lambda)^\ell\,a(\lambda)$
of order \ssf$d\!+\!\ell$\ssf.
\end{itemize}}

\begin{proof}
(i) Since \ssf$k$ \ssf is the Fourier transform of \ssf$a$\ssf,
then \ssf$x^Nk(x)$ is the Fourier transform of
$(i\ssf\partial_{\lambda})^Na(\lambda)$\ssf,
which is \ssf$\text{O}\big((1\!+\!|\lambda|)^{d-N}\big)$\ssf,
hence integrable when \ssf$N\!>\!d\!+\!1$\ssf.  

\noindent
(ii) If \,$d\!<\!-1$\ssf, we simply estimate\,:
\begin{equation*}
|k(x)|\ssf\le\int_{-\infty}^{+\infty}\hspace{-1mm}d\lambda\,|a(\lambda)|\,
\le\,\sup_{\lambda\in\R}\,(\ssf1\!+\!|\lambda|\ssf)^{-d}\,|a(\lambda)|\,
\int_{-\infty}^{+\infty}d\lambda\,(\ssf1\!+\!|\lambda|\ssf)^{\ssf d}\,.
\end{equation*} 
If \,$d\!\ge\!-1$\ssf, we split up
\begin{equation*}
k(x)\ssf=\underbrace{\int_{-\infty}^{+\infty}\hspace{-1mm}d\lambda\,
\chi_0(|x|\lambda)\,a(\lambda)\,e^{\,i\ssf\lambda\ssf x}}_{k_0(x)}\,
+\ssf\underbrace{\int_{-\infty}^{+\infty}\hspace{-1mm}d\lambda\,
\chi_\infty(|x|\lambda)\,a(\lambda)\,e^{\,i\ssf\lambda\ssf x}}_{k_\infty(x)}\,,
\end{equation*}
using smooth cut--off functions $\chi_0$ and $\chi_\infty$ on $[0,+\infty)$
such that $1\!=\ssb\chi_0\ssb+\ssb\chi_\infty$\ssf,
$\chi_0\ssb=\!1$ on $[\ssf0,1\ssf]$ and
$\chi_\infty\!=\!1$ on $[\ssf2,+\infty)$\ssf.
The first integral is estimated as above\,:
\begin{equation*}
\begin{aligned}
|\ssf k_0(x)|\ssf
&\le\int_{\ssf|\lambda|\le\ssf2\ssf|x|^{-1}}\hspace{-1mm}d\lambda\,|a(\lambda)|\\
&\le\,2\,\sup_{\lambda\in\R}\,(1+|\lambda|)^{-d}\,|a(\lambda)|\,
\int_{\,0}^{\,2\ssf|x|^{-1}}\hspace{-1.5mm}d\lambda\,(1\!+\!\lambda)^{\ssf d}\\
&\lesssim\,
\sup_{\lambda\in\R}\,(1+|\lambda|)^{-d}\,|a(\lambda)|\,
\begin{cases}
\,\log\frac1{|x|}
&\text{if \,}d\!=\!-1\ssf,\\
\,|x|^{-d-1}
&\text{if \,}d\!>\!-1\ssf.\\
\end{cases}
\end{aligned}
\end{equation*}
After $N$ integrations by parts,
the second integral becomes
\begin{equation*}\textstyle
k_\infty(x)=\bigl(\frac ix\bigr)^N
{\displaystyle\int_{-\infty}^{+\infty}}\hspace{-1mm}d\lambda\,
\bigl(\frac\partial{\partial\lambda}\bigr)^N
\bigl\{\ssf\chi_\infty(|x|\lambda)\,a(\lambda)\bigr\}\,
e^{\,i\ssf\lambda\ssf x}\,.
\end{equation*}
Hence
\begin{equation*}\begin{aligned}
|\ssf k_\infty(x)|\,
&\lesssim\,|x|^{-N}\ssf
\int_{\ssf|\lambda|\ssf\ge\ssf|x|^{-1}}\hspace{-1mm}d\lambda\,
|\,\partial_\lambda^{\ssf N}\ssb a(\lambda)\ssf|\\
&+\,\displaystyle\sum\nolimits_{\ssf0<\ell<N}|x|^{-\ell}\ssf
\displaystyle\int_{\ssf|x|^{-1}\le\ssf|\lambda|\ssf\le\ssf2\ssf|x|^{-1}}\hspace{-1mm}d\lambda\,|\,\partial_\lambda^{\ssf\ell}a(\lambda)\ssf|\\
&\lesssim\,|x|^{-d-1}\,\sum_{\ell=1}^{N-1}\,
\sup_{\lambda\in\R}\,(1+|\lambda|)^{\ssf\ell-d}\,|a(\lambda)|\,.
\end{aligned}\end{equation*}
This concludes the proof of (ii). The proof of (iii) is similar and we omit the details. 
\end{proof}

\noindent
\textbf{Lemma A.3.}\label{LemmaA3}\textit{
Assume that
\begin{equation*}\textstyle
a(\lambda)
=\zeta\,\chi_\infty(\lambda)\,\lambda^{-m-1-i\ssf\zeta}
+\ssf b(\lambda)
\end{equation*}
where \,$m\!\in\!\mathbb N$\ssf,
\ssf$\zeta\!\in\!\R$\,,
and \,$b$ \ssf is a symbol of order \,$d\ssb<\!-\ssf m\!-\!1$\ssf.
Then
\begin{equation*}\textstyle
\partial_{\ssf x}^{\ssf m}\ssf k(x)\ssf
={\displaystyle\int_{-\infty}^{+\infty}}\hskip-1mm
d\lambda\,a(\lambda)\,(i\ssf\lambda)^m\,e^{\,i\ssf\lambda\ssf x}
\end{equation*}
is a bounded function at the origin.
More precisely,
there exists \,$C\!\ge\!0$ \ssf such that,
for every \,$0\!<\!|x|\!<\!\frac12$\ssf,
\begin{equation*}
|\,\partial_{\ssf x}^{\ssf m}\ssf k(x)\ssf|
\le C\,\bigl\{\,1+\zeta^2\ssb+\ssf\sup_{\ssf\lambda\in\R}\,
(\ssf1\!+\!|\lambda|\ssf)^{-d}\,|\ssf b(\lambda)\ssf|\,\bigr\}\,.
\end{equation*}}

\begin{proof}
Let us split up
\begin{equation*}\begin{aligned}
\partial_{\ssf x}^{\ssf m}\ssf k(x)\,
&=\,i^{\ssf m}
\overbrace{\int_{\,2}^{\frac1{|x|}}\hskip-1mm
d\lambda\;\zeta\,\lambda^{-1-i\ssf\zeta}\,e^{\,i\ssf\lambda\ssf x}
}^{k_1(x)}\,
+\;i^{\ssf m}\,\zeta\,
\overbrace{\int_{\frac1{|x|}}^{+\infty}\hskip-1mm
d\lambda\,\lambda^{-1-i\ssf\zeta}\,e^{\,i\ssf\lambda\ssf x}
\vphantom{\int^{\frac1{|x|}}}}^{k_2(x)}\\
&+\hskip1mm i^{\ssf m}\,\zeta\underbrace{
\int_{\,1}^{\ssf2}d\lambda\,\chi_\infty(\lambda)\,
\lambda^{-1-i\ssf\zeta}\,e^{\,i\ssf\lambda\ssf x}
}_{k_3(x)}\,
+\;i^{\ssf m}\underbrace{
\int_{-\infty}^{+\infty}\hskip-1mm d\lambda\,
\lambda^m\,b(\lambda)\,e^{\,i\ssf\lambda\ssf x}
}_{k_4(x)}\,.
\end{aligned}\end{equation*}
The first two terms are estimated by integrations by parts.
Specifically,
\begin{equation*}
k_1(x)\,
=\,i\,\lambda^{-i\ssf\zeta}\,e^{\,i\ssf\lambda\ssf x}\,
\Big|_{\lambda=2}^{\lambda=\frac1{|x|}}
+\,x\,\int_{\,2}^{\frac1{|x|}}\!
d\lambda\,\lambda^{-i\ssf\zeta}\,e^{\,i\ssf\lambda\ssf x}\,,
\end{equation*}
with
\;$\Bigl|\,\lambda^{-i\ssf\zeta}\,e^{\,i\ssf\lambda\ssf x}\,
\Big|_{\lambda=2}^{\lambda=\frac1{|x|}}\ssf\Bigr|\le2$
\;and
\;$\Bigl|\,{\displaystyle\int_{\,2}^{\frac1{|x|}}}\!
d\lambda\,\lambda^{-i\ssf\zeta}\,e^{\,i\ssf\lambda\ssf x}\,\Bigr|
\le\frac1{|x|}$\,,
\,while
\begin{equation*}
k_2(x)\,
=\,-\,{\textstyle\frac ix}\,\lambda^{-1-i\ssf\zeta}\,e^{\,i\ssf\lambda\ssf x}\,
\Big|_{\lambda=\frac1{|x|}}^{\lambda=+\infty}
+\,{\textstyle\frac{\zeta-i}x}\,\int_{\frac1{|x|}}^{+\infty}\!
d\lambda\,\lambda^{-2-i\ssf\zeta}\,e^{\,i\ssf\lambda\ssf x}\,,
\end{equation*}
with
\;$\Bigl|\,\lambda^{-1-i\ssf\zeta}\,e^{\,i\ssf\lambda\ssf x}\,
\Big|_{\lambda=\frac1{|x|}}^{+\infty}\ssf\Bigr|\le|x|$
\;and
\;$\Bigl|\,{\displaystyle\int_{\frac1{|x|}}^{+\infty}}\!
d\lambda\,\lambda^{-2-i\ssf\zeta}\,e^{\,i\ssf\lambda\ssf x}\,\Bigr|
\le|x|$\,.
The last two terms are easy to estimate.
Obviously \,$|\ssf k_3(x)\ssf|\le1$\ssf,
\ssf while
\begin{equation*}
|\ssf k_4(x)\ssf|\ssf\le\,\sup_{\ssf\lambda\in\R}\;
(\ssf1\!+\!|\lambda|\ssf)^{-d}\,|\ssf b(\lambda)\ssf|\,
\underbrace{\int_{-\infty}^{+\infty}\hspace{-1mm}d\lambda\,
(\ssf1\!+\!|\lambda|\ssf)^{\ssf m+d}}_{<\,+\ssf\infty}\,.
\end{equation*}
We conclude by summing up these four estimates.
\end{proof}

\section*{Appendix B}\label{AppendixB}

In this appendix we collect some properties of the Riesz distributions.
We refer to \cite[ch.\;1, \S\;3 \& ch.\;2, \S\;2]{GC}
or \cite[ch.\;III, \S\;3.2]{Ho1} for more details.
The Riesz distribution $R_{\,z}^{\ssf+}$ is defined by
\begin{equation}\label{Riesz}
\textstyle
\langle\ssf R_{\,z}^{\ssf+},\varphi\ssf\rangle\ssf
=\,\frac1{\Gamma(z)}\,
{\displaystyle\int_{\,0}^{+\infty}}\hspace{-1mm}
d\lambda\;\lambda^{z-1}\,\varphi(\lambda)
\end{equation}
when \ssf$\Re z\!>\!0$\ssf.
It extends to a holomorphic family
\ssf$\{\ssf R_{\,z}^{\ssf+}\}_{z\in\C}$
\ssf of tempered distributions on \ssf$\R$
\ssf which satisfy the following properties\,:
\begin{itemize}

\item[(i)]
\,$\lambda\ssf R_{\,z}^{\ssf+}\ssb=z\ssf R_{\ssf z+1}^{\,+}$
\;$\forall\;z\!\in\ssb\C$\ssf,

\item[(ii)]
\,$(\frac d{d\lambda})\ssf R_{\,z}^{\ssf+}\ssb=R_{\ssf z-1}^{\,+}$
\;$\forall\;z\!\in\ssb\C$\ssf,

\item[(iii)]
\,$R_{\,0}^{\ssf+}\!=\delta_{\ssf0}$
\,and more generally
\,$R_{-m}^{\ssf+}\ssb=(\frac d{d\lambda})^m\ssf\delta_{\ssf0}$
\;$\forall\;m\!\in\!\N$\ssf,

\item[(iv)]
\,$R_{\ssf z+z'}^{\,+}\ssb=R_{\,z}^{\ssf+}\!*\ssb R_{\,z'}^{\ssf+}$
\;$\forall\;z,z'\!\in\ssb\C$\ssf.
\end{itemize}
Hence
\begin{equation*}\textstyle
\langle\ssf R_{\,z}^{\ssf+},\varphi\ssf\rangle\ssf
=\ssf\langle\ssf(\frac d{d\lambda})^m\ssf R_{\ssf z+m}^{\,+}\ssf,
\varphi\ssf\rangle\ssf
=\,\frac{(-1)^m}{\Gamma(z+m)}\,
{\displaystyle\int_{\,0}^{+\infty}}\hspace{-1mm}
d\lambda\;\lambda^{z+m-1}\,
\bigl(\frac d{d\lambda}\bigr)^m\varphi(\lambda)
\end{equation*}
when \ssf$\Re z\!>\!-m$\ssf.
The Riesz distribution
\ssf$R_{\,z}^{\ssf-}\!=\ssb(R_{\,z}^{\ssf+})^{\vee}$
is defined similarly.
Their Fourier transforms are given by
\begin{itemize}

\item[(v)]
\,$\mathcal{F}R_{\,z}^{\ssf\pm}\ssb
=e^{\ssf\pm\ssf i\frac\pi2z}\,
(\ssf x\ssb\pm\ssb i\ssf0\ssf)^{-z}$
\;$\forall\;z\!\in\!\C$\ssf,
\end{itemize}
\vspace{.5mm}

\noindent
where
\begin{equation*}\textstyle
\langle\ssf(x\ssb\pm\ssb i\ssf0)^z,\varphi\ssf\rangle\ssf
=\ssf\lim_{\ssf\varepsilon\searrow0}{\displaystyle\int_{\ssf\R}}
dx\;(x\ssb\pm\ssb i\ssf\varepsilon)^z\,\varphi(x)
\end{equation*}
when \,$\Re z\!>\!-1$ and 
\begin{equation*}\textstyle
(x\ssb\pm\ssb i\ssf0)^z
=\ssf\Gamma(z\!+\!1)\,\{\ssf R_{z+1}^{\,+}\hspace{-.5mm}
+\ssb e^{\ssf\pm\ssf i\ssf\pi z}\ssf R_{z+1}^{\,-}\ssf\}
\end{equation*}
in general
(notice that there are actually no singularities in the last expression).

\section*{Appendix C}\label{AppendixC}

In this appendix we prove the local kernel estimates
\begin{equation}
|\,\widetilde{w}_{\,t}^{\ssf\infty}(r)\ssf|\ssf
\lesssim\ssf\begin{cases}
\,|t|^{-\frac{n-1}2}
&\text{if \;}n\ssb\ge\ssb3\\
\,|t|^{-\frac12}\ssf(\ssf1\!-\ssb\log|t|\ssf)
&\text{if \;}n\ssb=\ssb2\\
\end{cases}\end{equation}
stated in Theorem \ref{Estimatewtildetinfty}.i.a
under the assumptions \ssf$0\!<\!|t|\!\le\!2$\ssf, $0\!\le\!r\!\le\!3$
\ssf and $\Re\sigma\!=\!\frac{n+1}2$.
By symmetry, we may assume again that \ssf$t\!>\!0$\ssf.
\smallskip

\noindent$\bullet$
\,\textit{Case 1}\,:
\,Assume that \ssf$r\!\le\!\frac t2$\ssf.
\smallskip

\noindent
By using the first integral representation
of the spherical functions in \eqref{intrepr},
we obtain
\begin{equation}\label{WaveKernel1}\textstyle
\widetilde{w}_{\,t}^{\ssf\infty}(r)
=\frac{e^{\ssf\sigma^2}}{\Gamma(-\ssf i\Im\sigma)}\,
{\displaystyle\int_K}dk\;e^{-\rho\,\text{H}(a_{-r}k)}
{\displaystyle\int_{\,1}^{\ssf\infty}}\!d\lambda\,
\chi_\infty(\lambda)\,a(\lambda)\,
e^{\,i\ssf\lambda\ssf\{\ssf t-\text{H}(a_{-r}k)\}}\,,
\end{equation} 
where
\begin{equation*}\textstyle
a(\lambda)\ssf=\,
|\mathbf{c}\hspace{.1mm}(\lambda)|^{-2}\,\lambda^{-\tau}\,
(\lambda^2\hskip-.75mm+\!{\tilde\rho}^{\ssf2})^{\frac\tau2-\frac\sigma2}\,.
\end{equation*}
According to Lemma A.2 in Appendix A,
since \ssf$\chi_\infty\ssf a$ \ssf is a symbol of order \ssf$\frac{n-3}2$
\ssf and
\begin{equation*}\textstyle
|\,t\ssb-\ssb\text{H}(a_rk)\ssf|\ssf\ge\ssf t\ssb-\ssb r\ssf\ge\frac t2\,,
\end{equation*}
the inner integral in \eqref{WaveKernel1} is
\begin{equation*}
\text{O}\ssf\bigl(\,|\sigma|^N\,|\,
t\ssb-\ssb\text{H}(a_{-r}k)|^{-\frac{n-1}2}\ssf\bigr)
=\ssf \text{O}\ssf\bigl(\,|\sigma|^N\,t^{\ssf-\frac{n-1}2}\ssf\bigr)\,,
\end{equation*}
where \ssf$N$ is the smallest integer $>\frac{n-1}2$\ssf.
Hence
\begin{equation*}
|\,\widetilde{w}_{\,t}^{\ssf\infty}(r)\ssf|\,
\lesssim\,t^{\ssf-\frac{n-1}2}\,.
\end{equation*}
 
\noindent$\bullet$
\textit{Case 2}\,:
\,Assume that \ssf$\frac t2\!<\!r\!<\!t$\ssf.
\smallskip

\noindent
By using the third integral formula for spherical functions in \eqref{intrepr},
we are lead to estimate the expression
\begin{equation}\label{WaveKernel2}
(\sinh r)^{2-n}
{\displaystyle\int_{-r}^{+r}}\!du\,
(\cosh r\!-\ssb\cosh u)^{\frac{n-3}2}
{\displaystyle\int_{\,1}^{\ssf\infty}}\!d\lambda\,
\chi_\infty(\lambda)\,a(\lambda)\,e^{\,i\ssf\lambda\ssf(t-u)}\,.
\end{equation}
Let us expand
\begin{equation*}
|\mathbf{c}\hspace{.1mm}(\lambda)|^{-2}\,\lambda^{-\tau}\,
(\lambda^2\hskip-.75mm+\!{\tilde\rho}^{\ssf2})^{\frac{\tau-\sigma}2}
=\ssf\const\ssf\lambda^{\frac{n-3}2\ssf-\,i\ssf\Im\sigma}
+\text{O}\bigl(\ssf|\sigma|\ssf\lambda^{\frac{n-5}2}\bigr)\,,
\end{equation*}
\vspace{1mm}
as \ssf$\lambda\!\to\!+\infty$\ssf,
and
\vspace{-5mm}
\begin{equation*}
a(\lambda)
=\ssf\overbrace{
\const\ssf\lambda^{\frac{n-3}2-\,i\ssf\Im\sigma}
}^{\widetilde{a}\ssf(\lambda)}
+\,b(\lambda)
\end{equation*}
accordingly.
Since \ssf$\chi_\infty\ssf b$ \ssf is a symbol of order \ssf$\frac{n-5}2$\ssf,
its contribution to \eqref{WaveKernel2} can be estimated by
\begin{equation}\label{Estimate1}
|\sigma|^N\,(\sinh r)^{2-n}
{\displaystyle\int_{-r}^{+r}}\!du\,
(\cosh r\!-\ssb\cosh u)^{\frac{n-3}2}
\,(t\!-\!u)^{-\frac{n-3}2}\,.
\end{equation}
Here we have applied again Lemma A.2
and \ssf$N$ is the smallest integer $>\frac{n-1}2$\ssf.
By using
\begin{equation*}\begin{cases}
\;\sinh r\asymp r\,,\\
\;\cosh r\!-\ssb\cosh u
=2\,\sinh\frac{r-u}2\,\sinh\frac{r+u}2
\asymp(r\!-\!u)\,(r\!+\!u)\,,\\
\;t\!-\!u\ge r\!-\!u\,,\\
\end{cases}\end{equation*}
we end up with the estimate
\begin{equation*}\textstyle
|\sigma|^N\,r^{\ssf2-n}
{\displaystyle\int_{-r}^{+r}}\!du\,
(r\!+\!u)^{\frac{n-3}2}\ssf
\lesssim\,|\sigma|^N\,r^{-\frac{n-3}2}\ssf
\asymp\,|\sigma|^N\,t^{-\frac{n-3}2}\,.
\end{equation*}
Notice that the previous computations are valid in dimension \ssf$n\!>\!3$\ssf.
In dimension \ssf$n\!=\!3\ssf$,
the last estimate becomes
\,$|\sigma|^2\,(\ssf1\!-\ssb\log t\ssf)$
\,while, in dimension \ssf$n\!=\!2$\ssf,
\eqref{Estimate1} is replaced by
\begin{equation*}\textstyle
|\sigma|\,{\displaystyle\int_{-r}^{+r}}\!\frac{du}{\sqrt{\cosh r-\cosh u}}\,
\asymp\,|\sigma|\,{\displaystyle\int_{-r}^{+r}}
\hspace{-1mm}\frac{du}{\sqrt{r^2-u^2}}\,
\asymp\,|\sigma|\,.
\end{equation*}
Similarly
\begin{equation*}
\int_{\,0}^{\,2}\!d\lambda\,\chi_0(\lambda)\,
\widetilde a(\lambda)\,e^{\,i\ssf\lambda\ssf(t-u)}
\end{equation*}
yields a bounded contribution to \eqref{WaveKernel2}.
Let us eventually analyze the remaining contribution of
\begin{equation}\label{distribution}
\int_{\,0}^{+\infty}\hspace{-1mm}d\lambda\,
\lambda^{\frac{n-3}2\ssf-\ssf i\ssf\Im\sigma}\,e^{\,i\ssf\lambda\ssf(t-u)}\,,
\end{equation}
which is a classical distribution.
According to the properties of the Riesz distributions (\ref{Riesz}) in Appendix B,
we have indeed
\begin{equation*}\textstyle
{\displaystyle\int_{\,0}^{+\infty}}\hspace{-1mm}d\lambda\;
\lambda^{\frac{n-3}2\ssf-\ssf i\ssf\Im\sigma}\,e^{\,i\ssf\lambda\ssf(t-u)}
=\Gamma({\frac{n-1}2\!-\ssb i\ssf\Im\sigma})\;
e^{\,\frac\pi2\Im\sigma\,+\,i\ssf\frac\pi4(n-1)}\,
(t\!-\!u)^{-\frac{n-1}2+\,i\Im\sigma}
\end{equation*}
and it remains for us to estimate the expression
\begin{equation}\label{WaveKernel3}\textstyle
\frac{\Gamma(\frac{n-1}2\ssf-\,i\Im\sigma)}{\Gamma(-\ssf i\Im\sigma)}\,
(\sinh r)^{2-n}{\displaystyle\int_{-r}^{+r}}\!du\,
(\cosh r\!-\ssb\cosh u\ssf)^{\frac{n-3}2}\,
(t\!-\!u)^{-\frac{n-1}2+\,i\Im\sigma}\,.
\end{equation}
In order to do so, we discuss separately the odd and even--dimensional cases.
\smallskip

\noindent$\circ$
\textit{Subcase 2.a}\,:
\,Assume that \ssf$n\!=\!2\ssf m\!+\!1$ \ssf is odd.
\smallskip

\noindent
After \ssf$m\!-\!1$ integrations by parts,
\eqref{WaveKernel3} becomes
\begin{equation*}\textstyle
(-\,i\Im\sigma)\,(\sinh r)^{1-2\ssf m}
{\displaystyle\int_{-r}^{+r}}\!du\,
(t\!-\!u)^{-1+i\Im\sigma}\;
{\displaystyle\sum\nolimits_{\ssf j=1}^{m-1}}
a_j(u)\,(\cosh r\!-\ssb\cosh u\ssf)^{\ssf m-j-1}\,,
\end{equation*}
where \ssf$a_j(u)$ is a linear combination
of monomials $(\sinh u)^{\ssf j'}\ssb(\cosh u)^{\ssf j''}$
\!with $j'\ssb,j''\hspace{-1mm}\ge\!0$\ssf,
\ssf$j'\!+\ssb j''\!=\ssb j$
\ssf and \ssf$j'\!\ge\ssb2\ssf j\!+\!1\!-\!m$\ssf.
In particular \ssf$a_{\ssf m\ssf-1}(u)\!=\!(m\!-\!1)\ssf!\,(\sinh u)^{m-1}$.
After one more integration by parts, we get
\begin{equation*}\begin{aligned}
&\textstyle
(m\!-\!1)\ssf!\,(\sinh r)^{-m}
\bigl\{(t\!-\!r)^{\ssf i\Im\sigma}\!
+\ssb(-1)^m(t\!+\!r)^{\ssf i\Im\sigma}\bigr\}\;+\\
&\textstyle+\ssf
(\sinh r)^{1-2\ssf m}
{\displaystyle\int_{-r}^{+r}}\!du\,
(t\!-\!u)^{\ssf i\Im\sigma}\;
{\displaystyle\sum\nolimits_{\ssf j=1}^{m-1}}\,
\widetilde a_j(u)\,(\cosh r\!-\ssb\cosh u\ssf)^{\ssf m-j-1}\,,
\end{aligned}\end{equation*}
where \ssf$\widetilde a_j(u)\ssb
=\ssb\text{O}(\ssf r^{\ssf\max\{0,\ssf2j-m\}})$
\ssf and \ssf$\cosh r\!-\ssb\cosh u
\ssb=\ssb2\ssf\sinh\frac{r-u}2\ssf\sinh\frac{r+u}2\ssb
=\ssb\text{O}(r^2)$\ssf,
hence the last sum is \ssf$\text{O}(r^{\ssf m-2})$
\ssf and the last integral is \ssf$\text{O}(r^{\ssf m-1})$\ssf.
Notice that these terms vanish when \ssf$m\!=\!1$\ssf.
Thus \eqref{WaveKernel3} is
\ssf$\text{O}\ssf(\ssf r^{-m})\ssb
=\ssb\text{O}\ssf(\ssf t^{-\frac{n-1}2})$\ssf,
when \ssf$n\!=\!2\ssf m\!+\!1$ \ssf is odd.
\smallskip

\noindent$\circ$
\textit{Subcase 2.b}\,:
Assume that \ssf$n\!=\!2\ssf m$ \ssf is even.
\smallskip

\noindent
After $m\!-\!1$ integrations by parts,
\eqref{WaveKernel3} becomes this time
\begin{equation}\label{WaveKernel4}\textstyle
\frac{\Gamma(\frac12-\ssf i\Im\sigma)}{\Gamma(-\ssf i\Im\sigma)}\,
(\sinh r)^{2-2\ssf m}
{\displaystyle\int_{-r}^{+r}}\!du\,
(t\!-\!u)^{-\frac12+\ssf i\Im\sigma}\;
{\displaystyle\sum\nolimits_{\ssf j=1}^{m-1}}
a_j(u)\,(\cosh r\!-\ssb\cosh u\ssf)^{\ssf m-j-\frac32}\,,
\end{equation}
where \ssf$a_{\ssf m\ssf-1}(u)\ssb
=\ssb\frac{\Gamma(m-\frac12)}{\sqrt{\pi\ssf}}\,(\sinh u)^{m-1}$
and the other \ssf$a_j(u)$ are as before.
Since
\begin{equation*}\begin{aligned}
\textstyle
\frac{\Gamma(\frac12-\ssf i\Im\sigma)}
{\Gamma(-\ssf i\Im\sigma)}\,
&=\,\text{O}\ssf(\ssf|\sigma|^{\frac12})\,,\\
a_j(u)\,(\cosh r\!-\ssb\cosh u\ssf)^{\ssf m-j-\frac32}\,
&=\,\text{O}\ssf(\ssf r^{\ssf m-2})
\qquad\forall\;1\!\le\!j\!\le\!m\!-\!2\,,\\
{\displaystyle\int_{-r}^{+r}}\!du\,(t\!-\!u)^{-\frac12}
&\textstyle
\asymp\frac r{\sqrt{\ssf t\,}}\ssf
\asymp\sqrt{\ssf t\,}\ssf,
\end{aligned}\end{equation*}
the \ssf$m\!-\!2$ \ssf first terms in \eqref{WaveKernel4} are
\,$\text{O}\ssf\bigl(\,|\sigma|^{\frac12}\,t^{-\frac{n-1}2}\bigr)$\ssf.
Let us turn to the last term
\begin{equation}\label{WaveKernel5}\begin{aligned}
&\textstyle
\frac{\Gamma(m-\frac12)}{\sqrt{\pi}}\,
\frac{\Gamma(\frac12-\ssf i\Im\sigma)}{\Gamma(-\ssf i\Im\sigma)}\,
(\sinh r)^{2-2\ssf m}\\
&\times{\displaystyle\int_{-r}^{+r}}\!du\,
(t\!-\!u)^{-\frac12+\ssf i\Im\sigma}\,
(\sinh u)^{m-1}\,(\cosh r\!-\ssb\cosh u\ssf)^{-\frac12}\,,
\end{aligned}\end{equation}
which is obtained by taking \,$j\!=\!m\ssb-\!1$ in \eqref{WaveKernel4}.
Let us split the integral in (\ref{WaveKernel5}) as follows\,:
\begin{equation}\label{Integral1}
\int_{-r}^{\ssf r}\,
=\,\int_{-r}^{\ssf0}
+\,\int_{\,0}^{\ssf2r-t}\hspace{-1mm}
+\,\int_{\ssf2r-t}^{\,r}\,.
\end{equation}
Notice that our current assumption \,$\frac t2\!<\!r\!<\!t$
\ssf implies that \ssf$0\!<\!2\ssf r\!-\!t\!<\!r$\ssf.
Since
\begin{equation*}\textstyle
\cosh r\!-\ssb\cosh u\,
=\,2\,\sinh\frac{r-u}2\,\sinh\frac{r+u}2\,
\asymp\,(r\!-\!u)\ssf(r\!+\!u)\,,
\end{equation*}
the contribution to \eqref{WaveKernel5}
of the first integral in \eqref{Integral1}
\ssf can be estimated by
\begin{equation*}\textstyle
|\sigma|^{\frac12}\,t^{-\frac12}\,r^{\frac12-m}
{\displaystyle\int_{-r}^{\ssf0}}\frac{du}{\sqrt{\ssf r\ssf+\ssf u\ssf}}\,
\asymp\,|\sigma|^{\frac12}\,t^{-\frac{n-1}2}
\end{equation*}
and the contribution to (\ref{WaveKernel5})
of the last integral in \eqref{Integral1} by 
\begin{equation*}\textstyle
|\sigma|^{\frac12}\,(t\!-\!r)^{-\frac12}\,r^{\frac12-m}
{\displaystyle\int_{2\ssf r-t}^{\,r}}\frac{du}{\sqrt{\ssf r\ssf-\ssf u\ssf}}\,
\asymp\,|\sigma|^{\frac12}\,t^{-\frac{n-1}2}\,.
\end{equation*}
We handle the remaining integral
by performing the change of variables
\begin{equation*}\textstyle
v=\frac{t\ssf-\ssf r}{t\ssf-\ssf u}
\quad\Longleftrightarrow\quad
u=t\ssb-\ssb\frac{t\ssf-\ssf r}v
\end{equation*}
and by integrating by parts
\vspace{-4mm}
\begin{equation*}\begin{aligned}
&\textstyle
\frac{\Gamma(\frac12-\ssf i\Im\sigma)}{\Gamma(-\ssf i\Im\sigma)}\,
{\displaystyle\int_{\,0}^{\ssf2\ssf r-t}}\!du\;
(t\!-\!u)^{-\frac12+\ssf i\Im\sigma}\,(r\!-\!u)^{-\frac12}\,
\overbrace{\textstyle
\Bigl(\frac{\sinh\frac{r-u}2}{\frac{r-u}2}\Bigr)^{-\frac12}\,
\bigl(\ssf\sinh\frac{r+u}2\bigr)^{-\frac12}\,
(\sinh u)^{m-1}}^{A(r,u)}\\
&=\textstyle
\frac{\Gamma(\frac12-\ssf i\Im\sigma)}{\Gamma(-\ssf i\Im\sigma)}\,
(t\!-\!r)^{\ssf i\Im\sigma}
{\displaystyle\int_{1-\frac rt}^{\frac12}}dv\;
v^{-1-\ssf i\Im\sigma}\,(1\!-\!v)^{-\frac12}\,
A(r,t\!-\!\frac{t-r}v)\\
&=\textstyle
\frac{\Gamma(\frac12-\ssf i\Im\sigma)}{\Gamma(1-\ssf i\Im\sigma)}\,
(t\!-\!r)^{\ssf i\Im\sigma}\,
v^{-\ssf i\Im\sigma}\,(1\!-\!v)^{-\frac12}\,
A(r,t\!-\!\frac{t-r}v)\,
\Big|_{\ssf v\ssf=\ssf1-\frac rt}^{\,v\ssf=\ssf\frac12}\\
&\textstyle
-\frac{\Gamma(\frac12-\ssf i\Im\sigma)}{2\,\Gamma(1-\ssf i\Im\sigma)}\,
(t\!-\!r)^{\ssf i\Im\sigma}
{\displaystyle\int_{1-\frac rt}^{\frac12}}dv\;
v^{-\ssf i\Im\sigma}\,(1\!-\!v)^{-\frac32}\,
A(r,t\!-\!\frac{t-r}v)\\
&\textstyle
-\frac{\Gamma(\frac12-\ssf i\Im\sigma)}{\Gamma(1-\ssf i\Im\sigma)}\,
(t\!-\!r)^{\ssf1+\ssf i\Im\sigma}
{\displaystyle\int_{1-\frac rt}^{\frac12}}dv\;
v^{-2-\ssf i\Im\sigma}\,(1\!-\!v)^{-\frac12}\,
\partial_{\ssf2}A(r,t\!-\!\frac{t-r}v)\,.
\end{aligned}\end{equation*}
All resulting expressions are
\ssf$\text{O}\ssf\bigl(\,|\sigma|^{-\frac12}\,t^{\ssf m-\frac32}\ssf\bigr)$\ssf,
since
\begin{equation*}\textstyle
\frac{\Gamma(\frac12-\ssf i\Im\sigma)}{\Gamma(1-\ssf i\Im\sigma)}
=\text{O}\bigl(\ssf|\sigma|^{-\frac12}\bigr)\ssf,\quad
A(r,u)=\text{O}\bigl(\ssf t^{\ssf m-\frac32}\bigr)
\quad\text{and}\quad
\partial_{\ssf2}A(r,u)=\text{O}\bigl(\ssf t^{\ssf m-\frac52}\bigr)\ssf.
\end{equation*}
Thus \eqref{WaveKernel5}
and hence \eqref{WaveKernel4}, \eqref{WaveKernel3} are 
\ssf$\text{O}\ssf\bigl(\ssf|\sigma|^{\frac12}\,t^{-\frac{n-1}2}\bigr)$\ssf.
\vspace{2mm}

As a conclusion,
we have obtained the following estimate
in all dimensions \ssf$n\!\ge\!2$\,:
\begin{equation*}\textstyle
|\,\widetilde{w}_{\,t}^{\ssf\infty}(r)\ssf|\,
\lesssim\,t^{\ssf-\frac{n-1}2}
\quad\text{when}\quad
\frac t2\!<\!r\!<\!t\,.
\end{equation*}
\smallskip

\noindent$\bullet$
\textit{Case 3}\,:
\,Assume that \ssf$r\!>\!t$\ssf.
\smallskip

\noindent
In this case we estimate \ssf$\widetilde{w}_t(r)$
using the inverse Abel transform.
More precisely,
we apply the inversion formulae \eqref{inv1} and \eqref{inv2}
to the Euclidean Fourier transform
\begin{equation*}\textstyle
\widetilde{g}_{\,t}^{\ssf\infty}(r)\ssf
=\ssf\frac{e^{\ssf\sigma^2}}{\Gamma(-\ssf i\Im\sigma)}\,
{\displaystyle\int_{\,1}^{+\infty}}\hskip-1mm
d\lambda\,\chi_\infty(\lambda)\,|\mathbf{c}\hspace{.1mm}(\lambda)|^{-2}\,\lambda^{-\tau}\,
(\lambda^2\hskip-.75mm+\!{\tilde\rho}^{\ssf2})^{\frac{\tau-\sigma}2}\,
e^{\,i\ssf t\ssf\lambda}\,\cos\lambda\ssf r\,.
\end{equation*}

\noindent$\circ$
\textit{Subcase 3.a}\,:
\,Assume that \ssf$n\!=\!2\ssf m\!+\!1$ \ssf is odd.
\smallskip

\noindent
Then, up to a multiplicative constant,
\begin{equation*}\textstyle
\widetilde{w}_{\,t}^{\ssf\infty}(r)
=\bigl(\frac1{\sinh r}\ssf\frac\partial{\partial r}\bigr)^{\ssb m}\,
\widetilde{g}_{\,t}^{\ssf\infty}(r)\,.
\end{equation*}
Let us expand
\vspace{-6mm}
\begin{equation}\label{expansion1}\textstyle
\bigl(\ssb\overbrace{\textstyle\frac1{\sinh r}}^{\frac r{\sinh r}\frac1r}
\hskip-1mm\frac\partial{\partial r}\bigr)^{\ssb m}
=\ssf{\displaystyle\sum\nolimits_{\ell=1}^{\,m}}
\alpha_{\ssf\ell}^{\ssf 0}(r)\,\bigl(\ssf\frac1r\frac\partial{\partial r}\bigr)^\ell
\end{equation}
and furthermore
\begin{equation}\label{expansion2}\textstyle
\bigl(\ssf\frac1r\frac\partial{\partial r}\bigr)^\ell
={\displaystyle\sum\nolimits_{k=1}^{\,\ell}}\ssf
\beta_{\ell,k}\,r^{\ssf k-2\ssf\ell}\,\bigl(\frac\partial{\partial r}\bigr)^k.
\end{equation}
The coefficients \ssf$\beta_{\ell,k}$ \ssf in \eqref{expansion2} are constants,
while the coefficients \ssf$\alpha_{\ssf\ell}^{\ssf 0}(r)$ \ssf in \eqref{expansion1}
are smooth functions on \ssf$\R$\ssf,
which are linear combinations of products
\begin{equation*}\textstyle
\bigl(\frac r{\sinh r}\bigr)\times
\bigl(\ssf\frac1r\frac\partial{\partial r}\bigr)^{\ell_2}\bigl(\frac r{\sinh r}\bigr)
\times\,\cdots\,\times
\bigl(\ssf\frac1r\frac\partial{\partial r}\bigr)^{\ell_m}\bigl(\frac r{\sinh r}\bigr)
\end{equation*}
with \,$\ell_2\!+{\dots}+\ssb\ell_m\!=\ssb m\!-\!\ell$\ssf.
Consider first
\begin{equation}\label{expression1}\textstyle
\frac{e^{\ssf\sigma^2}}{\Gamma(-\ssf i\Im\sigma)}\,
{\displaystyle\int_{\,1}^{\ssf\frac6r}}
d\lambda\;\chi_\infty(\lambda)\,\lambda^{-\tau}\,
(\lambda^2\!+\!\tilde\rho^{\ssf2})^{\frac{\tau-\sigma}2}\,
e^{\,i\ssf t\ssf\lambda}\,
\bigl(\ssf\frac1r\frac\partial{\partial r}\bigr)^\ell
\cos\lambda\ssf r\,.
\end{equation}
Since
\,$\chi_\infty(\lambda)\ssf\lambda^{-\tau}\ssf
(\lambda^2\!+\ssb\tilde\rho^{\ssf2})^{\frac{\tau-\sigma}2}\ssb
e^{\,i\ssf t\ssf\lambda}
=\text{O}\ssf(\lambda^{-m-1})$
\,according to the assumption \,$\Re\sigma\!=\!m\!+\!1$
\,and
\,$\bigl(\frac1r\frac\partial{\partial r}\bigr)^\ell\ssb\cos\lambda\ssf r
=\text{O}\ssf(\lambda^{2\ssf\ell\ssf})$
\,by Taylor's formula,
the expression \eqref{expression1} is
\begin{equation*}\begin{cases}
\,\text{O}\ssf(1)
&\text{if \,}1\!\le\!\ell\!<\!\frac m2\,,\\
\,\text{O}\ssf(\ssf\log\frac1r\ssf)
&\text{if \;}\ell\!=\!\frac m2\,,\\
\,\text{O}\ssf(\ssf r^{\ssf m-2\ssf\ell\ssf})
&\text{if \,}\frac m2\!<\!\ell\!\le\!m\,,\\
\end{cases}\end{equation*}
hence \,$\text{O}\ssf(\ssf r^{-m}\ssf)$ \,in all cases.
Consider next
\begin{equation}\label{expression2}\textstyle
\frac{e^{\ssf\sigma^2}}{\Gamma(-\ssf i\Im\sigma)}\,
{\displaystyle\int_{\,\frac6r}^{+\infty}}\hskip-1mm
d\lambda\,\lambda^{-\tau}\,
(\lambda^2\!+\!\tilde\rho^{\ssf2})^{\frac{\tau-\sigma}2}\,
r^{\ssf k-2\ssf\ell\ssf}\bigl(\textstyle\frac\partial{\partial r}\bigr)^k
e^{\,i\ssf(t\pm r)\ssf\lambda}\,.
\end{equation}
Since
\,$\bigl(\textstyle\frac\partial{\partial r}\bigr)^k
e^{\,i\ssf(t\pm r)\ssf\lambda}\ssb
=\ssb(\pm\ssf i\lambda)^k\ssf e^{\,i\ssf(t\pm r)\ssf\lambda}$
\,and
\begin{equation*}
\lambda^{-\tau}\,
(\lambda^2\!+\ssb\tilde\rho^{\ssf2})^{\frac{\tau-\sigma}2}\,
(\pm\ssf i\lambda)^k\,e^{\,i\ssf(t\pm r)\ssf\lambda}
=\text{O}\ssf(\lambda^{k-m-1})\,,
\end{equation*}
the expression \eqref{expression2} is easily seen to be
\ssf$\text{O}\ssf(\ssf r^{\ssf m-2\ssf\ell\ssf})$
\ssf as long as \ssf$k\!<\!m$\ssf. For the remaining case, where \ssf$k\!=\!\ell\!=\!m$\ssf, let us expand
\begin{equation*}\textstyle
\lambda^{-\tau}\,
(\lambda^2\!+\ssb\tilde\rho^{\ssf2})^{\frac{\tau-\sigma}2}
\lambda^m
=\ssf\lambda^{-1-\ssf i\Im\sigma}\,
\bigl(\ssf1\!+\ssb\frac{\tilde\rho^{\ssf2}}{\lambda^2}
\bigr)^{\ssb\frac{\tau-\sigma}2}\ssb
=\ssf\lambda^{-1-\ssf i\Im\sigma}
+\ssf\text{O}\ssf\bigl(\,|\sigma|\ssf\lambda^{-3}\,\bigr)
\end{equation*}
and split
\begin{equation*}
\int_{\ssf\frac6r}^{+\infty}
=\,\int_{\ssf\frac6r}^{\frac6r+\frac1{r\pm t}}
+\,\int_{\frac6r+\frac1{r\pm t}}^{+\infty}
\end{equation*}
in \eqref{expression2}.
On the one hand, the resulting integrals
\begin{equation}\label{IntegralI}\textstyle
I_\pm=\ssf\frac{e^{\ssf\sigma^2}}{\Gamma(-\ssf i\Im\sigma)}\,
{\displaystyle\int_{\ssf\frac6r}^{\frac6r+\frac1{r\pm t}}}\ssb d\lambda\;
\lambda^{-1-\ssf i\Im\sigma}\,e^{\,i\ssf(t\pm r)\ssf\lambda}
\end{equation}
\vspace{-5mm}

\noindent
and
\vspace{-2mm}
\begin{equation}\label{IntegralII}\textstyle
I\!I_\pm=\ssf\frac{e^{\ssf\sigma^2}}{\Gamma(-\ssf i\Im\sigma)}\,
{\displaystyle\int_{\frac6r+\frac1{r\pm t}}^{+\infty}}\ssb d\lambda\;
\lambda^{-1-\ssf i\Im\sigma}\,e^{\,i\ssf(t\pm r)\ssf\lambda}
\end{equation}
are uniformly bounded. This is proved by integrations by parts\,:
\vspace{-1mm}
\begin{equation*}\begin{aligned}
\textstyle
I_\pm&\textstyle
=\ssf\frac{e^{\ssf\sigma^2}}{\Gamma(1-\ssf i\Im\sigma)}\ssf
\overbrace{
\lambda^{-\ssf i\Im\sigma}\,e^{\,i\ssf(t\pm r)\ssf\lambda}\,\Big|
_{\ssf\lambda\ssf=\ssf\frac6r}
^{\ssf\lambda\ssf=\ssf\frac6r+\frac1{r\pm t}}
}^{\text{O}\ssf(\ssf1\ssf)}\\
&\textstyle
\ssf\mp\,i\,\frac{e^{\ssf\sigma^2}}{\Gamma(1-\ssf i\Im\sigma)}\,
(\ssf r\!\pm\!t\ssf)\ssf\underbrace{
{\displaystyle\int_{\frac1r}^{\frac1r+\frac1{r\pm t}}}\!d\lambda\;
\lambda^{-\ssf i\Im\sigma}\,e^{\,i\ssf(t\pm r)\ssf\lambda}
}_{\text{O}\ssf\left(\frac1{r\pm t}\right)}\,
=\,\text{O}(1)\ssf,
\end{aligned}\end{equation*}
\vspace{-4mm}

\noindent
while
\begin{equation*}\begin{aligned}
\textstyle
I\!I_\pm&\textstyle
=\ssf\mp\,i\,\frac{e^{\ssf\sigma^2}}{\Gamma(-\ssf i\Im\sigma)}\,
\frac1{r\pm t}\ssf\overbrace{
\lambda^{-1-\ssf i\Im\sigma}\ssf e^{\,i\ssf(t\pm r)\ssf\lambda}\,\Big|
_{\ssf\lambda\ssf=\ssf\frac6r+\frac1{r\pm t}}
^{\ssf\lambda\ssf=\ssf+\infty}
}^{\text{O}\ssf(\ssf r\ssf\pm\ssf t\ssf)}\\
&\textstyle
\ssf\mp\,i\,\frac{e^{\ssf\sigma^2}(1+\ssf i\Im\sigma)}{\Gamma(-\ssf i\Im\sigma)}\,
\frac1{r\pm t}\ssf\underbrace{
{\displaystyle\int_{\frac6r+\frac1{r\pm t}}^{+\infty}}\!d\lambda\;
\lambda^{-2-\ssf i\Im\sigma}\ssf
e^{\,i\ssf(t\pm r)\ssf\lambda}
}_{\text{O}\ssf(\ssf r\ssf\pm\ssf t\ssf)}
=\,\text{O}(1)\ssf.
\end{aligned}\end{equation*}
\vspace{-2mm}

\noindent
Hence the contributions of \eqref{IntegralI} and \eqref{IntegralII}
to \eqref{expression2} are \,$\text{O}\ssf(\ssf r^{-m}\ssf)\ssf$.
On the other hand,
the remainder's contribution to \eqref{expression2} is obviously
\ssf$\text{O}\ssf(\ssf r^{\ssf2-m}\ssf)\ssf$.
As a conclusion,
\begin{equation*}\textstyle
|\,\widetilde{w}_{\,t}^{\ssf\infty}(r)\ssf|\,
\lesssim\,r^{-m}\,\lesssim\,t^{-\frac{n-1}2}
\end{equation*}
when \ssf$n\!=\!2\ssf m\!+\!1$ \ssf is odd.
\smallskip

\noindent$\circ$
\textit{Subcase 3.b}\,:
\,Assume that \ssf$n\!=\!2\ssf m$ \ssf is even $\ge\ssb4$\ssf.
\smallskip

\noindent
Then, up to a multiplicative constant,
\begin{equation}\label{IAT}\textstyle
\widetilde{w}_{\,t}^{\ssf\infty}(r)\ssf
=\ssf\frac{e^{\ssf\sigma^2}}{\Gamma(-\ssf i\Im\sigma)}\,
{\displaystyle\int_{\,r}^{+\infty}}\hskip-1mm ds\,
\frac{\sinh s}{\sqrt{\cosh s-\cosh r}}\,
\bigl(\frac1{\sinh s}\ssf\frac\partial{\partial s}\bigr)^{\ssb m}\,
\widetilde{g}_{\,t}^{\ssf\infty}(s)\,.
\end{equation}
Let us split
\begin{equation}\label{Integral2}
\int_{\ssf r}^{+\infty}\hspace{-1mm}
=\;\int_{\ssf r}^{\,6}\,+\;\int_{\ssf6}^{+\infty}\,.
\end{equation}
The following estimate is obtained
by resuming the proof of Theorem \ref{Estimatewtildetinfty}.i.b
in the odd--dimensional case\,:
\begin{equation*}\textstyle
\bigl|\ssf\bigl(\ssf\frac1{\sinh s}\ssf\frac\partial{\partial s}\bigr)^{\ssb m}\,
\widetilde{g}_{\,t}^{\ssf\infty}(s)\ssf\bigr|
\lesssim\ssf e^{-m\ssf s}
\quad\forall\;s\!\ge\!6\,.
\end{equation*}
Since
\begin{equation*}\textstyle
{\displaystyle\int_{\ssf6}^{+\infty}}\hspace{-1mm}ds\,
\frac{\sinh s}{\sqrt{\cosh s-\cosh r}}\,e^{-m\ssf s}\,
\lesssim{\displaystyle\int_{\,0}^{+\infty}}\hspace{-1.5mm}
\frac{du\vphantom{\big|}}{\sqrt{\ssf\sinh u\ssf}}\,
<+\infty\,,
\end{equation*}
the contribution to \eqref{IAT}
of the second integral in \eqref{Integral2}
is uniformly bounded.
Thus we are left with the contribution of the first integral,
which is a purely local estimate.
\medskip

\noindent
\textbf{Lemma C.1}\label{LemmaC1}\textit{
Let \ssf$m$ be an integer $\ge\!2$
and let \ssf$\lambda\!\ge\!1$\ssf, \ssf$r\!\le\!3$\ssf.
\begin{itemize}
\item[(i)]
Assume that \ssf$\lambda\ssf r\!\le\!6$\ssf.
Then
\begin{equation*}\textstyle
\theta(\lambda,r)\ssf=
{\displaystyle\int_{\,r}^{\,6}}\ssb ds\;
\frac{\sinh s}{\sqrt{\cosh s\,-\,\cosh r}}\;
\bigl(\frac1{\sinh s}\,\frac\partial{\partial s}\bigr)^m
\cos\lambda\ssf s
\end{equation*}
is \,$\mathrm{O}\ssf(\ssf\lambda^{2\ssf m-1-\varepsilon}\,
r^{-\varepsilon}\ssf)$\ssf,
\ssf for every \,$\varepsilon\!>\!0$\ssf.
\item[(ii)]
Assume that $\lambda\ssf r\!\ge\!6$\ssf.
Then
\begin{equation*}\textstyle
\theta^{\ssf\pm}(\lambda,r)\ssf=
{\displaystyle\int_{\,r}^{\,6}}\ssb ds\;
\frac{\sinh s}{\sqrt{\cosh s\,-\,\cosh r}}\;
\bigl(\frac1{\sinh s}\,\frac\partial{\partial s}\bigr)^m\,
e^{\ssf\pm\ssf i\ssf\lambda\ssf s}
\end{equation*}
has the following behavior\,:
\begin{equation*}\textstyle
\theta^{\ssf\pm}(\lambda,r)\ssf=\ssf
c_{\ssf\pm}\,\lambda^{m-\frac12}\,
(\ssf\sinh r)^{\frac12-m}\,
e^{\ssf\pm\ssf i\ssf\lambda\ssf r}
+\mathrm{O}\ssf(\ssf\lambda^{m-1}\,r^{-m}\ssf)
\end{equation*}
\end{itemize}
where \,$c_{\ssf\pm}$ is a nonzero complex constant.
}

\begin{proof}
We first prove (i). Recall that
\begin{equation*}\textstyle
\bigl(\frac1{\sinh s}\ssf\frac\partial{\partial s}\bigr)^m(\cos\lambda\ssf s)\,
=\,\begin{cases}
\,\text{O}\ssf(\lambda^{2\ssf m})
&\text{if \,}\lambda\ssf s\!\le\!6\ssf,\\
\,\text{O}\ssf(\lambda^m\ssf s^{-m})
&\text{if \,}\lambda\ssf s\!\ge\!6\ssf,\\
\end{cases}\end{equation*}
hence
\,$\bigl(\frac1{\sinh s}\ssf\frac\partial{\partial s}\bigr)^m(\cos\lambda\ssf s)
=\text{O}\ssf(\lambda^{2\ssf m-1-\varepsilon}\ssf s^{-1-\varepsilon})$
\,in both cases.
By combining this estimate with
\begin{equation*}
\sinh s\asymp s\,,
\qquad {\rm{and~}} \qquad
\cosh s-\cosh r\asymp s^2\ssb-r^2\ssf,
\end{equation*}
and by performing an elementary change of variables,
we reach our conclusion\,:
\begin{equation*}\textstyle
|\ssf\theta(\lambda,r)|\,
\lesssim\,\lambda^{2\ssf m-1-\varepsilon}\,
{\displaystyle\int_{\,r}^{\,6}}\ssb ds\;
s^{-\varepsilon}\,(s^2\!-\ssb r^2)^{-\frac12}\ssf
\le\,\lambda^{2\ssf m-1-\varepsilon}\,r^{-\varepsilon}
\underbrace{{\displaystyle\int_{\,1}^{+\infty}}\!ds\;
s^{-\varepsilon}\,(s^2\!-\!1)^{-\frac12}}_{<+\infty}.
\end{equation*}
We next prove (ii). Recall that
\begin{equation*}\textstyle
\bigl(\frac1{\sinh s}\ssf\frac\partial{\partial s}\bigr)^m\ssf(
e^{\ssf\pm\ssf i\ssf\lambda\ssf s})
=\bigl(\frac{\pm\,i\ssf\lambda}{\sinh s}\bigr)^m\ssf
e^{\ssf\pm\ssf i\ssf\lambda\ssf s}
+\text{O}\ssf(\lambda^{m-1}s^{-m-1})
\end{equation*}
The remainder's contribution to \ssf$\theta^{\ssf\pm}(\lambda,r)$
\ssf is estimated as above\,:
\begin{equation*}\textstyle
{\displaystyle\int_{\,r}^{\,6}}\ssb ds\;
\frac{\sinh s}{\sqrt{\cosh s\,-\,\cosh r}}\;
\lambda^{m-1}\,s^{-m-1}\,
\lesssim\,\lambda^{m-1}\,
{\displaystyle\int_{\,r}^{\,6}}\ssb ds\;
s^{-m}\,(s^2\!-\ssb r^2)^{-\frac12}\,
\lesssim\,\lambda^{m-1}\,r^{-m}\ssf.
\end{equation*}
In order to handle the contribution of
\ssf$\bigl(\frac{\pm\,i\ssf\lambda}{\sinh s}\bigr)^m
\ssf e^{\ssf\pm\ssf i\ssf\lambda\ssf s}$
\ssf to \ssf$\theta^{\ssf\pm}(\lambda,r)$\ssf,
\ssf let us perform the change of variables
\,$s\!=\!r(1\!+\!u)$\ssf,
\ssf so that
\begin{equation*}
\int_{\,r}^{\,6}ds\;=\,r\int_{\,0}^{\frac6r-1}\hspace{-1mm}du\;,
\end{equation*}
and let us expand
\begin{equation*}\begin{aligned}
&\textstyle
(\overbrace{\cosh s\ssb-\ssb\cosh r
\vphantom{\big|}
}^{2\,\sinh\frac{s-r}2\,\sinh\frac{s+r}2})^{-\frac12}\,
(\ssf\sinh s\ssf)^{1-m}\,
e^{\ssf\pm\ssf i\ssf\lambda\ssf s}\\
&\textstyle=r^{-m}\,e^{\ssf\pm\ssf i\ssf\lambda\ssf r}
\underbrace{\textstyle
(\ssf\frac{\sinh\frac{r\ssf u}2}{\frac{r\ssf u}2}\bigr)^{-\frac12}\ssf
\bigl(\ssf\frac{\sinh r(1+\frac u2)}{r(1+\frac u2)}\bigr)^{-\frac12}\ssf
\bigl(\ssf\frac{\sinh r(1+u)\vphantom{\frac12}}
{r(1+u)\vphantom{\frac12}}\bigr)^{1-m}
}_{A(r,u)}
\underbrace{\textstyle
u^{-\frac12}\ssf(1\!+\!\frac u2)^{-\frac12}\ssf(1\!+\!u\ssf)^{1-m}
\vphantom{\frac{\frac12}{\frac12}}
}_{B(u)}
e^{\ssf\pm\ssf i\ssf\lambda\ssf r\ssf u}\,.
\end{aligned}\end{equation*}
Notice that the expressions \,$A(r,u)$ \ssf and \ssf$B(u)$
\ssf can be expanded as follows\,:
\begin{equation}\textstyle
A(r,u)\ssf=\ssf\bigl(\frac{\sinh r}r)^{\frac12-m}\ssf
+\ssf\underbrace{\sum\nolimits_{j=1}^{+\infty}\,
A_j(r)\,(\ssf r\ssf u\ssf)^{\ssf j}}_{\widetilde A(r,u)}\ssf,
\end{equation}
\vspace{-4mm}
\begin{equation}\textstyle
B(u)\ssf=\,u^{-\frac12}\,
+\ssf\underbrace{\sum\nolimits_{j=1}^{+\infty}\,
B_{\ssf j}^{\,0}\,u^{\,j-\frac12}}_{\widetilde B(u)}
\quad\text{for \ssf$u$ \ssf small\,,}
\end{equation}
\vspace{-4mm}
\begin{equation}\textstyle
B(u)\ssf=\ssf\sqrt{\ssf2\ssf}\,u^{-m}\,
+\,{\displaystyle\sum\nolimits_{j=1}^{+\infty}}\,
B_{\,j}^{\ssf\infty}\,u^{-j-m}
\quad\text{for \ssf$u$ \ssf large\,.}
\end{equation}
Using these behaviors and integrating by parts,
we can estimate
\begin{equation*}\begin{aligned}
{\displaystyle\int_{\ssf0}^{\frac6r-1}}\hspace{-1mm}du\;
e^{\ssf\pm\ssf i\ssf\lambda\ssf r\ssf u}\,\widetilde A(r,u)\,B(u)\ssf
&\textstyle=\ssf
\frac1{\pm\ssf i\ssf\lambda\ssf r}\;
e^{\ssf\pm\ssf i\ssf\lambda\ssf r\ssf u}\,
\widetilde A(r,u)\,B(u)\,
\Big|_{\ssf u=0}^{\ssf u=\frac6r-1}\\
&\textstyle-\frac1{\pm\ssf i\ssf\lambda\ssf r}\,
{\displaystyle\int_{\ssf0}^{\frac6r-1}}\hspace{-1mm}du\;
e^{\ssf\pm\ssf i\ssf\lambda\ssf r\ssf u}\,
\frac\partial{\partial u}\ssf
\bigl\{\widetilde A(r,u)\,B(u)\bigr\}
\end{aligned}\end{equation*}
by \,$\text{O}\ssf(\frac1{\lambda\ssf r})$. The integrals
\begin{equation*}
{\displaystyle\int_{\,0}^{\ssf1}}du\;
e^{\ssf\pm\ssf i\ssf\lambda\ssf r\ssf u}\,\widetilde B(u)
\quad\text{and}\quad
{\displaystyle\int_{\,1}^{\frac6r-1}}\!du\;
e^{\ssf\pm\ssf i\ssf\lambda\ssf r\ssf u}\,B(u)
\end{equation*}
are estimated similarly.
In summary, we showed that
\begin{equation*}
\theta^{\ssf\pm}(\lambda,r)
=(\pm\ssf i\ssf)^m\,\lambda^m\,(\sinh r)^{\frac12-m}\,
r^{\ssf\frac12}\,e^{\ssf\pm\ssf i\ssf\lambda\ssf r}
{\displaystyle\int_{\,0}^{\ssf 1}}\ssb du\;
e^{\ssf\pm\ssf i\ssf\lambda\ssf r\ssf u}\,u^{-\frac12}
+\text{O}\ssf(\lambda^{m-1}\ssf r^{-m})
\end{equation*}
and we conclude by using the behavior of the elementary integral
\begin{equation*}
{\displaystyle\int_{\,0}^{\ssf 1}}du\;
e^{\ssf\pm\ssf i\ssf\lambda\ssf r\ssf u}\,u^{-\frac12}\,
=\,\lambda^{-\frac12}\,r^{-\frac12}
\underbrace{
\int_{\,0}^{+\infty}\hspace{-1mm}du\;
e^{\ssf\pm\ssf i\ssf u}\,u^{-\frac12}
}_{\text{constant}}\,
+\;\text{O}\ssf(\lambda^{-1}\ssf r^{-1})\,.
\end{equation*}
\end{proof}

\noindent
From now on, the discussion of Subcase 3.b is similar to Subcase 3.a.
On the one hand, we deduce from Lemma C.1.i that
\begin{equation*}\textstyle
{\displaystyle\int_{\,1}^{\ssf\frac6r}}
d\lambda\;\chi_\infty(\lambda)\,\lambda^{-\tau}\,
(\lambda^2\!+\!\tilde\rho^{\ssf2})^{\frac{\tau-\sigma}2}\,
e^{\,i\ssf t\ssf\lambda}\;\theta(\lambda,r)\,
=\,\text{O}\ssf(\ssf r^{\frac12-m}\ssf)\,.
\end{equation*}
On the other hand, by expanding
\begin{equation*}\textstyle
\lambda^{-\tau}\,
(\ssf\lambda^2\!+\ssb\tilde\rho^{\ssf2}\ssf)^{\frac{\tau-\sigma}2}
=\ssf\lambda^{-\sigma}\,
\bigl(\,1\ssb+\ssb\frac{\tilde\rho^{\ssf2}}{\lambda^2}\ssf\bigr)^{\frac{\tau-\sigma}2}
=\ssf\lambda^{-m-\frac12-\ssf i\Im\sigma}
+\ssf\text{O}\ssf\bigl(\,|\sigma|\,\lambda^{-m-\frac52}\ssf\bigr)
\qquad\forall\;\lambda\!\ge\!2
\end{equation*}
and \,$\theta^{\ssf\pm}(\lambda,r)$ \ssf according to Lemma C.1.ii\ssf,
we have
\begin{equation*}\begin{aligned}
&\textstyle
\frac{e^{\ssf\sigma^2}}{\Gamma(-\ssf i\Im\sigma)}\,
{\displaystyle\int_{\ssf\frac6r}^{+\infty}}\hspace{-1mm}
d\lambda\;\chi_\infty(\lambda)\,\lambda^{-\tau}\,
(\lambda^2\!+\!\tilde\rho^{\ssf2})^{\frac{\tau-\sigma}2}\,
e^{\,i\ssf t\ssf\lambda}\;\theta^{\ssf\pm}(\lambda,r)\\
&\textstyle
=\,c_{\ssf\pm}\,(\ssf I_\pm\!+\ssb I\!I_\pm\ssf)\,
(\ssf\sinh r)^{\frac12-m}\ssf
+\,\text{O}\bigl(\ssf r^{\frac12-m}\ssf\bigr)\,,
\end{aligned}\end{equation*}
where \,$I_\pm$ and \,$I\!I_\pm$ denote
the integrals \eqref{IntegralI} and \eqref{IntegralII},
which are uniformly bounded and whose sum is equal to
\begin{equation*}\textstyle
\frac{e^{\ssf\sigma^2}}{\Gamma(-\ssf i\Im\sigma)}\,
{\displaystyle\int_{\ssf\frac6r}^{+\infty}}\hspace{-1mm}d\lambda\;
\lambda^{-1-i\Im\sigma}\,e^{\,i\ssf(t\ssf\pm\ssf r)\ssf\lambda}\,.
\end{equation*}
As a conclusion, we obtain again
\begin{equation*}\textstyle
|\,\widetilde{w}_{\,t}^{\ssf\infty}(r)\ssf|\,
\lesssim\,r^{\frac12-m}\,
\lesssim\,t^{-\frac{n-1}2}\,.
\end{equation*}

\medskip

\noindent
\textbf{Remark C.2.}\textit{
The analysis above still holds in dimension \ssf$n\!=\!2$\ssf,
except for the first estimate in Lemma C.1,
which is replaced by
\begin{equation*}\textstyle
\theta(\lambda,r)=\mathrm{O}\ssf(\ssf\lambda\log\frac2r\ssf)\ssf.
\end{equation*}
As a result,
\begin{equation*}\textstyle
|\ssf\widetilde{w}_{\,t}^\infty(r)|\ssf
\lesssim\ssf|t|^{-\frac12}\ssf(\ssf1\!-\ssb\log|t|\ssf)\,.
\end{equation*}
}

\noindent
\textbf{Remark C.3.}\textit{
In order to estimate the wave kernel for small time,
we might have used the \textit{Hadamard parametrix\/} \cite[\S\;17.4]{Ho3}
instead of spherical analysis.
}


\begin{thebibliography}{9999}

\bibitem{AP}
J.-Ph. Anker, V. Pierfelice, 
\textit{Nonlinear Schr\"odinger equation on real hyperbolic spaces\/},
Ann. Inst. H. Poincar\'e (C) {\it Non Linear Analysis\/} 26 (2009), 1853--1869

\bibitem{AP4}
J.-Ph. Anker, V. Pierfelice, 
\textit{Wave and Klein--Gordon equation on hyperbolic spaces\/},
preprint [arXiv:1104.0177]

\bibitem{BL}
J. Bergh, J. L\"ofstr\"om,
\textit{Interpolation spaces\/} (\textit{an introduction\/}),
Springer--Verlag (1976)

\bibitem{CK}
M. Christ, A. Kiselev,
\textit{Maximal functions associated to filtrations\/},
J. Funct. Anal. 179 (2001), 409-425

\bibitem{Co1}
M.G. Cowling, 
\textit{The Kunze--Stein phenomenon\/},
Ann. Math 107 (1978), 209--234

\bibitem{Co2}
M.G. Cowling, 
\textit{Herz's``principe de majoration'' and the Kunze-Stein phenomenon\/},
in \textit{Harmonic analysis and number theory\/} (Montreal, 1996),
CMS Conf. Proc. 21, Amer. Math. Soc. (1997), 73--88

\bibitem{CGM1}
M. Cowling, S. Giulini, S. Meda, 
\textit{$L^p-L^q$ estimates for functions of the Laplace-Beltrami operator
on noncompact symmetric spaces I\/},
Duke Math. J. 72 (1993), 109--150

\bibitem{DGK}
P. D'Ancona, V. Georgiev, H. Kubo,
\textit{Weighted decay estimates for the wave equation\/},
J. Diff. Eq. 177 (2001), 146--208

\bibitem{F1}
J. Fontaine,
\textit{Une \'equation semi--lin\'eaire des ondes sur \ssf$\mathbb{H}^3$},
C. R. Acad. Sci. Paris S\'er. I \textit{Math\'ematiques\/} 319 (1994), 935--948

\bibitem{F2}
J. Fontaine,
\textit{A semilinear wave equation on hyperbolic spaces\/},
Comm. Partial Diff. Eq. 22 (1997), 633--659

\bibitem{G}
V. Georgiev,
\textit{Semilinear hyperbolic equations\/},
Mem. Math. Soc. Japan 7 (2000)

\bibitem{GLS}
V. Georgiev, H. Lindblad, C. Sogge,
\textit{Weighted Strichartz estimates and global existence
for semilinear wave equations\/},
Amer. J. Math. 119 (1997), 1291--1319

\bibitem{GV}
J. Ginibre, G. Velo, 
\textit{Generalized Strichartz inequalities for the wave equation\/},
J. Funct. Anal. 133 (1995), 50--68

\bibitem{GC}
I.M. Guelfand, G.E. Chilov, 
\textit{Les distributions\/}, tome I,
Dunod (1962)

\bibitem{Ha}
A. Hassani, 
\textit{Wave equation on symmetric spaces\/},
J. Math. Phys. 52 (2011), 043514

\bibitem{Hel1}
S. Helgason, 
\textit{Differential geometry, Lie groups, and symmetric spaces\/},
Academic Press (1978), Amer. Math. Soc. (2001)

\bibitem{Hel2}
S. Helgason, 
\textit{Groups and geometric analysis}
(\textit{integral geometry, invariant differential operators,
and spherical functions\/}),
Academic Press (1984), Amer. Math. Soc. (2002) 

\bibitem{Hel3}
S. Helgason,
\textit{Geometric analysis on symmetric spaces\/},
Amer. Math. Soc. (1994) 

\bibitem{Her}
C.S. Herz,
\textit{Sur le ph\'enom\`ene de Kunze--Stein\/},
C. R. Acad. Sci. Paris S\'er. A 271 (1970 ), 491--493
 
\bibitem{Ho1}
L.V. H\"ormander,
{\it The analysis of linear partial differential operators I\/}
({\it distribution theory and Fourier analysis\/}),
Springer--Verlag (1983, 1990, 2003)

\bibitem{Ho3}
L.V. H\"ormander,
{\it The analysis of linear partial differential operators III\/}
({\it pseudo--differential operators\/}),
Springer--Verlag (1985, 1994, 2007)

\bibitem{I1}
A.D. Ionescu,
\textit{Fourier integral operators on noncompact symmetric spaces of real rank one\/},
J. Funct. Anal.  174 (2000), 274--300

\bibitem{I2}
A.D. Ionescu,
\textit{An endpoint estimate for the Kunze--Stein phenomenon
and related maximal operators\/},
Ann. of Math. (2) 152 (2000), 259--275

\bibitem{Ka}
L. Kapitanski,
\textit{Weak and yet weaker solutions of semilinear wave equations\/},
Comm. Partial Diff. Eq. 19 (1994), 1629--1676

\bibitem{KT}
M. Keel, T. Tao,
\textit{Endpoint Strichartz estimates\/}
Amer. J. Math. 120 (1998), 955--980

\bibitem{Ko}
T.H. Koornwinder,
\textit{Jacobi functions and analysis on noncompact semisimple Lie groups\/},
in \textit{Special functions\/} (\textit{group theoretical aspects and applications\/}),
R.A. Askey \& al. (eds.), Reidel (1984), 1--85

\bibitem{LS}
H. Lindblad, C. Sogge,
\textit{On existence and scattering with minimal regularity
for semilinear wave equations\/}, 
J. Funct. Anal.  130 (1995), 357--426

\bibitem{MT} J. Metcalfe, M.E. Taylor, 
\textit{Nonlinear waves on 3D hyperbolic space\/},
to appear in Trans. Amer. Math. Soc.

\bibitem{P}
V. Pierfelice,
\textit{Weighted Strichartz estimates
for the Schr\"odinger and wave equations
on Damek--Ricci spaces\/},
Math. Z.  260 (2008), 377--392

\bibitem{Ta}
D. Tataru,
\textit{Strichartz estimates in the hyperbolic space
and global existence for the semilinear wave equation\/},
Trans. Amer. Math. Soc. 353 (2001), 795--807

\bibitem{Tr2}
H. Triebel,
\textit{Theory of function spaces II\/},
Monographs Math. 84, Birkh\"auser (1992)

\end{thebibliography}
\end{document}